\pdfoutput=1
\RequirePackage{ifpdf}
\ifpdf % We~are running pdfTeX in pdf mode
\documentclass[pdftex]{sigma}
\else
\documentclass{sigma}
\fi

\numberwithin{equation}{section}

\newtheorem{Theorem}{Theorem}[section]
\newtheorem{Corollary}[Theorem]{Corollary}
\newtheorem{Lemma}[Theorem]{Lemma}
\newtheorem{Proposition}[Theorem]{Proposition}
 { \theoremstyle{definition}

\newtheorem{Example}[Theorem]{Example}
\newtheorem{Remark}[Theorem]{Remark} }

\usepackage{subcaption}
\usepackage{bbm}
\usepackage{bm}
\usepackage{mathtools}

\newcommand{\bfR}{\mathbf{R}}

\newcommand{\bfH}{{\mathbf H}}
\newcommand{\bfK}{{\mathbf K}}

\newcommand{\bfkappa}{{\bm \varkappa}}

\newcommand{\erfc}{\operatorname{erfc}}
\newcommand{\erf}{\operatorname{erf}}

\newcommand{\I}{{\mathrm{I}}}
\newcommand{\II}{{\mathrm{II}}}
\newcommand{\III}{{\mathrm{III}}}
\newcommand{\R}{{\mathbb R}}

\newcommand{\C}{{\mathbb C}}
\newcommand{\E}{{\mathbf E}}

\newcommand{\eps}{{\varepsilon}}

\newcommand{\re}{\operatorname{Re}}
\newcommand{\im}{\operatorname{Im}}

\newcommand{\Ham}{\mathbf{H}}

\newcommand{\Prob}{{\mathbf{P}}}

\renewcommand{\d}{{\partial}}

\newcommand{\Pf}{{\operatorname{Pf}}}
\newcommand{\sgn}{\operatorname{sgn}}

\def\bp{{\bar\partial}}
\def\wh{\widehat}

\newcommand{\RN}[1]{%
	\textup{\uppercase\expandafter{\romannumeral#1}}%
}

\usepackage[skip=5pt,labelfont=bf,labelsep=period]{caption}

\begin{document}
%\allowdisplaybreaks

\newcommand{\arXivNumber}{2106.09345}

\renewcommand{\PaperNumber}{007}

\FirstPageHeading

\ShortArticleName{Scaling Limits of Planar Symplectic Ensembles}

\ArticleName{Scaling Limits of Planar Symplectic Ensembles}

\Author{Gernot AKEMANN~$^{\rm a}$, Sung-Soo BYUN~$^{\rm b}$ and Nam-Gyu KANG~$^{\rm b}$}

\AuthorNameForHeading{G.~Akemann, S.-S.~Byun and N.-G.~Kang}

\Address{$^{\rm a)}$~Faculty of Physics, Bielefeld University, P.O. Box 100131, 33501 Bielefeld, Germany}
\EmailD{\href{mailto:akemann@physik.uni-bielefeld.de}{akemann@physik.uni-bielefeld.de}}

\Address{$^{\rm b)}$~School of Mathematics, Korea Institute for Advanced Study, Seoul, 02455, Republic of Korea}
\EmailD{\href{mailto:sungsoobyun@kias.re.kr}{sungsoobyun@kias.re.kr}, \href{mailto:namgyu@kias.re.kr}{namgyu@kias.re.kr}}

\ArticleDates{Received June 23, 2021, in final form January 19, 2022; Published online January 25, 2022}

\Abstract{We consider various asymptotic scaling limits $N\to\infty$ for the $2N$ complex eigen\-values of non-Hermitian random matrices in the symmetry class of the symplectic Ginibre ensemble. These are known to be integrable, forming Pfaffian point processes, and we obtain limiting expressions for the corresponding kernel for different potentials. The first part is devoted to the symplectic Ginibre ensemble with the Gaussian potential. We obtain the asymptotic at the edge of the spectrum in the vicinity of the real line. The unifying form of the kernel allows us to make contact with the bulk scaling along the real line and with the edge scaling away from the real line, where we recover the known determinantal process of the complex Ginibre ensemble.	Part two covers ensembles of Mittag-Leffler type with a~singularity at the origin. For potentials $Q(\zeta)=|\zeta|^{2\lambda}-(2c/N)\log|\zeta|$, with $\lambda>0$ and $c>-1$, the limiting kernel obeys a linear differential equation of fractional order $1/\lambda$ at the origin. For integer $m=1/\lambda$ it can be solved in terms of Mittag-Leffler functions. In the last part, we derive Ward's equation for planar symplectic ensembles for a general class of potentials. It serves as a tool to investigate the Gaussian and singular Mittag-Leffler universality class. This allows us to determine the functional form of all possible limiting kernels (if they exist) that are translation invariant, up to their integration domain.}


\Classification{60B20; 33C45; 33E12}

\section{Introduction}

In the pioneering work of Ginibre \cite{ginibre1965statistical}, it was first discovered that the complex eigenvalues of Gaussian random matrices with quaternion elements (also known as the symplectic Ginibre ensemble) behave like equally charged particles with complex conjugation symmetry. They interact via the two-dimensional Coulomb repulsion, subject to a confining
Gaussian potential $Q(\zeta)=|\zeta|^2$.
Below we will consider more general potentials $Q$, where the joint probability distribution $\Prob_N$ of the point process $\boldsymbol{\zeta}=(\zeta_1,\dots,\zeta_N) \in \C^N$ is given by
\begin{gather}\label{Gibbs}
	{\rm d} \Prob _N(\boldsymbol{\zeta}):= \frac{1}{Z_N} {\rm e}^{-\bfH_N( \boldsymbol{\zeta} )} \prod_{j=1}^N {\rm d}A(\zeta_j).
\end{gather}
Here the Hamiltonian $\bfH_N$ is given by
\begin{gather} \label{Ham}
	\bfH_N( \boldsymbol{\zeta} ):=\sum_{j\neq k} \log \frac{1}{|\zeta_j-\zeta_k| \big|\zeta_j-\overline{\zeta}_k\big| }+\sum_{j=1}^{N} \left( \log \frac{1}{\big| \zeta_j- \overline{\zeta}_j \big|^2}+ NQ(\zeta_j) \right),
\end{gather}
the normalisation constant $Z_N$ is called the partition function, which turns $\Prob _N$ into a probability measure, and ${\rm d}A(\zeta):=\tfrac{1}{\pi}{\rm d}^2\zeta$ is the two-dimensional Lebesgue measure divided by~$\pi$.

Compared to the eigenvalues of the complex Ginibre or more general random normal matrix ensembles, which correspond to a genuine two-dimensional Coulomb gas at specific inverse temperature $\beta=2$ without further symmetries, one of the most distinguished features of the symplectic ensemble is the existence of a local repulsion from the real axis, which follows from the term $-\log \big| \zeta_j- \overline{\zeta}_j \big|^2$ in~\eqref{Ham}.
To be more precise, this local repulsion originates from complex conjugate eigenvalue pairs which repel each other when approaching the real axis.
For illustration, see Figure~\ref{Fig. QGinibre}(a) for some random samplings of eigenvalues from the symplectic Ginibre ensemble.

\begin{figure}[h!]
	\begin{subfigure}{0.39\textwidth}
		\begin{center}	
			\includegraphics[width=2in]{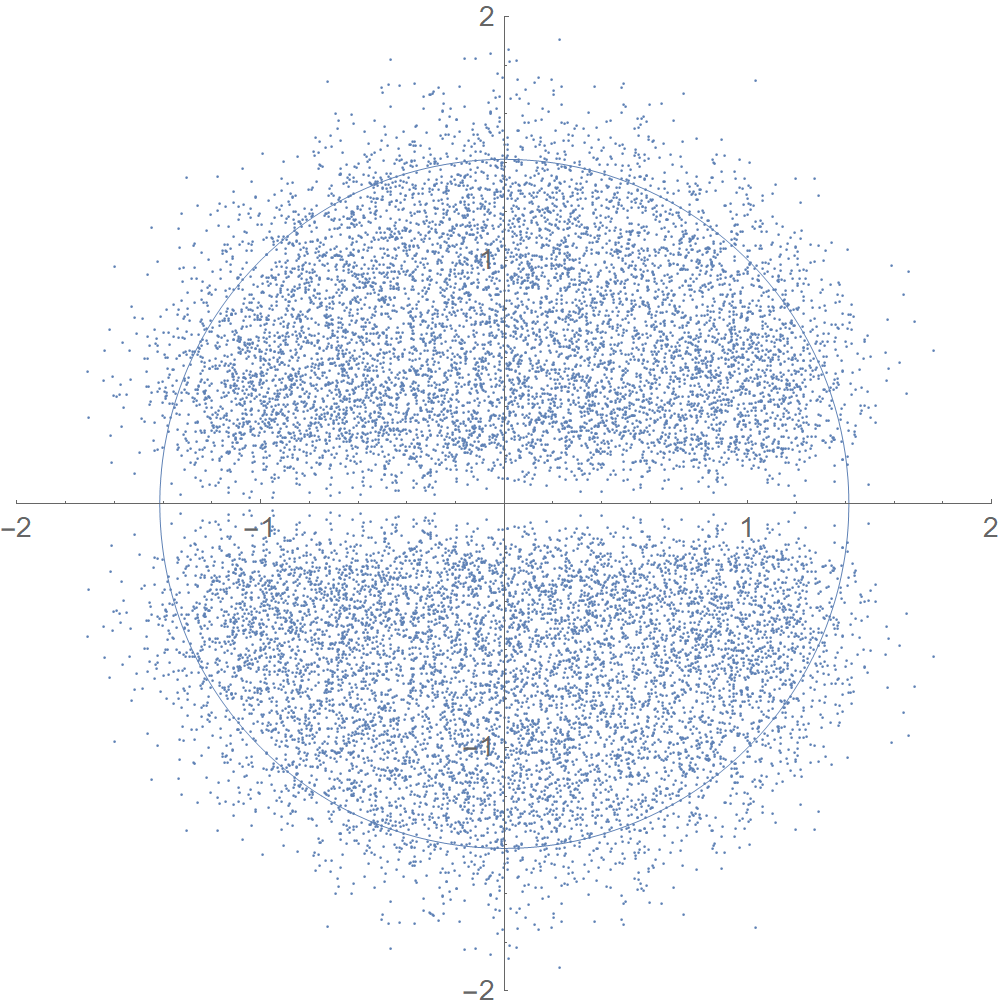}
		\end{center}
		\subcaption{$1024$ samples for $N=8$}
	\end{subfigure}	
	\begin{subfigure}[h]{0.59\textwidth}
		\begin{center}
			\includegraphics[width=2.792in]{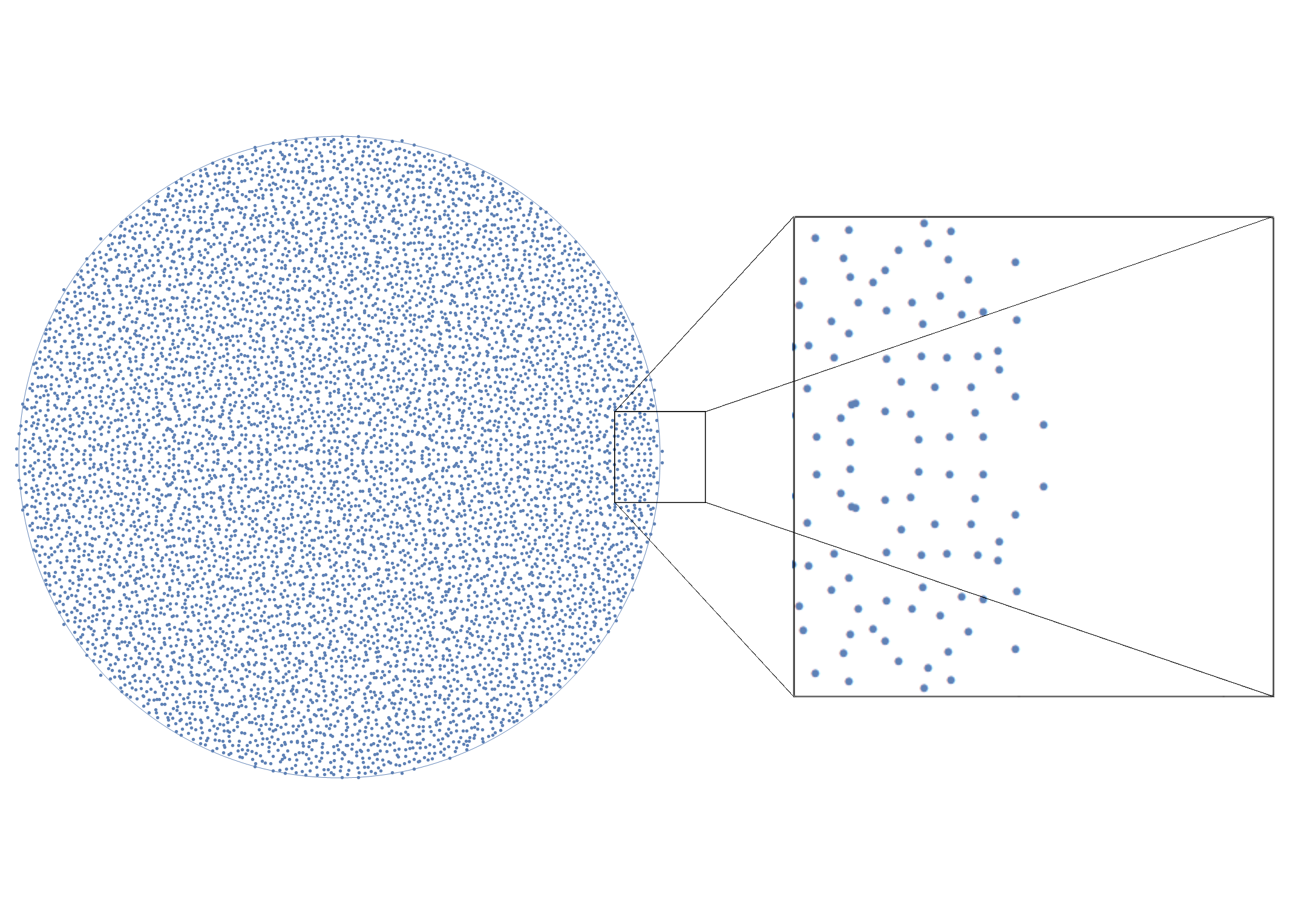}
		\end{center}
		\subcaption{a single sample for $N=8192$}
	\end{subfigure}
	\caption{The plots show $2N$ eigenvalues of random matrices from the symplectic Ginibre ensemble.
		In (a), the local repulsion from the real axis is clearly visible. In (b), the rescaled process at the right endpoint of the spectrum clearly displays the complex conjugation symmetry.} \label{Fig. QGinibre}
\end{figure}

Why are symplectic ensembles interesting, apart from their statistical mechanics interpretation?
Together with its real and complex counterparts, the symplectic Ginibre ensemble represents one of the few examples which is integrable, in the sense that it constitutes a Pfaffian point process where the matrix-valued kernel can be explicitly constructed.
This fact will be recalled in more detail in the next section.
It makes the asymptotic analysis of the kernel in various scaling limits possible, as will be the main topic of this work. The situation is much more difficult for Coulomb gases at general values of $\beta\neq2$, see \cite{Ameur18low,Abanov,Serfaty} and references therein for recent developments.

\looseness=1 Moreover, it has been found rather recently that non-Hermitian random matrices enjoy a~much wider class of universality than their Hermitian counterparts, in the sense that away from the real axis the limiting complex eigenvalue correlation functions of all three Ginibre ensembles agree
\cite{akemann2019universal,MR2530159,Dubach2020,MR1986426}. The same phenomenon has been observed very recently when comparing the spectra of truncated unitary and symplectic random matrices \cite{BL,ZS}, respectively.

While random matrices with complex eigenvalues have many applications in physics in general, we provide some examples
where the symplectic ensemble yields unique predictions and differs from other non-Hermitian symmetry classes, notably at the origin. These include disordered non-Hermitian Hamiltonians with an imaginary magnetic field~\cite{EK},
thermal conduction in superconducting quantum dots~\cite{DBB} in the related circular quaternion ensemble,
or the spectrum of the Dirac operator in quantum chromodynamics with two colours at non-vanishing chemical
potential, for the symplectic ensemble with additional chiral symmetry \cite{MR2180006}.
In the latter the complex eigenvalues of the Dirac operator in the vicinity of the origin are of particular importance because of their role in chiral symmetry breaking. The limiting random matrix predictions have been confirmed from field theory in~\cite{ABi06}.

It is the goal of this article to derive further asymptotic results which are specific for the symplectic ensembles, notably at the edge of the real axis and in the presence of singularities at the origin. The representation for the limiting
edge kernel that we will derive provides a unifying picture that allows us to relate results at different parts of the spectrum. Before we present our main results in the next section, let us summarise what was previously known about symplectic ensembles in various parts of the spectrum. Here we include the elliptic symplectic Ginibre ensemble with the potential $Q(\zeta)=
\tfrac{1}{1-\tau^2} \big( |\zeta|^2-\tau \re \zeta^2 \big)$, where the parameter $\tau\in[0,1)$ controls the degree of non-Hermiticity. Its joint probability distribution \eqref{Gibbs} and kernel at finite-$N$ were derived by Kanzieper~\cite{MR1928853}.

In the asymptotic analysis of the kernel at various points of the spectrum one has to distinguish local or microscopic from global or macroscopic scales.
Since the above-mentioned repulsion from the real line affects only the microscopic scale when the complex conjugate eigenvalues become close, it could be expected that the leading form of the macroscopic eigenvalue density is the same as that of the two-dimensional Coulomb gas in the symmetry class of the complex Ginibre ensemble, with external potential~$Q/2$,
see, e.g.,~\cite{MR3458536}. Indeed, for a general $Q$ satisfying complex conjugation symmetry $Q(\zeta)=Q\big(\bar{\zeta}\big)$ and suitable potential theoretic assumptions, it was shown by Benaych-Georges and Chapon that as $N \to \infty$ the empirical measure of~$\boldsymbol{\zeta}$ converges to Frostman's equilibrium measure associated with $Q/2$, see \cite[Theorem~3.1]{MR2934715}.
In particular, the density of $\boldsymbol{\zeta}$ tends to
\begin{gather} \label{density}
	\tfrac12 \Delta Q(\zeta) \cdot \mathbbm{1}_S(\zeta), \qquad \Delta:=\partial \bar{\partial},
\end{gather}
where $S$ is a certain compact set called the \emph{droplet}. For the Ginibre ensemble this is the well-known circular (or elliptic) law.

In the local scaling limit at the origin at maximal non-Hermiticity ($\tau=0$), the limiting kernel of the symplectic Ginibre was derived in \cite{MR1928853,Mehta}. At weak non-Hermiticity, when $\tau$ scales as $1-\tau=O\big(N^{-1}\big)$, a different limiting kernel was found at the origin by Kanzieper \cite{MR1928853}. It interpolates between the former
at $\tau=0$ and the sine kernel of the Gaussian symplectic ensemble in the limit $\tau\to1$. Both limiting kernels are invariant under translations along the real line, and we will come back to this feature below.

At the edge of the spectrum, it was shown by Rider \cite{MR1986426} and spelled out by Dubach \cite{Dubach2020} that the maximal modulus fluctuations of complex and symplectic Ginibre ensemble agree.
The agreement between the two ensembles was also studied in \cite{akemann2019universal}. It was shown that in the bulk away from the real axis both ensembles yield the same determinantal point processes. In this work we will first focus on the edge on the real line. The local statistics along the real line for the elliptic symplectic Ginibre ensemble is found in \cite{BE}.

\looseness=1 Secondly, we investigate what happens when a non-Gaussian potential develops a singularity by inserting a point charge at the origin. In the generic case of a potential of Mittag-Leffler type we will be able to provide an explicit expression for the limiting origin kernel,
that differs from the Ginibre universality class.
Our findings can be thought of as the counterparts for previous results in random normal matrix ensembles \cite{ameur2018random,MR1643533}.
(See also \cite{MR3668632,MR3849128,MR3670735} for extensive studies on the orthogonal polynomials associated with Mittag-Leffler type potentials.)

As the third issue, we study the universality for kernels that are translation-invariant along the real line. In setting up Ward's equation for the symplectic ensemble -- an identity satisfied by limits of the rescaled one-point functions, under some natural assumptions, we can completely characterise the class of all such possible limiting kernels for general potentials by an integral representation. It is unique (if it exists) up to the integration domain, which is a connected interval symmetric around the origin, see \cite{MR3975882,MR4030288} for analogous works in random normal matrix ensembles.

\section{Main results}\label{sec:main}
Let us now come to the main objects in this work. We denote by $\mathbb{D}(\eta,r)$ the disc with centre $\eta \in \C$ and radius $r$. For a given sequence of points $p_N$, the positive number $r_N=r_N(p_N)$ is called the \textit{micro-scale} if it satisfies
\begin{gather} \label{micro-scale}
	\int_{ \mathbb{D}(p_N,r_N) } \frac{ \Delta Q(\zeta) }{2} \,{\rm d}A(\zeta)= \frac{1}{N}.
\end{gather}
We drop the subscript and write $p \equiv p_N$ if the sequence does not depend on $N$.
By \eqref{density}, the micro-scale $r_N$ corresponds to the mean eigenvalue spacing in radial distance of the ensemble~\eqref{Gibbs} at the point $p\in S$
(cf.~\cite{BE} for a situation where $p$ is outside the droplet).
We define the rescaled process $\boldsymbol{z}=\{ z_j \}_{ j=1 }^N$ as follows: for all $j$,
\begin{gather} \label{rescaling}
	z_j:=
	\begin{cases}
		r_N^{-1} \cdot (\zeta_j-p) &\text{if } p \in \operatorname{int}(S) ,
		\\
		{\rm e}^{-{\rm i}\theta} r_N^{-1} \cdot (\zeta_j-p) &\text{if } p \in \partial S ,
	\end{cases}
\end{gather}
distinguishing the interior and the boundary of the droplet $S$.
Here the angle $\theta\equiv \theta_N \in \R$ is chosen so that ${\rm e}^{{\rm i}\theta}$ is outer normal to $\partial S$ at $p$, see Figure~\ref{Fig. QGinibre}(b)
inset
for an illustration of the rescaled process.
The $k$-point correlation function of
the rescaled process $\boldsymbol{z}$ is defined by
\begin{gather*}
R_{N,k}(z_1,\dots,z_k)
:= \lim_{\eps \downarrow 0} \frac{\mathbb{P}(\exists \text{ at least one particle in } \mathbb{D}(z_j,\eps),\, j=1,\dots,k ) }{\eps^{2k}}.
\end{gather*}
In \eqref{bfRNk def} a more standard definition of the $k$-point correlation function before rescaling denoted by $\bfR_{N,k}(\zeta_1,\dots,\zeta_k)$ is given.
Throughout the paper we distinguish these and all other objects (kernels, Hamiltonian, joint distribution) before rescaling by bold symbols.
For the precise relation on the level of the kernels see~\eqref{kappa_N}. It results into the following relation between the $k$-point correlation functions
\begin{gather*}
R_{N,k}(z_1,\dots, z_k)=r_N^{2k} \bfR_{N,k}(\zeta_1,\dots,\zeta_k).
\end{gather*}
It is well known that before and after rescaling the
set $\boldsymbol{z}$ forms a Pfaffian point process (see \cite{MR1928853,Mehta}), i.e., $R_{N,k}$ is expressed in terms of a certain $2\times 2$ matrix-valued kernel $K_N$ as
\begin{gather} \label{RNk Pf}
	R_{N,k}(z_1,\dots, z_k)=\prod_{j=1}^{k}
(\bar{z}_j-z_j) \Pf \big[ K_{N}(z_j,z_l) \big]_{ j,l=1 }^k ,
\end{gather}
where the kernel $K_N$ is of the form
\begin{gather*}
	K_N(z,w):={\rm e}^{ -\frac{N}{2} (Q( \zeta )+Q( \eta )) } \begin{pmatrix}
		\kappa_N(z,w) & \kappa_N(z,\bar{w}) \\
		\kappa_N(\bar{z},w) & \kappa_N(\bar{z},\bar{w})
	\end{pmatrix}.
\end{gather*}
Here $\Pf$ denotes a Pfaffian of the $2k \times 2k$ skew-symmetric matrix and
\begin{gather} \label{zeta eta}
\zeta=\begin{cases}
	p+r_N z &\text{if }p \in \textup{int}(S),
	\\
	p+{\rm e}^{{\rm i}\theta} r_N z &\text{if }p \in \partial S,
\end{cases} \qquad
 \eta=\begin{cases}
	 p+r_N w &\text{if }p \in \textup{int}(S),
	\\
	 p+{\rm e}^{{\rm i}\theta} r_N w &\text{if }p \in \partial S,
\end{cases}
\end{gather}
where $\theta$ is given as in \eqref{rescaling}.
The arguments \eqref{zeta eta} are given according to the rescaling \eqref{rescaling}.
For the spectral density at $k=1$, let us denote $R_N \equiv R_{N,1}$.

The primary goal of this work is to derive the large-$N$ limit of the kernel $K_N$ for various potentials $Q$.
For the Gaussian potential $Q(\zeta)=|\zeta|^2$, where the associated ensemble~\eqref{Gibbs} corresponds to the symplectic Ginibre ensemble, it follows from the circular law that $S=\mathbb{D}\big(0,\sqrt{2}\big)$.
In \cite{MR1928853}, Kanzieper studied the elliptic potential $Q(\zeta)=\frac{1}{1-\tau^2}\big(|\zeta|^2-\re\zeta^2\big)$ and derived the scaling limit
for the pre-kernel at the origin $p=0$ in the almost-Hermitian regime when $1-\tau=O\big(\frac{1}{N}\big)$.
Also at $p=0$ and for maximally non-Hermiticity at $\tau=0$,
he showed that the associated $\infty$-point process $\{z_j\}_{j=1}^\infty$
of the symplectic Ginibre ensemble has the correlation kernel
\begin{gather} \label{K bulk}
	K_{{\rm bulk}}^\R(z,w)= {\rm e}^{-|z|^2-|w|^2}
	\begin{pmatrix}
		\kappa_{{\rm bulk}}^\R(z,w) & \kappa_{{\rm bulk}}^\R(z,\bar{w})
		\vspace{1mm}\\
		\kappa_{{\rm bulk}}^\R(\bar{z},w) & \kappa_{{\rm bulk}}^\R(\bar{z},\bar{w})
	\end{pmatrix},
\end{gather}
where the pre-kernel $\kappa_{{\rm bulk}}^\R$ is given by
\begin{gather} \label{kappa bulk}
	\kappa_{{\rm bulk}}^\R(z,w):= \sqrt{\pi} {\rm e}^{ z^2+w^2 } \erf (z-w).
\end{gather}
We also refer to \cite{Mehta} for an alternative derivation of \eqref{kappa bulk}.
Furthermore, it was shown in \cite{BE,BL2} that the kernel \eqref{K bulk} also appears when $p \in \textup{int}(S) \cap \R$
in the symplectic elliptic Ginibre ensemble
and in that sense is universal
(in \cite[Appendix B]{akemann2019universal} a different strategy at $\tau=0$ was mentioned).

Our first main result Theorem~\ref{Thm_edge kernel} below provides the boundary scaling limit when $p \in \partial S \cap \R=\big\{ {\pm} \sqrt{2} \big\}$, see Figure~\ref{Fig. RN_Gaussian} for the graphs of $R_N$, respectively.

\begin{figure}[h!]
	\begin{subfigure}{0.48\textwidth}
		\begin{center}	
			\includegraphics[width=0.6667\textwidth]{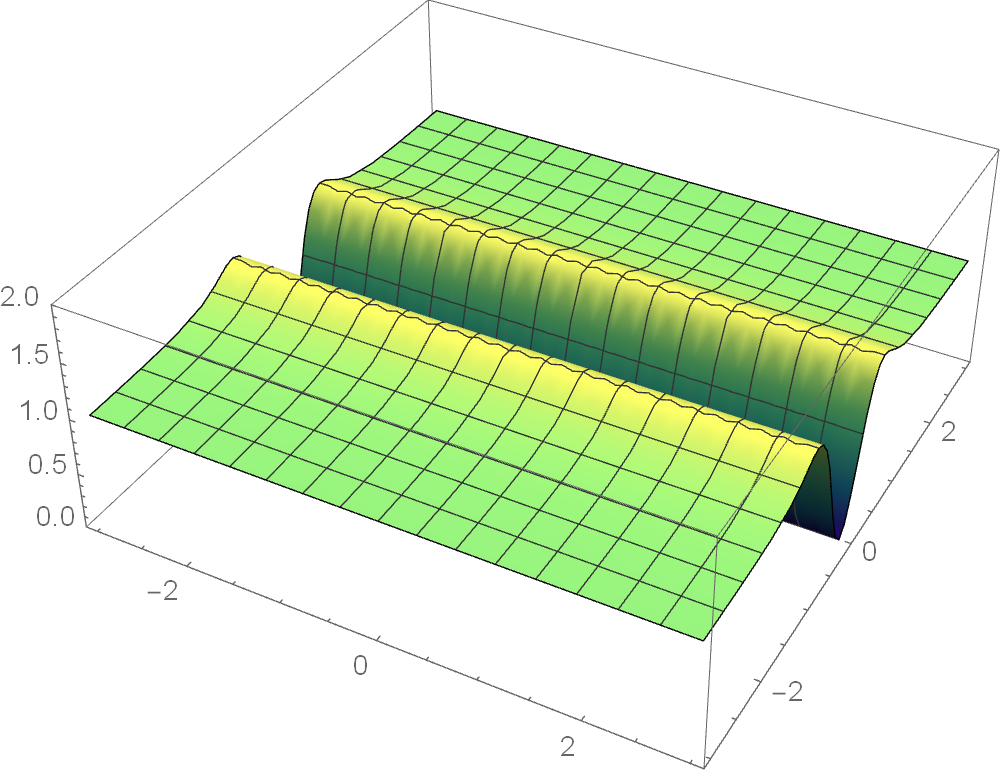}
		\end{center}
		\subcaption{$p=0$}
	\end{subfigure}	
	\begin{subfigure}[h]{0.48\textwidth}
		\begin{center}
			\includegraphics[width=0.6667\textwidth]{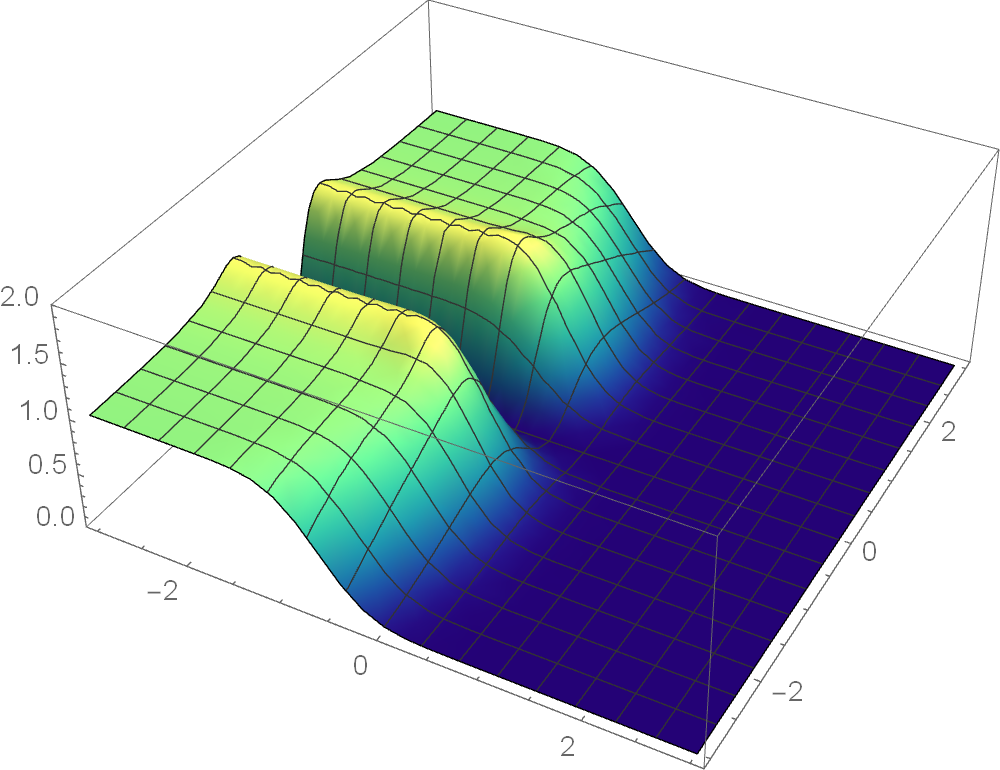}
		\end{center} \subcaption{$p=\sqrt{2}$}
	\end{subfigure}
	
\begin{subfigure}{0.32\textwidth}
		\begin{center}	
			\includegraphics[width=\textwidth]{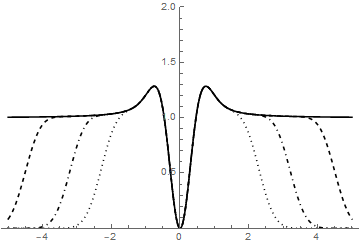}
		\end{center}
		\subcaption{$p=0$, $x=0$}
	\end{subfigure}	
	\begin{subfigure}{0.32\textwidth}
		\begin{center}	
			\includegraphics[width=\textwidth]{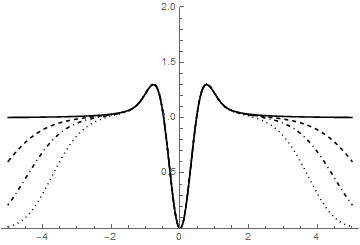}
		\end{center}
		\subcaption{$p=\sqrt{2}$, $x=-1$}
	\end{subfigure}	
	 \begin{subfigure}{0.32\textwidth}
		\begin{center}	
			\includegraphics[width=\textwidth]{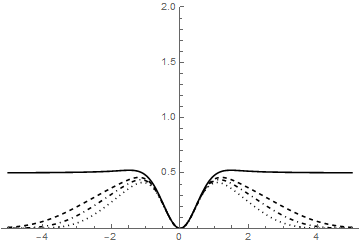}
		\end{center}
		\subcaption{$p=\sqrt{2}$, $x=0$}
	\end{subfigure}	
	\caption{Plots (a) and (b) show the surface graphs of the local densities $R_{N,1}$, where $Q(\zeta)=|\zeta|^2$ and $N=50$.
	Plot (c) displays the graph of the limiting density $R_1(x+{\rm i}y)$ and its comparison with $R_{N,1}(x+{\rm i}y)$ for $p=0$ restricted to $x=0$. Here $N=5$ (dotted line), $N=10$ (dot-dashed line), $N=20$ (dashed line) and $N=\infty$ (full line).
	Plots (d) and (e) are the same figures for $p=\sqrt{2}$ restricted to $x=-1$ and $x=0$ respectively, where $N=50$ (dotted line), $N=100$ (dot-dashed line), $N=200$ (dashed line) and $N=\infty$ (full line).
 From plots (c)--(e) it can be observed that the speed of convergence is faster in the bulk than at the edge.} \label{Fig. RN_Gaussian}
\end{figure}

Let
\begin{gather} \label{f_z(u)}
	f_z(u):=\tfrac12 \erfc\big( \sqrt{2}(z-u) \big)
\end{gather}
and write $W(f,g)$ for the Wronskian of two functions $f$, $g$ to formulate our first main result.

\begin{Theorem} \label{Thm_edge kernel}
	Let $Q(\zeta)=|\zeta|^2$ and $p=\pm \sqrt{2}$.
Then $R_{N,k}(z_1,\dots, z_k)$ converges uniformly for $z_1,\dots, z_k$ in compact subsets of $\C$ to
	\begin{gather*}
		R_k(z_1,\dots, z_k)= \prod_{j=1}^{k} ( \bar{z}_j-z_j ) \Pf \big[ K_{{\rm edge}}^{\R}(z_j,z_l) \big]_{ j,l=1 }^k,
	\end{gather*}
	where
	\begin{gather*} %\label{K Gkernel}
		K_{{\rm edge}}^\R(z,w)={\rm e}^{-|z|^2-|w|^2}
		\begin{pmatrix}
			\kappa_{{\rm edge}}^\R (z,w) & \kappa_{{\rm edge}}^\R (z,\bar{w})
			\vspace{1mm}\\
			\kappa_{{\rm edge}}^\R (\bar{z},w) & \kappa_{{\rm edge}}^\R (\bar{z},\bar{w})
		\end{pmatrix},
	\end{gather*}
	and
	\begin{gather} \label{kappa edge}
		\kappa_{{\rm edge}}^\R(z,w)
		:= \sqrt{\pi} {\rm e}^{z^2+w^2} \int_{E} W(f_w,f_z)(u) \, {\rm d}u, \qquad E=(-\infty,0).
	\end{gather}
\end{Theorem}

We remark that the pre-kernel $	\kappa_{{\rm edge}}^\R$ has the following alternative representation
\begin{gather*}
	\kappa_{{\rm edge}}^\R(z,w)={\rm e}^{2zw} \int_{-\infty}^{0} {\rm e}^{-t^2} \sinh( 2t(w-z) ) \erfc(z+w-t)\,{\rm d}t.
\end{gather*}
As a consequence of Theorem~\ref{Thm_edge kernel}, we obtain the following corollary.

\begin{Corollary} \label{Cor_Pf to det}
	Let $Q(\zeta)=|\zeta|^2$, $p_N=\pm \sqrt{2} {\rm e}^{{\rm i}\theta_N}$, where $\theta_N=\frac{t}{\sqrt{N}}$, $t \in \R$. Then $R_{N,k}(z_1,\dots, z_k)$ converges uniformly for $z_1,\dots, z_k$ in compact subsets of $\C$ to
	\begin{gather}\label{R_k^t Pf}
		R_k^t(z_1,\dots, z_k)= \prod_{j=1}^{k} ( \bar{z}_j-z_j-2{\rm i}t ) \Pf \big[ K_{{\rm edge}}^t(z_j,z_l) \big]_{ j,l=1 }^k,
	\end{gather}
	where
	\begin{gather*}
		K_{{\rm edge}}^t(z,w):={\rm e}^{-|z+{\rm i}t|^2-|w+{\rm i}t|^2}
		\begin{pmatrix}
			\kappa_{{\rm edge}}^\R (z+{\rm i}t,w+{\rm i}t) & \kappa_{{\rm edge}}^\R (z+{\rm i}t,\bar{w}-{\rm i}t)
			\vspace{1mm}\\
			\kappa_{{\rm edge}}^\R (\bar{z}-{\rm i}t,w+{\rm i}t) & \kappa_{{\rm edge}}^\R (\bar{z}-{\rm i}t,\bar{w}-{\rm i}t)
		\end{pmatrix}.
	\end{gather*}
	Moreover, as $t \to \infty$, we have
	\begin{gather} \label{Pf to det}
		R_k^t(z_1,\dots, z_k)= \det \big[ K_{{\rm edge}}^\C (z_j,z_l) \big]_{j,l=1}^k \cdot (1+o(1)),
	\end{gather}
	where the $o(1)$-term is uniform on $z_1,\dots, z_k$ in compact subsets of $\C$ and
	\begin{gather} \label{K erfc cplx}
		K_{{\rm edge}}^\C(z,w):={\rm e}^{ -|z|^2-|w|^2+2z\bar{w} } \, \tfrac{1}{2}\erfc(z+\bar{w}).
	\end{gather}
\end{Corollary}

We emphasise that the kernel \eqref{K erfc cplx} corresponds to that of the complex Ginibre ensemble at the edge of the spectrum (up to a trivial rescaling), which is known to form a determinantal point process, see, e.g., \cite{MR2530159,MR1690355,kanzieper2005exact}. Therefore the convergence \eqref{Pf to det} implies that away from the real axis, the local edge statistics of the symplectic and complex Ginibre ensemble are equivalent in the large-$N$ limit.
We also refer the reader to \cite{Dubach2020,MR1986426} and \cite{akemann2019universal} for the equivalence of the symplectic and complex Ginibre ensemble in the context of the scaled maximal modulus and the local bulk statistics away from the real axis, respectively.

Before moving on to our next topic, let us give some further remarks on Theorem~\ref{Thm_edge kernel}. These complete the asymptotic study of the symplectic Ginibre ensemble with the Gaussian potential.

\begin{Remark}\samepage \quad
	\begin{enumerate}\itemsep=0pt
\item[(i)] Recently, the scaling limit of the symplectic Ginibre ensemble at the edge of the spectrum has been obtained independently by Khoruzhenko and Lysychkin \cite{BL2}, see also \cite{Lysychkin} for more details. Contrary to our approach using a differential equation satisfied by the pre-kernel, their methods are based on contour integral representations, which yields the same result.
	
\item[(ii)] %\label{uni rep}
		For $p \in \big({-}\sqrt{2},\sqrt{2}\big)$, the pre-kernel \eqref{kappa bulk} can be obtained from the same method used in the proof of Theorem~\ref{Thm_edge kernel}. Indeed, the pre-kernel \eqref{kappa bulk} can be represented by \eqref{kappa edge} with $E=(-\infty,\infty)$ using the well-known integral
		\begin{gather*}
			\int_{ \R } \frac{1}{\sqrt{2\pi} \sigma} {\rm e}^{ -\frac{(x-\mu)^2}{2\sigma^2} } \erf(ax+b) \, {\rm d}x=\erf\left( \frac{a\mu+b}{\sqrt{1+2a^2\sigma^2}} \right).
		\end{gather*}
		The same integral identity holds for the complementary error function, replacing $\erf\to\erfc$.
		
\item[(iii)] Theorem~\ref{Thm_edge kernel} can be generalised to a moving boundary point $p_N$. More precisely, if we set
		\begin{gather*}
			p_N:=\sqrt{2}-a\sqrt{ \frac{2}{N} } , \qquad a\in \R,
		\end{gather*}
		and rescale the process as \eqref{rescaling}, then the limiting correlation kernel is of the form \eqref{kappa edge} with $E=(-\infty,a).$ When $a \gg 1$ we recover the bulk limit due to~(ii).
		
\item[(iv)] It is well known that the local bulk/edge correlation kernel $K^\C$ of the random normal matrix ensemble is given by
		\begin{gather*}
		K^\C(z,w)={\rm e}^{z\bar{w}-|z|^2/2-|w|^2/2} \frac{1}{\sqrt{2\pi}} \int_{E} {\rm e}^{-(z+\bar{w}-u)^2/2 } \, {\rm d}u,
		\end{gather*}
		where $E=(-\infty,\infty)$ for the bulk case and $E=(-\infty,0)$ for the edge case, see \cite{MR2817648, hedenmalm2017planar}.
		Thus one can interpret the integral representation \eqref{kappa edge} as a symplectic analogue of the above expression.
		
\item[(v)]	In \cite{MR3192169}, the authors studied the ensemble \eqref{Gibbs} with the elliptic potential
		\begin{gather*}
			Q(\zeta)=\tfrac{1}{1-\tau^2} \big( |\zeta|^2-\tau \re \zeta^2 \big), \qquad \tau \in [0,1).
		\end{gather*}
		For this model, the local correlation kernel at the right/left endpoint of the spectrum was derived in the regime of weak non-Hermiticity when $1-\tau=O\big( N^{-1/3}\big)$. The kernel \eqref{kappa edge} was derived there as well from the large argument limit of such an intermediate process, see \cite[equation~(4.28)]{MR3192169}.
		
\item[(vi)] It was proved in \cite{MR4229527} that the limiting local kernel of the chiral complex Ginibre ensemble at multi-criticality is given by the edge kernel of the complex Ginibre ensembles in squared variables. A quaternion version of such a result will appear in future work.
	\end{enumerate}
\end{Remark}

Next, we investigate ensembles containing certain types of singularities at the origin. More precisely, we consider the \textit{Mittag-Leffler ensemble} whose associated potential $Q$ is of the form
\begin{gather} \label{Q ML}
	Q(\zeta)=|\zeta|^{2\lambda}-\frac{2c}{N} \log|\zeta|, \qquad \lambda>0, \quad c>-1.
\end{gather}
Here the condition $c>-1$ is required to guarantee $Z_N < \infty$, where $Z_N$ is the partition function in \eqref{Gibbs}. Note that when $c \not=0$, a ``conical singularity'' at the origin arises from an insertion of a point charge $c$.
On the other hand, since the limiting global density of the ensemble is $\Delta \tfrac12 Q(\zeta)=\tfrac{\lambda^2}2 |\zeta|^{2\lambda-2}$, when $\lambda>1$ (resp., $\lambda<1$) it vanishes (resp., diverges) at the origin.

We aim to discover the local statistics of the symplectic Mittag-Leffler ensemble, which provides ``non-standard'' or multi-critical universality classes \eqref{K ML} beyond \eqref{kappa bulk}.
Away from the origin we expect to be back in the Gaussian universality class of the Ginibre ensemble.
Note that the micro-scale at the origin is given here by
\begin{gather} \label{micro scale ML}
	r_N=\left(\frac{\lambda}{2}N\right)^{-\frac{1}{2\lambda}}.
\end{gather}
For each $\lambda>0$ and $c>-1$, the rescaled limiting local kernel $K_{\lambda,c}$ at $p=0$ is of the form
\begin{gather} \label{K ML}
	K_{\lambda,c}(z,w)
	={\rm e}^{ -\frac{ |z|^{2\lambda}+|w|^{2\lambda} }{\lambda} } \begin{pmatrix}
		\kappa_{\lambda,c}(z,w) & \kappa_{\lambda,c}(z,\bar{w}) \\
		\kappa_{\lambda,c}(\bar{z},w) & \kappa_{\lambda,c}(\bar{z},\bar{w})
	\end{pmatrix} ,
\end{gather}
where $\kappa_{\lambda,c}$ is a holomorphic function in both variables (with branch cuts).
Since the correlation functions do not depend on the specific choice of these branches, we will ignore this issue in the sequel, such a monodromy of the pre-kernel can also be interpreted as a cocycle. (Cf.\ Section~\ref{Subsec_bdy Ginibre} for the definition of a cocycle.)
In general, we derive a certain $(1/\lambda)$-order fractional differential equation for $\kappa_{\lambda,c}$ (Proposition~\ref{Thm_bulkS}), which can be recognised as a version of the Christoffel--Darboux formula for the kernel~\eqref{K ML}.
Moreover, for $\lambda=1/m$ $(m\in \mathbb{N})$, we obtain an explicit formula for $\kappa_{\lambda,c}$ by solving the associated differential equation of order~$m$.

To describe the local kernel of the Mittag-Leffer ensemble, let us first recall the definitions of the Mittag-Leffler functions. By definition, the \emph{two-parametric Mittag-Leffler function} $E_{a,b}$ is given by
\begin{gather} \label{ML function}
	E_{a,b}(z):=\sum_{k=0}^{\infty} \frac{z^k}{\Gamma(ak+b)},
\end{gather}
and the \emph{three-parametric $($Kilbas--Saigo$)$ Mittag-Leffler function} $E_{\alpha,m,l}$ is given by
\begin{gather} \label{ML3}
	E_{\alpha,m,l}(z):=1+\sum_{k=1}^\infty z^k \prod_{j=0}^{k-1} \frac{ \Gamma(\alpha(jm+l)+1) }{ \Gamma(\alpha(jm+l+1)+1) },
\end{gather}
see \cite[Chapters~4 and~5]{gorenflo2014mittag}.
For a positive integer $m$ we write
\begin{gather}\label{F integer}
	F_{m,c}(z):=z^{m(1+c)-1} E_{2m,m(1+c)}\big(z^{2m}\big).
\end{gather}
For each $j =1,\dots,m$, we write
\begin{gather} \label{m indep sol}
g_{j,m}(z):=z^{j-1} E_{m,2,1+\frac{j-1}{m}}\big(z^{2m}\big),
\end{gather}
and
\begin{gather} \label{W_j(z)}
W_{j,m}(z):=W( g_{1,m},\dots, g_{j-1,m}, g_{j+1,m}, \dots, g_{m,m} ),
\end{gather}
where $W$ is the Wronskian.

\begin{Theorem} \label{Thm_ML integer}
	For $Q(\zeta)=|\zeta|^{2/m}-\frac{2c}{N} \log|\zeta|$ with $\lambda=1/m$ $(m\in \mathbb{N})$, $c>-1$, we have
	\begin{gather*}
		\kappa_{\lambda,c}(z,w)=\left( \frac{2}{\lambda} \right)^{ \frac{1}{2\lambda} } (zw)^{\lambda-1} \widetilde{\kappa}_{\lambda,c} \left( \sqrt{ \frac{2}{\lambda} } z^\lambda, \sqrt{ \frac{2}{\lambda} } w^\lambda \right),
	\end{gather*}
	where for $m=1$,
	\begin{gather*}
	\widetilde{\kappa}_{1,c}(z,w)= \int_0^1 \big(z {\rm e}^{ \frac12 (1-s^2)z^2 } -w {\rm e}^{ \frac12 (1-s^2)w^2 }\big)F_{1,c}(szw)\,{\rm d}s,
	\end{gather*}
	and for $m>1$,
	\begin{gather*}
	\widetilde{\kappa}_{\lambda,c}(z,w)= \sum_{j=1}^{m} \frac{(-1)^{m-j}}{G(m+1)} \int_0^1 \big( z g_{j,m}(z)W_{j,m}(sz)-w g_{j,m}(w)W_{j,m}(sw)\big)F_{m,c}(szw)\,{\rm d}s.
	\end{gather*}
	Here $G(m+1):=\prod_{j=1}^{m-1} j!$ is the Barnes $G$-function.
\end{Theorem}

\begin{Remark}
	For any $m \ge 1$, one can write $	\widetilde{\kappa}_{\lambda,c}$ in a unified way as
	\begin{gather*}
	\widetilde{\kappa}_{\lambda,c}(z,w)= \sum_{j=1}^{m} (-1)^{m-j} \int_0^1 \frac{ z g_{j,m}(z)W_{j,m}(sz)-w g_{j,m}(w)W_{j,m}(sw) }{ W(g_{1,m},\dots,g_{m,m})(sz) } F_{m,c}(szw)\,{\rm d}s.
	\end{gather*}
\end{Remark}

Theorem \ref{Thm_ML integer} is our second main result, and we shall present some examples for $m=1,2$.

\begin{Example}[$\lambda=1$]
	We first discuss the case $\lambda=1$, where the Mittag-Leffler ensemble corresponds to the eigenvalue statistics of the so-called \emph{induced symplectic Ginibre ensemble}. This name was proposed for the matrix representation of real and complex Ginibre ensembles in the presence of zero eigenvalues~\cite{MR2881072}.
	We refer to~\cite{MR2180006} for the complex eigenvalue correlation functions in the symplectic Ginibre ensemble in the presence of zero eigenvalues in terms of skew-orthogonal polynomials at finite-$N$.
	
	For $m=1$, we have $g_{1,1}(z)={\rm e}^{z^2/2}$.
	This immediately gives
	\begin{gather} \label{kappa c alt2}
		\kappa_{1,c}(z,w)= 2 (2zw)^c \int_0^1 s^c \big(z{\rm e}^{(1-s^2)z^2}-w{\rm e}^{(1-s^2)w^2}\big) E_{2,1+c}\big((2szw)^2\big)\,{\rm d}s. 	
	\end{gather}
	Note that by \eqref{ML function}, we have
	\begin{gather*}
		2 E_{2,1+c}(z^2) = E_{1,1+c}(z)+E_{1,1+c}(-z)
		={\rm e}^z z^{-c} P(c,z)+{\rm e}^{-z} (-z)^{-c} P(c,-z),
	\end{gather*}
	where
	\begin{gather*}
		P(c,z):=\frac{1}{\Gamma(c)}\int_{0}^{z} t^{c-1} {\rm e}^{-t} \, {\rm d}t, \qquad c>0,
	\end{gather*}
	is the (regularised) incomplete Gamma function. Using this, we have an alternative representation for $c\ge 0$,
	\begin{gather*}
	\kappa_{1,c}(z,w)= \int_0^1\!\! \big(z{\rm e}^{(1-s^2)z^2}\!-w{\rm e}^{(1-s^2)w^2}\big)\big( {\rm e}^{2szw} P(c,2szw)+(-1)^{-c} {\rm e}^{-2szw} P(c,-2szw) \big)\,{\rm d}s.
	\end{gather*}
	
	Furthermore, for a non-negative integer $c$ one can express the pre-kernel~\eqref{kappa c alt2} in terms of error functions. For this we denote by $(a)_n$ Pochhammer's symbol:
	\begin{gather*}
	(a)_0=1,\qquad (a)_n=a(a+1)\cdots (a+n-1).
	\end{gather*}
	Then for an even integer $c=2n$ $(n\in \mathbb{N})$ we have
	\begin{gather}
			\kappa_{1,c}(z,w)=\sqrt{\pi} {\rm e}^{z^2+w^2} \erf(z-w)	+\sum_{ 0 \le l < k \le \frac{c}{2}-1 } \frac{ w^{2k} z^{2l+1}-z^{2k} w^{2l+1} }{k! (1/2)_{l+1} }
			\nonumber\\
\hphantom{\kappa_{1,c}(z,w)=}{}
 +\sqrt{\pi} {\rm e}^{w^2}\erf(w) \sum_{k=0}^{c/2-1} \frac{ z^{2k} }{k!} -\sqrt{\pi} {\rm e}^{z^2}\erf(z) \sum_{k=0}^{c/2-1} \frac{ w^{2k} }{k!},\label{kappa c even}
	\end{gather}
	and for an odd integer $c=2n-1$ $(n\in \mathbb{N})$, we have
\begin{gather}
\kappa_{1,c}(z,w)=\sqrt{\pi} {\rm e}^{z^2+w^2} ( \erf(z-w)-\erf(z) +\erf(w) )+\sum_{ 1 \le l < k \le \frac{c-1}{2} } \frac{ w^{2k-1} z^{2l}-z^{2k-1} w^{2l} }{l!(1/2)_k}\nonumber\\
\hphantom{\kappa_{1,c}(z,w)=}{} + ({\rm e}^{w^2}-1) \sum_{k=1}^{(c-1)/2} \frac{ z^{2k-1}}{(1/2)_k } - ({\rm e}^{z^2}-1) \sum_{k=1}^{(c-1)/2} \frac{ w^{2k-1}}{(1/2)_k}.\label{kappa c odd}
\end{gather}
	Here we use the convention that the summation with an empty index equals zero.
	Both \eqref{kappa c even} and \eqref{kappa c odd} can be obtained by straightforward computations using
	\begin{align*}
		P(n,z)=1-{\rm e}^{-z} \sum_{k=0}^{n-1} \frac{z^k}{k!}, \qquad n=0,1,2,\dots,
	\end{align*}
	see, e.g., \cite[equation~(8.4.9)]{olver2010nist}.
	
	In \cite{MR2302902}, the $k$-point correlation function $R_{N,k}$ for an integer-valued point charge $c$ is presented in a different way as the ratio of Pfaffians of the correlation kernels of $c=0$.
	
	\begin{figure}[h!]
		\begin{subfigure}{0.48\textwidth}
			\begin{center}	
		 \includegraphics[width=0.6667\textwidth]{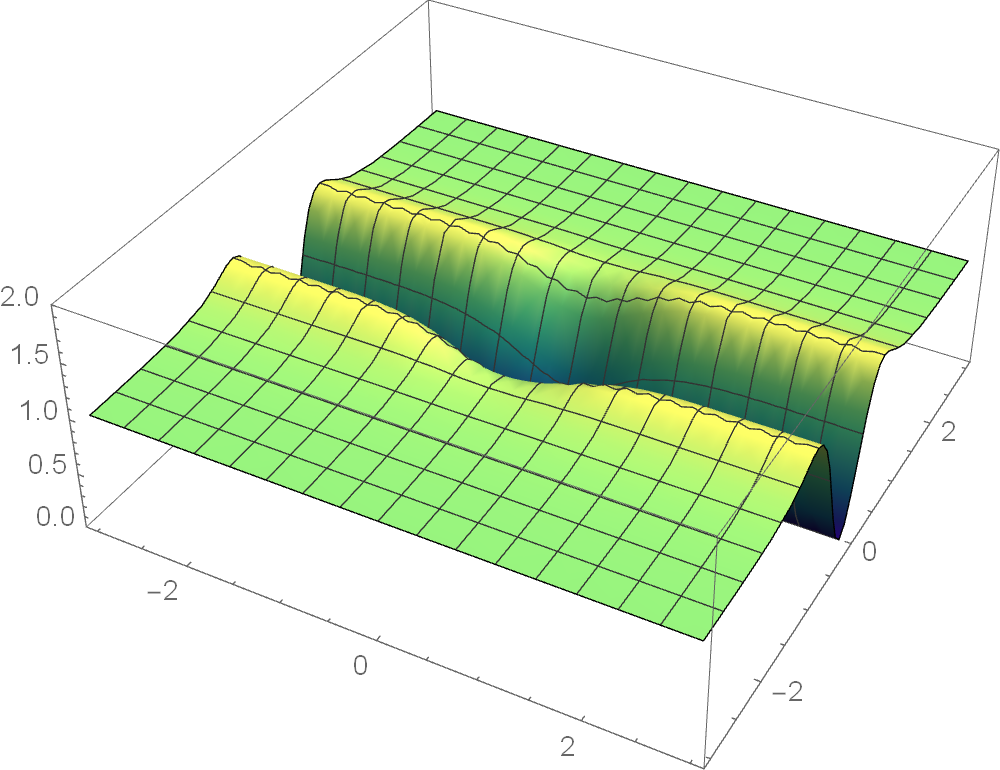}
			\end{center}
			\subcaption{$c=\frac12$}
		\end{subfigure}	
		\begin{subfigure}[h]{0.48\textwidth}
			\begin{center}
		\includegraphics[width=0.6667\textwidth]{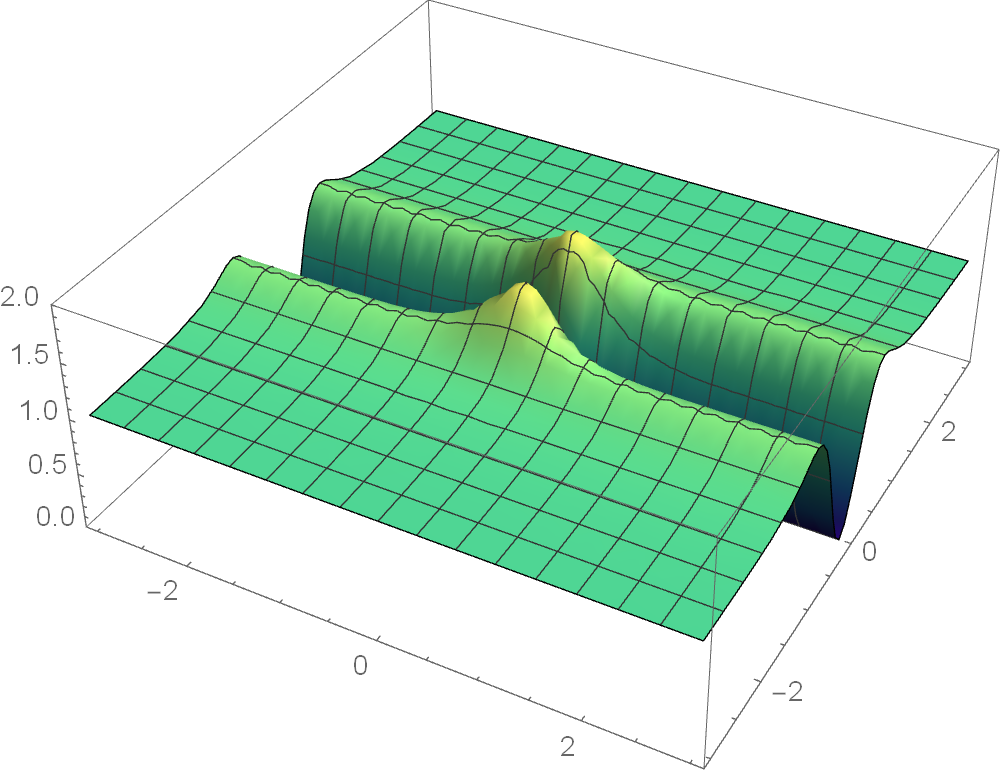}
			\end{center} \subcaption{$c=-\frac12$}
		\end{subfigure}
		
		\begin{subfigure}{0.48\textwidth}
			\begin{center}	
		 \includegraphics[width=0.6667\textwidth]{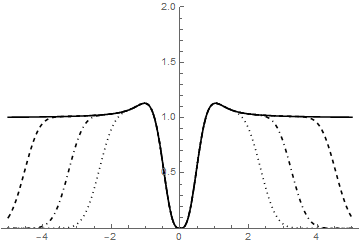}
			\end{center}
			\subcaption{$c=\frac12$, $x=0$}
		\end{subfigure}	
		\begin{subfigure}[h]{0.48\textwidth}
			\begin{center}
		\includegraphics[width=0.6667\textwidth]{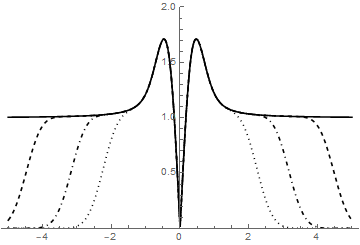}
			\end{center} \subcaption{$c=-\frac12$, $x=0$}
		\end{subfigure}
		\caption{Plots (a) and (b) show the microscopic density $R_{N,1}$, where $\lambda=1$, $c=\pm \frac12$ and $N=50$. Compared to Figure~\ref{Fig. RN_Gaussian}(a) when $c=0$, one can observe an additional local repulsion (resp., attraction) when $c>0$ (resp., $c<0$) at the origin.	Plots (c) and (d) are graphs of $R_{N,1}(x+{\rm i}y)$ and their comparisons with the large-$N$ limit $R_1(x+{\rm i}y)$ restricted to $x=0$. Here $N=5$ (dotted line), $N=10$ (dot-dashed line), $N=20$ (dashed line) and $N=\infty$ (full line).		} \label{Fig. RN_Insertion}
	\end{figure}
\end{Example}

\begin{Example}[$\lambda=\frac12$] For $m=2$, we have
	\begin{gather}\label{g 12 22}
		g_{1,2}(z)=\tfrac{1}{\sqrt{2}}\Gamma\big(\tfrac34\big) \sqrt{z} I_{-\frac14}\big(\tfrac{z^2}{2}\big), \qquad g_{2,2}(z)=\sqrt{2} \Gamma\big(\tfrac54\big) \sqrt{z} I_{\frac14} \big( \tfrac{z^2}{2} \big),
	\end{gather}
	where $I_\nu$ is the modified Bessel function of the first kind \cite[Chapter~10]{olver2010nist}:
	\begin{gather*}
		I_\nu(z):=\sum_{k=0}^\infty \frac{(z/2)^{2k+\nu}}{k! \Gamma(\nu+k+1)}.
	\end{gather*}
	Then using \eqref{g 12 22}, we have
	\begin{gather} \label{kappa 12 Wronskian}
		\kappa_{\frac12,c}(z,w)=\pi 2^{\frac72} (16zw)^{c}
		\int_0^1 s^{c} (z h_s(z)-wh_s(w)) E_{4,2+2c}\big((16szw)^2\big) \,{\rm d}s,	
	\end{gather}
	where
	\begin{gather*}
		h_s(z):= s^{\frac14} \big( I_{-\frac14}(2sz)I_{\frac14}(2z) - I_{\frac14}(2sz) I_{-\frac14}(2z) \big).
	\end{gather*}
	
	In particular, for $c =0,-\frac12$, the pre-kernel \eqref{kappa 12 Wronskian} can also be expressed as
	\begin{gather}
	\begin{split}
& \kappa_{\frac12,0}(z,w)= \frac{2}{\sqrt{z w}} \int_{0}^{\frac{\pi}{2}} (\sin \theta)^{-\frac12} \sinh\big( 4 \sqrt{zw} \sin\theta) \sinh( 2 (z-w) \cos\theta ) \, {\rm d}\theta,
		\\
& \kappa_{\frac12,-\frac12}(z,w)= \frac{2}{\sqrt{z w}} \int_{0}^{\frac{\pi}{2}} (\sin \theta)^{-\frac12} \cosh\big( 4 \sqrt{zw} \sin\theta\big) \sinh( 2 (z-w) \cos\theta ) \, {\rm d}\theta.
	\end{split} \label{kappa 12 special}
	\end{gather}
The expression \eqref{kappa 12 special} follows from \cite[Appendix~B]{MR2180006} due to the relation between the Mittag-Leffler potential at $m=2$, $Q(\zeta)=|\zeta|-\frac{2c}{N}\log|\zeta|$ for $c=0,-\frac12$, and the potential of the chiral symplectic Ginibre ensemble at maximal non-Hermiticity $\tau=0$, $Q(\zeta)=-\frac{1}{N}\log \big[|\zeta|^{2\nu}K_{2\nu}(N|\zeta|)\big]$ (corresponding to $\mu=1$ and the change of variables $\zeta=z^2$ therein). Namely, at $\nu=\pm\frac14$ the modified Bessel-function of the second kind simplifies, $K_{\pm\frac12}(x)=\sqrt{\frac{\pi}{2x}}{\rm e}^{-x}$, matching the two cases for $c=0,-\frac12$ up to an additive constant. In~\cite{MR2180006} the relation between the chiral symplectic Ginibre ensemble in the origin scaling limit and quantum chromodynamics with two colours and chemical potential was pointed out, cf.~\cite{ABi06}.
We also refer to the recent work~\cite{akemann2021skew} where the pre-kernel at the origin was determined in the more general case of the chiral elliptic potential, with $\tau \in [0,1)$.
	
Notice, however, that the equivalence between \eqref{kappa 12 Wronskian} and \eqref{kappa 12 special} is far from being obvious. One can verify this by showing that both of these expressions satisfy the same differential equation of second order (Proposition~\ref{Thm_bulkS}) with the same initial conditions, which uniquely determine the solution. In particular, the initial conditions can be easily checked using the integral representation \cite[equation~(11.5.6)]{olver2010nist} of the modified Struve function $\textbf{L}_\nu$.
	
	\begin{figure}[h!]
		\begin{subfigure}{0.32\textwidth}
			\begin{center}	
				\includegraphics[width=\textwidth]{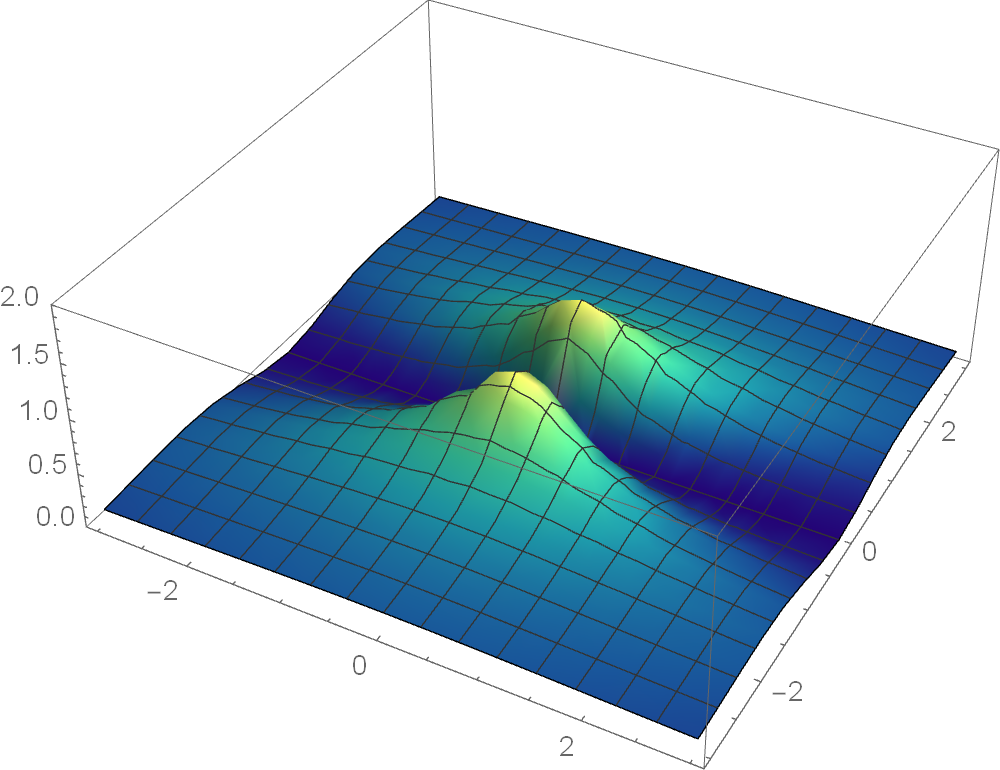}
			\end{center}
			\subcaption{$c=\frac14$}
		\end{subfigure}	
		\begin{subfigure}{0.32\textwidth}
			\begin{center}	
				\includegraphics[width=\textwidth]{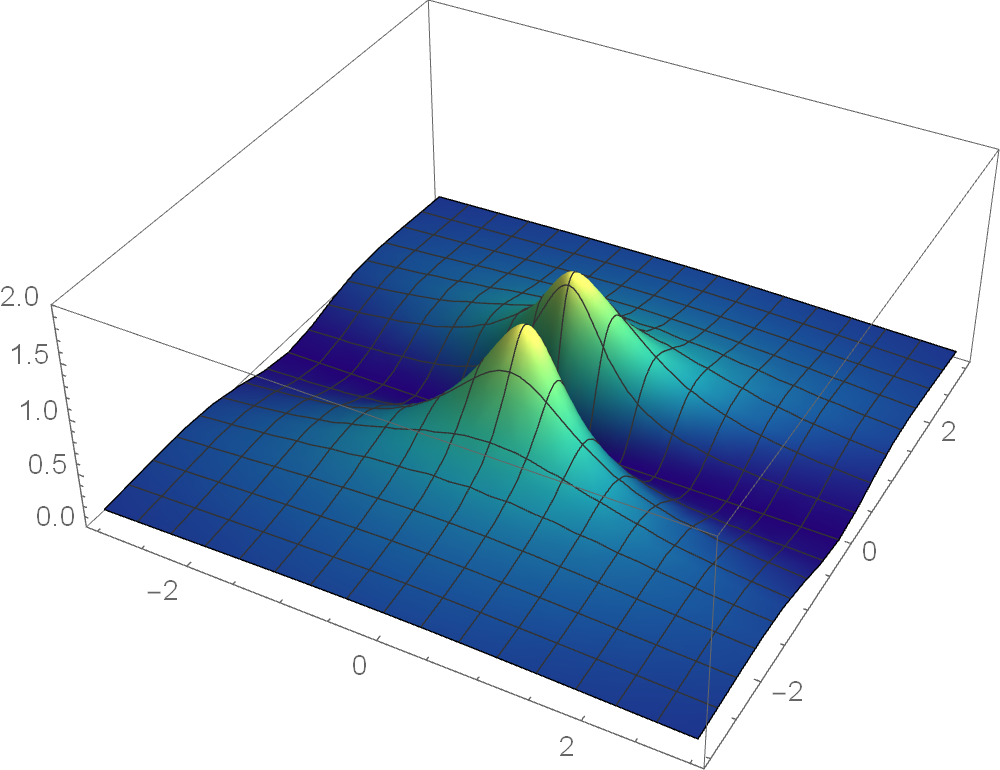}
			\end{center}
			\subcaption{$c=0$}
		\end{subfigure}	
		\begin{subfigure}{0.32\textwidth}
			\begin{center}	
				\includegraphics[width=\textwidth]{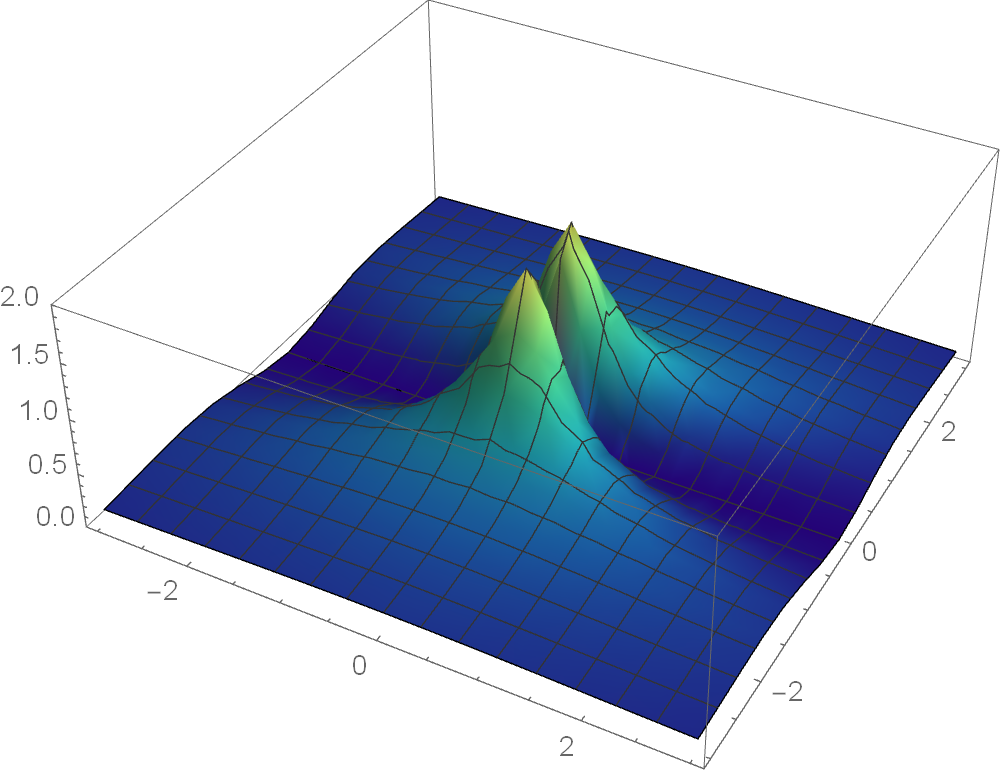}
			\end{center}
			\subcaption{$c=-\frac14$}
		\end{subfigure}
		
			\begin{subfigure}{0.32\textwidth}
			\begin{center}	
				\includegraphics[width=\textwidth]{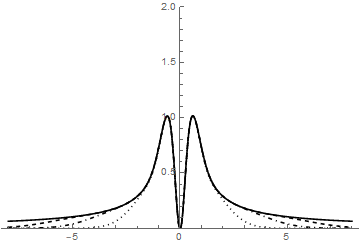}
			\end{center}
			\subcaption{$c=\frac14$, $x=0$}
		\end{subfigure}	
		\begin{subfigure}{0.32\textwidth}
			\begin{center}	
				\includegraphics[width=\textwidth]{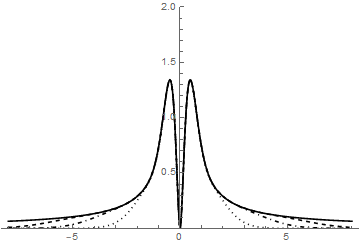}
			\end{center}
			\subcaption{$c=0$, $x=0$}
		\end{subfigure}	
		\begin{subfigure}{0.32\textwidth}
			\begin{center}	
				\includegraphics[width=\textwidth]{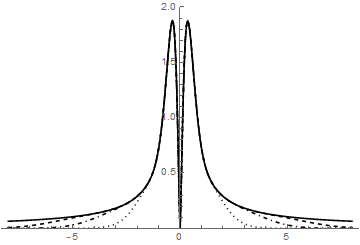}
			\end{center}
			\subcaption{$c=-\frac14$, $x=0$}
		\end{subfigure}
		\caption{Plots (a)--(c) display the surface graphs of the $1$-point function $R_{N,1}$, where $\lambda=\frac12$, $c= 0,\pm \frac14$ and $N=50$. Plots (d)--(f) are graphs of $R_{N,1}(x+{\rm i}y)$ and their comparisons with the large-$N$ limit $R_1(x+{\rm i}y)$ restricted to $x=0$. Here $N=2$ (dotted line), $N=4$ (dot-dashed line), $N=6$ (dashed line) and $N=\infty$ (full line).} \label{Fig. BS}
	\end{figure}
\end{Example}

In the third part, let us turn to the symplectic ensemble with general external potential $Q$. We present two important functional equations satisfied by the correlation kernel, the \textit{mass-one} and \textit{Ward's equation}.

First, we define the \textit{Berezin kernel}
\begin{gather*}
	B_N(z,w):=\frac{R_{N}(z)R_{N}(w)-R_{N,2}(z,w)}{ R_{N}(z)}=R_N(w)-R_{N-1}^{(z)}(w),
\end{gather*}
where $R_{N-1}^{(z)}(w):=R_{N,2}(z,w)/R_N(z)$ is the $1$-point function of the $N$-point process, conditioned to contain the prescribed point $z$.
See Figure~\ref{Fig. BN_Gaussian} for graphs of the Berezin kernel.

\begin{figure}[h!]
	\begin{subfigure}{0.48\textwidth}
		\begin{center}	
			\includegraphics[width=0.6667\textwidth]{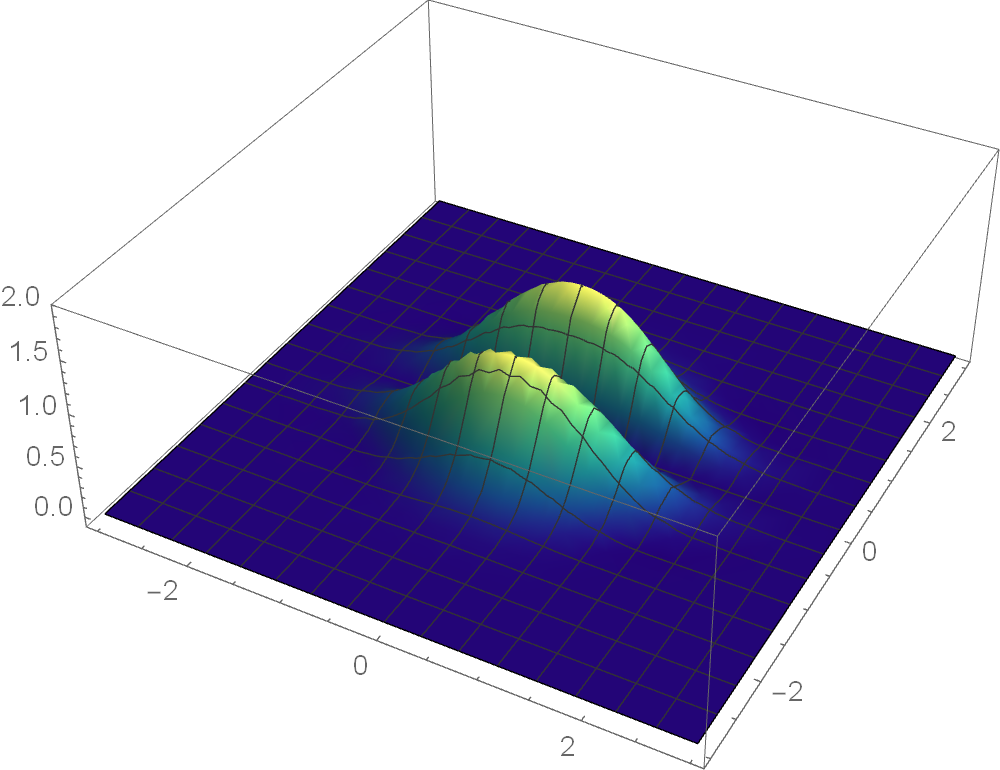}
		\end{center}
		\subcaption{$p=0$}
	\end{subfigure}	
	\begin{subfigure}[h]{0.48\textwidth}
		\begin{center}
			\includegraphics[width=0.6667\textwidth]{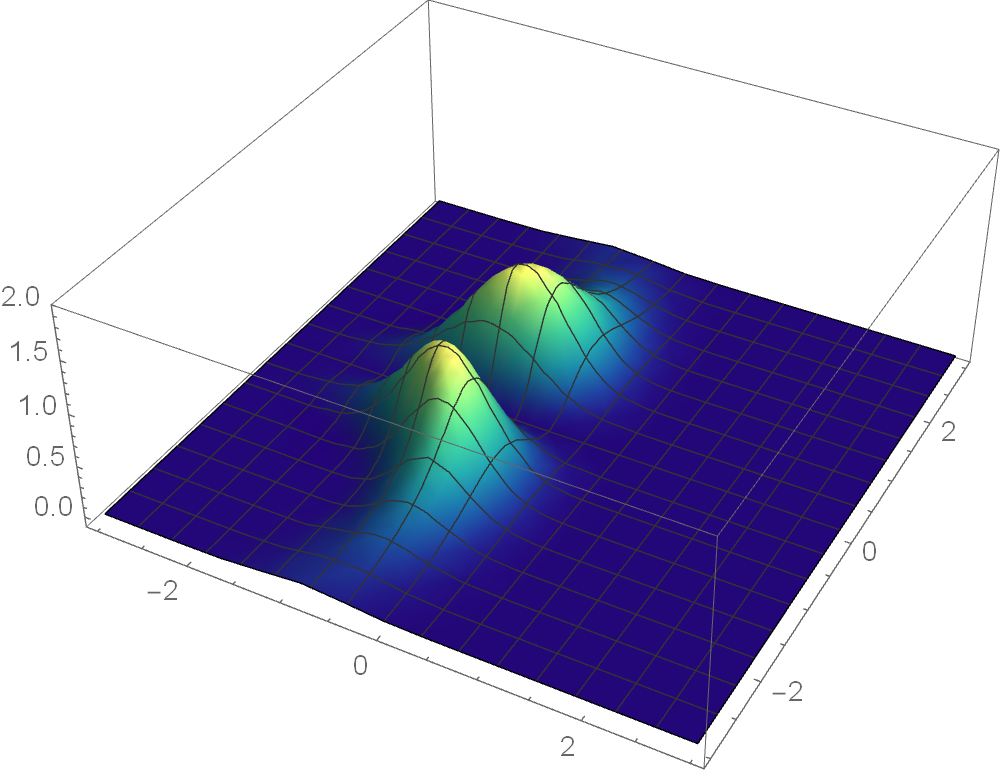}
		\end{center} \subcaption{$p=\sqrt{2}$}
	\end{subfigure}
	
	\begin{subfigure}{0.32\textwidth}
		\begin{center}	
			\includegraphics[width=\textwidth]{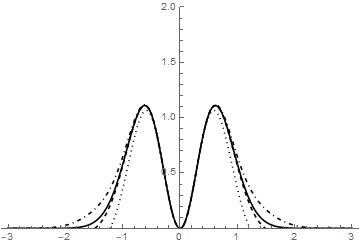}
		\end{center}
		\subcaption{$p=0$, $x=0$}
	\end{subfigure}	
	\begin{subfigure}{0.32\textwidth}
		\begin{center}	
			\includegraphics[width=\textwidth]{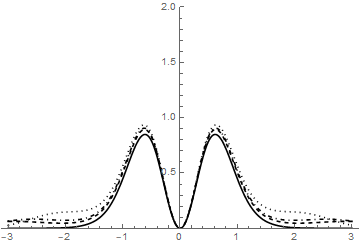}
		\end{center}
		\subcaption{$p=\sqrt{2}$, $x=-1$}
	\end{subfigure}	
	 \begin{subfigure}{0.32\textwidth}
		\begin{center}	
			\includegraphics[width=\textwidth]{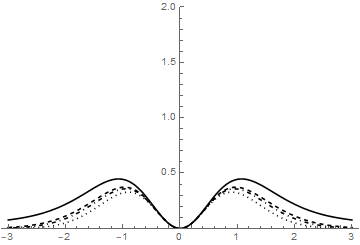}
		\end{center}
		\subcaption{$p=\sqrt{2}$, $x=0$}
	\end{subfigure}	
	\caption{Plots (a) and (b) show surface graphs of the Berezin kernel $B_N(0,w)$ for the Gaussian potential $Q(\zeta)=|\zeta|^2$, where $N=30$.
	Plot (c) is a graph of $B_N(0,x+{\rm i}y)$ and its comparison with its large $N$ limit $B(0,x+{\rm i}y)$ restricted to $x=0$. Here $N=2$ (dotted line), $N=3$ (dot-dashed line), $N=4$ (dashed line) and $N=\infty$ (full line).
	The plots (d) and (e) are graphs for $p=\sqrt{2}$ restricted to $x=-1$ and $x=0$ respectively, where $N=10$ (dotted line), $N=20$ (dot-dashed line), $N=30$ (dashed line) and $N=\infty$ (full line). As in Figure~\ref{Fig. RN_Gaussian}(c)--(e), the speed of convergence is faster in the bulk than at the edge. } \label{Fig. BN_Gaussian}
\end{figure}

By the definition of $R_{N,k}$ in \eqref{bfRNk def}, one can easily see that the mass-one equation
\begin{gather}\label{mass-one N}
	\int_{ \C } B_N(z,w)\,{\rm d}A(w)=1
\end{gather}
holds for finite-$N$. For a Pfaffian $\infty$-point process, let us define the associated Berezin kernel as
\begin{gather*}
B(z,w):=\frac{R(z)R(w)-R_{2}(z,w)}{ R(z)},
\end{gather*}
where $R_k:=\lim\limits_{N \to \infty} R_{N,k}$ and $R=R_1$.
Here and in the sequel, a point process is called Pfaffian $\infty$-point process if its correlation functions $R_k$ are expressed in terms of the Pfaffian of a certain correlation kernel.
The kernel is only unique up to a cocycle, as discussed in more detail in Section~\ref{Subsec_bdy Ginibre}.
By definition, such a Pfaffian $\infty$-point process is said to satisfy the \emph{mass-one equation} if
\begin{gather} \label{mass-one}
	\int_{ \C } B(z,w)\, {\rm d}A(w)=1.
\end{gather}

\begin{Proposition} \label{Thm_mass1}
	The Pfaffian $\infty$-point process with kernel
	\begin{gather} \label{K Gkernel E}
		K_{a}^\R(z,w)={\rm e}^{-|z|^2-|w|^2}
		\begin{pmatrix}
			\kappa_{a}^\R (z,w) & \kappa_{a}^\R (z,\bar{w})
			\vspace{1mm}\\
			\kappa_{a}^\R (\bar{z},w) & \kappa_{a}^\R (\bar{z},\bar{w})
		\end{pmatrix},
	\end{gather}
	specified by
	\begin{gather} \label{kappa a}
		\kappa_{a}^\R(z,w)
		:= \sqrt{\pi} {\rm e}^{z^2+w^2} \int_{-\infty}^a W(f_w,f_z)(u) \,{\rm d}u,
	\end{gather}
	satisfies the mass-one equation. Here $a \in \R$ and $f_z$ is given in \eqref{f_z(u)}.
\end{Proposition}

For a correlation kernel $K$ with (associated) pre-kernel $\kappa$ of the form
\begin{gather} \label{K regular}
		K(z,w)={\rm e}^{-|z|^2-|w|^2}
		\begin{pmatrix}
			\kappa (z,w) & \kappa (z,\bar{w})
			\\
			\kappa (\bar{z},w) & \kappa (\bar{z},\bar{w})
		\end{pmatrix},
	\end{gather}
the mass-one equation is equivalent to{\samepage
\begin{gather} \label{mass1 var}
	\kappa(z,\bar{z})=2 \int_{ \C } (\bar{w}-w) {\rm e}^{ -2|w|^2 } |\kappa(z,w)|^2 \, {\rm d}A(w).
\end{gather}
Thus Proposition~\ref{Thm_mass1} indicates that the kernels~\eqref{kappa a} are solutions to the integral equa\-tion~\eqref{mass1 var}.$\!$}

In the opposite direction, for a general potential~$Q$, suppose that the associated limiting Pfaffian point process $\boldsymbol{z}$ satisfies the mass-one equation~\eqref{mass1 var}. From an analytic point of view, this equation can be regarded as an integral equation satisfied by a pre-kernel~$\kappa$.
If one can characterise the solution to~\eqref{mass1 var} under some conditions derived from intrinsic properties of~$Q$, and the nature of the rescaling point~$p$ (e.g., whether $p \in \textrm{int}(S) \cap \R$ or $p \in \partial S \cap \R$), this provides an alternative way to obtain all possible candidates for the scaling limit of the symplectic ensemble associated with potential~$Q$.
This would show local universality as this approach does not depend on the specific choice of~$Q$.
(We refer to~\cite{MR3975882,MR4030288} for extensive studies on the universality for random normal matrix ensembles based on this approach.)
However, the mass-one equation~\eqref{mass1 var1} may have further solutions beyond~\eqref{kappa bulk} and~\eqref{kappa edge}.
Therefore, one needs additional information about the kernel to determine the solution uniquely.
This calls for an investigation of Ward's equation.

\begin{Remark}[Ward's equation and universality for the random normal matrix ensemble with translation invariant kernel]
We pause here to briefly introduce recent developments in Ward's equation and its significance in the context of universality for the random normal matrix ensemble.
Hopefully this gives more intuition as to why we aim to examine Ward's equation for symplectic ensembles as well.

For random normal matrix ensembles, the rescaled version of Ward's equation was studied in~\cite{MR3975882}. An important feature of this equation is that in the large-$N$ limit, it does not depend on the choice of potential $Q$ if $p$ is regular, in that sense that $0<\Delta Q(p) <\infty$ is non-vanishing and bounded. Due to this property, Ward's equation has been utilised to show the local universality conjectures in various situations, see, e.g., \cite{AB21, MR4030288, ameur2018random} and references therein.
To be more precise, the overall strategy for the universality proof using Ward's equation is as follows.
\begin{itemize}\itemsep=0pt
 \item \emph{Derivation of Ward's equation} (cf.\ \cite[Theorem~1.3]{MR3975882}, \cite[Lemma~2]{MR4030288}, \cite[Lemma~3.1]{ameur2018random}). The first step is to show that for a $C^2$-smooth potential $Q$ and a regular point $p$, the following form of Ward's equation holds:
 \begin{gather}\begin{split}
 & \bp \mathcal{C}(z)=\mathcal{R}(z)-1-\Delta \log \mathcal{R}(z), \\
 & \mathcal{C}(z):=\frac{1}{\mathcal{R}(z)}\int_\C \frac{\mathcal{R}(z)\mathcal{R}(w)-\mathcal{R}_2(z,w)}{z-w}\,{\rm d}A(w),
 \end{split}\label{Ward cplx}
 \end{gather}
 where $\mathcal{R} \equiv \mathcal{R}_1$ and $\mathcal{R}_2$ are the (limiting) rescaled $1$- and $2$-point functions of the random normal matrix ensemble with potential $Q$, i.e.,
 \begin{gather*}
 \mathcal{R}_k(z)=\lim_{N \to \infty} \frac{1}{(N \Delta Q(p))^k} \boldsymbol{\mathcal{R}}_{N,k}\left(p+\frac{z}{\sqrt{N\Delta Q(p)}}\right).
 \end{gather*}
 Here $\boldsymbol{\mathcal{R}}_{N,k}$ denotes the unscaled $k$-point functions.
 We emphasise that due to the rescaling factor $\sqrt{N\Delta Q(p)}$ chosen according to the mean eigenvalue spacing of the random normal matrix model, the equation \eqref{Ward cplx} does not depend on the choice of~$Q$.
 \item \emph{Structure of the correlation kernel} (cf.\ \cite[Theorem~1.1]{MR3975882}, \cite[Theorem~1.3]{ameur2018random}). The next step is to show that for a general potential $Q$ and a regular point $p$, the limiting correlation kernel $\mathcal{K}$ exists and is of the form
 \begin{gather*}
 \mathcal{K}(z,w)=\mathcal{G}(z,w) \Psi(z,w), \qquad \mathcal{G}(z,w):={\rm e}^{z\bar{w}-|z|^2/2-|w|^2/2},
 \end{gather*}
 where $\Psi$ is Hermitian, i.e., $\Psi(z,w)=\overline{ \Psi(w,z) }$ and $\Psi$ is entire as a function of $z$ and $\bar{w}$.

 \item \emph{Characterisation of translation-invariant solution to Ward's equation} (cf.\ \cite[Theorem~4]{MR4030288}). The third step is to show that the only non-trivial horizontally translation-invariant solutions $\mathcal{R}$ (i.e., $\mathcal{R}(z+t)=\mathcal{R}(z)$ for $t \in \R$
 along the real axis) to Ward’s equation \eqref{Ward cplx} are given by
 \begin{gather*}
 \mathcal{R}(z)=\frac{1}{\sqrt{2\pi}} \int_I {\rm e}^{-(2 \im z-t)^2/2}\,{\rm d}t,
 \end{gather*}
 where $I$ is a connected interval. Due to this step, the possible translation-invariant scaling limit is determined only by one interval~$I$.

 \item \emph{Specifying the interval} (cf.\ \cite[Theorem~1.5]{MR3975882}, \cite[Theorem~3.8]{AB21}). The interval $I$ depends on the situation. For instance, $I=\R$ if $p$ is in the bulk \cite{MR2817648,MR3342661} and $I=(-\infty,0)$ if $p$ is at the edge \cite{MR3975882,hedenmalm2017planar} of the droplet. Furthermore, $I=(-a,a)$ for some $a>0$ if one considers the almost-Hermitian limit \cite{AB21} or $p$ is close to a cusp type singularity \cite{MR4030288}. Determining the interval in each situation requires a separate analysis.

 \item \emph{Translation invariance of the correlation functions}. The translation invariance is easy to check if $Q$ is radially symmetric. As a consequence, for a radially symmetric potential, edge universality was shown in \cite[Theorem~1.8]{MR3975882}.
\end{itemize}

We remark that for a more general class of non-radially symmetric potentials, edge universality was recently shown by Hedenmalm and Wennman in \cite{hedenmalm2017planar}.
The authors used a different approach based on the theory of quasi-orthogonal polynomials.
However, this theory is not directly applicable to the symplectic ensemble since it is far from obvious how to construct
skew-orthogonal polynomials using (quasi-)orthogonal polynomials in general
(see however \cite{akemann2021skew}).

In contrast, the overall strategy using Ward's equation described above can be applied in parallel to the symplectic ensemble.
This is the primary purpose of our remaining discussion.
In the rest of this section, as an analogue of \cite{MR4030288}, we aim to particularly address the characterisation of the translation-invariant solution to Ward's equation for the symplectic ensemble.
\end{Remark}

Our next result is the resulting Ward's equation for symplectic ensembles. For this, let
\begin{gather*}
C_N(z):=\int_{ \C }\frac{B_N(z,w)}{z-w}\, {\rm d}A(w).
\end{gather*}
Then, we obtain the following form of Ward's equation at finite-$N$ for general $Q$.

\begin{Proposition} \label{Thm_Ward N}
	Suppose that $Q$ is $C^2$-smooth and $p \in \R$.
	Then for each $N,$ we have
	\begin{gather}		\label{Ward rescale N}
		\bar{\partial} C_N(z)= R_N(z) -\frac{Nr_N^2}{2} \Delta Q( p+r_N z )
		-\frac12 \Delta \log R_N(z) + \frac{ 1 }{ (z-\bar{z})^2 }.
	\end{gather}
\end{Proposition}

Let us discuss the large-$N$ limit of Ward's equation.
Since the micro-scale $r_N$ is given by~\eqref{micro-scale}, if $p$ is regular, we have
\begin{gather} \label{rN regular Lap}
r_N= \sqrt{ \frac{2}{N \Delta Q(p)} }(1+o(1)).
\end{gather}
This leads to
\begin{gather*}
\lim_{N \to \infty} \frac{Nr_N^2}{2} \Delta Q( p+r_N z )=1.
\end{gather*}
Therefore if we formally take the large-$N$ limit of Ward's equation \eqref{Ward rescale N} and if $R$ is non-trivial, for a general potential $Q$ and $p$ regular, we arrive at
\begin{gather} \label{Ward}
	\bar{\partial} C(z)=R(z)-1-\dfrac12 \Delta \log R(z)+\dfrac{1}{(z-\bar{z})^2},
\end{gather}
where
\begin{gather*}
C(z):=\int_{ \C }\frac{B(z,w)}{z-w}\, {\rm d}A(w).
\end{gather*}
One may compare Ward's equation \eqref{Ward} for the symplectic ensemble with that for the random normal matrix ensemble~\eqref{Ward cplx}.

\begin{Remark}[limiting Ward's equation at a singular point of Mittag-Leffler type]
As an analogue of \cite{ameur2018random}, we briefly discuss the limiting form of Ward's equation at a singular point of Mittag-Leffler type.
For $\lambda>0$, let $\widetilde{Q}$ be a potential satisfying $\widetilde{Q}(\zeta)=|\zeta|^{2\lambda}+o\big(|\zeta|^{2\lambda}\big)$ as $\zeta \to 0$. Then we consider a potential $Q$ of the form
\begin{gather*}
	Q(\zeta)=\widetilde{Q}(\zeta)-\frac{2c}{N}\log |\zeta|, \qquad c>-1.
\end{gather*}
By \eqref{micro scale ML}, in the sense of distributions, we have
\begin{gather*}
\frac{Nr_N^2}{2} \Delta Q( r_N z )=\lambda |z|^{2\lambda-2}+\frac{c}{2} \delta_0(z)+ o(1),
\end{gather*}
where $\delta_0$ is the Dirac delta at the origin. Taking the limit $N \to \infty$ of the equation \eqref{Ward rescale N}, at least formally we arrive at the distributional Ward's equation
\begin{gather} \label{Ward lim NS}
	\bar{\partial} C(z)=R(z)-\lambda |z|^{2\lambda-2}-\frac{c}{2} \delta_0(z)-\dfrac12 \Delta \log R(z)+\dfrac{1}{(z-\bar{z})^2}.
\end{gather}
For $\lambda \not= 1$ and $c \not = 0$, this form of Ward's equation at a singular point differs from \eqref{Ward}.
Notice that equation \eqref{Ward lim NS} does not depend on the particular choice of $\widetilde{Q}.$
\end{Remark}

For the symplectic ensemble with a general potential $Q$ and for a regular point $p$, suppose that the limiting correlation kernel exists and is of the form \eqref{K regular}.
To obtain the Gaussian factor ${\rm e}^{-|z|^2-|w|^2}$ in \eqref{K regular}, notice the Taylor series expansion:
\begin{gather*}
\frac{N}{2} Q(p+r_Nz) = \frac{N}{2} Q(p)+ \frac{Nr_N}{2} \big( \partial_z Q(p) z +( \bp_z Q )(p) \bar{z} \big)+ \frac{N r_N^2}{4} \big( \partial_z^2 Q (p)z^2 + \bp_z^2 Q (p) \bar{z}^2 \big)
 \\
\hphantom{\frac{N}{2} Q(p+r_Nz) =}{}
 +\frac{N r_N^2}{2} \Delta Q(p) |z|^2+o(1).
\end{gather*}
Then it follows from \eqref{rN regular Lap} and \eqref{bfK Pf} that the second line on the right-hand side contributes to the Gaussian factor, whereas the first line contributes to the pre-kernel.
We refer to \cite[Section~3.5]{MR3975882} for a similar computation.

In the spirit of \cite{MR4030288}, we aim to characterise (horizontally) translation-invariant kernels, that is
those invariant under translations along the real axis. Moving away from the real axis will change the universality class.

It is not difficult to see, that a pre-kernel $\kappa$ of the form
\begin{gather}\label{kappa Psi}
	\kappa(z,w)={\rm e}^{z^2+w^2} \Psi(z-w),
\end{gather}
leads to (horizontally) translation-invariant correlation functions,
when $\Psi$ is some odd function.
%due to the anti-symmetry of the pre-kernel $\kappa$.
Inserting the pre-kernel \eqref{kappa Psi} with correlation kernel \eqref{K regular} into
the $k$-point function $R_k$, and using rules for the Pfaffian determinant leads to
\begin{align*}
R_k(z_1,\dots, z_k) & = \prod_{j=1}^k (\bar{z}_j-z_j) \Pf\left[ {\rm e}^{-|z_j|^2-|z_l|^2} 	\begin{pmatrix}
			{\rm e}^{z_j^2+z_l^2} \Psi(z_j-z_l) & {\rm e}^{z_j^2+\bar{z}_l^2} \Psi(z_j-\bar{z}_l)
			\\
			{\rm e}^{\bar{z}_j^2+z_l^2}\Psi(\bar{z}_j-z_l) & {\rm e}^{\bar{z}_j^2+\bar{z}_l^2 } \Psi(\bar{z}_j-\bar{z}_l)
		\end{pmatrix} \right]_{j,l=1}^k
		\\
		 & = \prod_{j=1}^k (\bar{z}_j-z_j){\rm e}^{(\bar{z}_j-z_j)^2} \Pf\left[ 	\begin{pmatrix}
			\Psi(z_j-z_l) & \Psi(z_j-\bar{z}_l)
			\\
			\Psi(\bar{z}_j-z_l) & \Psi(\bar{z}_j-\bar{z}_l)
		\end{pmatrix} \right]_{j,l=1}^k.
\end{align*}
Therefore it is clear that the horizontal translation invariance of the $k$-point function
$R_k$ holds along the real axis, i.e., for $t \in \R$,
\begin{gather*}
R_k(z_1+t,\dots,z_k+t)=R_k(z_1,\dots,z_k).
\end{gather*}

We write $\Psi$ as Fourier's inversion form
\begin{gather} \label{Psi J hat}
	\Psi(z)=\frac{1}{2\pi} \int_{\R} {\rm e}^{{\rm i}zu} \wh{J}(u)\,{\rm d}u
\end{gather}
for some odd function $\wh{J}.$ Here and in the sequel, let us assume that $\wh{J} \in L^1 \cap L^2$.

Ward's equation of the form \eqref{Ward} will be used to show our final result valid for general potentials.
To be more precise, we have discussed that for a general $C^2$-smooth potential $Q$ and a regular point $p$ in the real bulk, the limiting point process should (at least intuitively) satisfy the mass-one equation~\eqref{mass-one} and Ward's equation~\eqref{Ward}.
The steps that lack in the proof are existence of the limiting correlation kernel, and the one of taking the limit \mbox{$N \to \infty$} of~\eqref{mass-one N} and~\eqref{Ward rescale N}.
These would require a separate analysis, see~\cite{MR3975882} for the random normal matrix ensemble.
As these are beyond the scope of this paper, we so far are able to characterise the translation-invariant solution of the limiting mass-one equation~\eqref{mass-one} and Ward's equation~\eqref{Ward}.

\begin{Theorem}\label{Thm_TIsol}
	A pre-kernel \eqref{kappa Psi} leading to translation invariance satisfies the mass-one equation \eqref{mass-one} and Ward's equation \eqref{Ward} if and only if
	\begin{gather} \label{kappa t.i.}
		\kappa(z,w)= \frac{1}{\sqrt{\pi}} {\rm e}^{ z^2+w^2 } \int_E {\rm e}^{-u^2} \sin(2u(z-w) ) \frac{{\rm d}u}{u},
	\end{gather}
	where $E=(-a,a)$ for some $a >0$.
\end{Theorem}

The overall proof of this theorem is parallel to that of \cite[Theorem~4]{MR4030288}.

\begin{Remark} \noindent
	\begin{enumerate}\itemsep=0pt
\item[(i)] Note that if $a=\infty$, the kernel \eqref{kappa t.i.} corresponds to the symplectic Ginibre kernel in the bulk along the real line \eqref{kappa bulk}.
		On the other hand, if $a<\infty$ is fixed, then it corresponds to the kernel in the almost-Hermitian limit at the origin \cite{MR1928853} after an appropriate rescaling of the eigenvalues.
		(This result has been extended to the entire bulk along the real line in~\cite{byun2021wronskian}.)
		In particular, note that for the density we have
		\begin{gather} \label{R AH}
		R(x+{\rm i}y)=\frac{2}{\sqrt{\pi} {\rm i}} y {\rm e}^{-4y^2} \int_E {\rm e}^{-u^2} \sin(4{\rm i}y u) \frac{{\rm d}u}{u},
		\end{gather}
		see Figure~\ref{Fig. RAH}.
		For random normal matrix ensembles in the almost-Hermitian regime, a way to characterise the precise interval $E$ was presented in a recent work~\cite{AB21}. 	
\item[(ii)] 	Similar to random normal matrix ensembles discussed in \cite{MR3975882,MR4030288}, the translation-invariant scaling limit~\eqref{Psi} enjoys a \emph{Gaussian convolution structure}.
		To be precise, let us write $\gamma(z):=\frac{1}{\sqrt{\pi}} {\rm e}^{-z^2} $ for the Gaussian kernel and define
		\begin{gather*}
			\gamma * \varphi(z):=\int_\R \varphi(t) \gamma(z-t)\,{\rm d}t,
		\end{gather*}
		where $\varphi$ is a suitable tempered distribution on $\R$.
		Then we have
		$\Psi(z)=	\gamma * \varphi(z),$ where~$\varphi$ satisfies
		\begin{gather*}
			\wh{\varphi}(u):=\int_{\R} \varphi(x) {\rm e}^{-{\rm i}u x}\,{\rm d}x =\frac{2\sqrt{\pi}}{{\rm i}} \frac{1}{u}\cdot \mathbbm{1}_{E}(u/2).
		\end{gather*}
	\end{enumerate}
\end{Remark}

\begin{figure}[h!]
	\begin{subfigure}{0.32\textwidth}
		\begin{center}	
			\includegraphics[width=\textwidth]{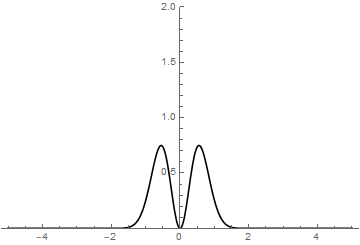}
		\end{center}
		\subcaption{$a=1$}
	\end{subfigure}	
	\begin{subfigure}{0.32\textwidth}
		\begin{center}	
			\includegraphics[width=\textwidth]{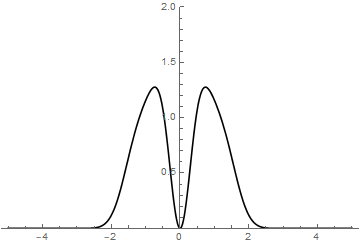}
		\end{center}
		\subcaption{$a=3$}
	\end{subfigure}	
	\begin{subfigure}{0.32\textwidth}
		\begin{center}	
			\includegraphics[width=\textwidth]{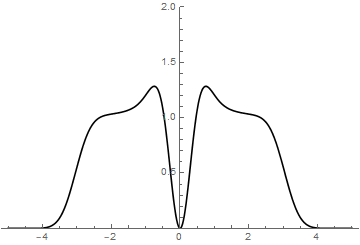}
		\end{center}
		\subcaption{$a=6$}
	\end{subfigure}
	
	\caption{ The plots display graphs of $R(x+{\rm i}y)$ given in \eqref{R AH} for a few values of $a$. The graphs shown are against the $y$-variable as they are invariant under translation in $x$-direction.		
They can be seen to approach Figure~\ref{Fig. RN_Gaussian}(a) and (c) for increasing values of~$a$.} \label{Fig. RAH}
\end{figure}

\section{Scaling limits of the Ginibre ensemble}

In the first part of this section, we review the canonical representation of the correlation kernel in terms of skew-orthogonal polynomials. In the second part, we prove Theorem~\ref{Thm_edge kernel} and Corollary~\ref{Cor_Pf to det}.

\subsection{Skew-orthogonal polynomials}
For a given potential $Q$, the \textit{anti-symmetric scalar product} $\langle \cdot | \cdot \rangle$ is given by
\begin{gather*}
	\langle h|g \rangle:= \int_\C \big( \bar{\zeta}-\zeta \big) {\rm e}^{-N Q(\zeta) } \big( h(\zeta) \overline{g(\zeta)}-\overline{h(\zeta)}g(\zeta) \big) \, {\rm d}A(\zeta).
\end{gather*}
By definition, a family $\{ q_k \}_{k=0}^\infty$ of polynomials is called
\textit{skew-orthogonal polynomials} associated with $Q$ if it satisfies
\begin{gather*}
	\langle q_{2k+1} | q_{2l} \rangle=-\langle q_{2l} | q_{2k+1} \rangle =s_k \, \delta_{kl}, \qquad \langle q_{2k+1} | q_{2l+1} \rangle=\langle q_{2l} | q_{2k} \rangle=0,
\end{gather*}
where $\delta_{kl}$ is the Kronecker delta and $s_k$ is a positive number that depends only on $k$.

For a radially symmetric potential $Q$, let
\begin{gather*}
	h_j:=\int_\C |\zeta|^{2j} {\rm e}^{-N Q(\zeta)} \, {\rm d}A(\zeta).
\end{gather*}
Then by \cite[Theorem 3.1]{akemann2021skew} (see also \cite[p.~7]{MR3066113})
\begin{gather} \label{skew op_rad}
	q_{2k+1}(\zeta)=\zeta^{2k+1}, \qquad q_{2k}(\zeta)=\zeta^{2k}+\sum_{l=0}^{k-1} \zeta^{2l} \prod_{j=0}^{k-l-1} \frac{h_{2l+2j+2} }{ h_{2l+2j+1} }
\end{gather}
form a family of skew-orthogonal polynomials associated with the potential $Q$. Here we have $s_k=2h_{2k+1}.$

Recall that the particle system \eqref{Gibbs} forms a Pfaffian point process, namely, the $k$-point correlation function
\begin{gather} \label{bfRNk def}
	\bfR_{N,k}(\zeta_1,\dots,\zeta_k):=\frac{1}{Z_N} \frac{N!}{(N-k)!} \int_{ \C^{N-k} } {\rm e}^{- \bfH_N( \boldsymbol{\zeta} ) } \prod_{j=k+1}^N {\rm d}A(\zeta_j)
\end{gather}
is expressed as
\begin{gather} \label{bfRNk Pf}
	\bfR_{N,k}(\zeta_1,\dots, \zeta_k)= \prod_{j=1}^{k} \big(\bar{\zeta}_j-\zeta_j\big) \Pf \big[ \bfK_{N}(\zeta_j,\zeta_l) \big]_{ j,l=1 }^k,
\end{gather}
where the $2 \times 2$ matrix-valued kernel $\bfK_N$ is of the form
\begin{gather} \label{bfK Pf}
	\bfK_N(\zeta,\eta)={\rm e}^{ -NQ(\zeta)/2-NQ(\eta)/2 }
	\begin{pmatrix}
		\bfkappa_N(\zeta,\eta) & \bfkappa_N(\zeta,\bar{\eta}) \\
		\bfkappa_N(\bar{\zeta},\eta) & \bfkappa_N(\bar{\zeta},\bar{\eta})
	\end{pmatrix}.
\end{gather}
In particular, we write $\bfR_{N} \equiv \bfR_{N,1}$ for the $1$-point function.
It is well known that the pre-kernel~$\bfkappa_N$ takes the form
\begin{gather} \label{bfkappaN}
	\bfkappa_N(\zeta,\eta)=\sum_{k=0}^{N-1} \frac{q_{2k+1}(\zeta) q_{2k}(\eta) -q_{2k}(\zeta) q_{2k+1}(\eta)}{s_k},
\end{gather}
see \cite{MR1928853}, where the quaternionic determinant $Q\det$ is used instead of the Pfaffian $\Pf$. We also refer to \cite{Mehta} for the relation between $Q\det$ and $\Pf$.

The relation between the pre-kernels $\bfkappa_N$ and $\kappa_N$
for the change of variables \eqref{zeta eta}
at point $p\in S$ is given as
\begin{gather} \label{kappa_N}
	\kappa_N(z,w):=r_N^3
	\begin{cases}
		\bfkappa_N( p+r_Nz, p+r_Nw) &\text{if } p \in \textup{int}(S),
		\\
	\bfkappa_N\big( p+r_Nz {\rm e}^{{\rm i}\theta}, p+r_Nw {\rm e}^{{\rm i}\theta} \big)&\text{if } p \in \partial S.
	\end{cases} 	
\end{gather}
Thus the rescaled process $\boldsymbol{z}$ retains its Pfaffian structure in terms of the rescaled matrix-valued kernel
\begin{gather} \label{K_N}
	K_N(z,w):={\rm e}^{ -\frac{N}{2} (Q( \zeta )+Q( \eta )) } \begin{pmatrix}
		\kappa_N(z,w) & \kappa_N(z,\bar{w}) \\
		\kappa_N(\bar{z},w) & \kappa_N(\bar{z},\bar{w})
	\end{pmatrix},
\end{gather}
cf.~\eqref{zeta eta} for the unscaled variables $\zeta$ and $\eta$ inside the arguments of the potential~$Q$. We remark that in the factor $r_N^3$ in~\eqref{kappa_N}, one factor $r_N$ comes from the term $\prod_{j=1}^k \big(\bar{\zeta}_j-\zeta_j\big)$ in~\eqref{bfRNk Pf} (see~\eqref{RNk Pf}), whereas the other two factors originate from the change of variables for the pre-kernel. This leads precisely to~\eqref{RNk Pf} as follows
	\begin{gather*}
		 R_{N,k}(z_1,\dots, z_k)
		 = 	r_N^{2k} \prod_{j=1}^{k} \big(\bar{\zeta}_j-\zeta_j\big) \Pf \big[ \bfK_{N}(\zeta_j,\zeta_l) \big]_{ j,l=1 }^k
		= \prod_{j=1}^{k} \frac{\bar{\zeta}_j-\zeta_j}{r_N} \Pf \big[ K_{N}(z_j,z_l) \big]_{ j,l=1 }^k.
	\end{gather*}

\subsection{Boundary Ginibre point processes}\label{Subsec_bdy Ginibre}
For the Gaussian case $Q(\zeta)=|\zeta|^2$, by \eqref{skew op_rad}, we have
\begin{gather*}
	q_{2k+1}(\zeta)=\zeta^{2k+1}, \qquad q_{2k}(\zeta)=\frac{(2k)!!}{N^k} \sum_{l=0}^k \frac{N^l}{(2l)!!} \zeta^{2l}, \qquad s_k=2\frac{(2k+1)!}{N^{2k+2}}.
\end{gather*}
Therefore it follows from \eqref{bfkappaN} that the associated kernel $\bfkappa_N$ has an expression
\begin{gather}\label{bfkappa_N Gaussian}
	\bfkappa_N(\zeta,\eta):=\boldsymbol{G}_N(\zeta,\eta)-\boldsymbol{G}_N(\eta,\zeta),
\end{gather}
where
\begin{gather*}
\boldsymbol{G}_N(\zeta,\eta):=\frac{N\sqrt{N}}{2} \sum_{k=0}^{N-1} \frac{ \big(\sqrt{N}\zeta\big)^{2k+1}}{(2k+1)!!} \sum_{l=0}^k \frac{ (\sqrt{N}\eta) ^{2l}}{(2l)!!}.
\end{gather*}
This recovers the expression obtained in Mehta's book \cite{Mehta} in a different approach, compared to the skew-orthogonal polynomials we use here from~\cite{MR1928853}.

By definition, a function $c(z,w)$ is called a \textit{cocycle} if there exists a (continuous) unimodular function $h$ (i.e., $h(z)=h(1/\bar{z})$) such that $c(z,w) = h(z)h(w).$ Since cocycles cancel out when multiplying the pre-kernel and taking the Pfaffian, if $(c_N)_{N=1}^\infty$ is a sequence of cocycles, then~$c_N\kappa_N$ is an equivalent realisation of the pre-kernel~$\kappa_N$
for the same point process $\boldsymbol{z}$. Let us denote by~$c_N \cdot K_N$ the correlation kernel associated with $c_N \kappa_N$, i.e.,
\begin{gather*}
c_N \cdot K_N(z,w):= {\rm e}^{ -\frac{N}{2} (Q( \zeta )+Q( \eta )) } \begin{pmatrix}
	 c_N(z,w)\, \kappa_N(z,w) & c_N(z,\bar{w})\, \kappa_N(z,\bar{w}) \\
	c_N(\bar{z},w)\, \kappa_N(\bar{z},w) & c_N(\bar{z},\bar{w}) \, \kappa_N(\bar{z},\bar{w})
	\end{pmatrix}.
\end{gather*}

Now we prove Theorem~\ref{Thm_edge kernel}.

\begin{proof}[Proof of Theorem~\ref{Thm_edge kernel}]
	We prove the theorem for $p =\sqrt{2}$ only. The other case $p=-\sqrt{2}$ can be proved in the same way.
	Note that for $Q=|\zeta|^2,$ we have $r_N=\sqrt{2/N}$, see~\eqref{micro-scale}.
	
	By \eqref{K_N}, we have
	\begin{gather*}
		K_N(z,w)={\rm e}^{ - |\sqrt{N}+z|^2-|\sqrt{N}+w|^2 } \begin{pmatrix}
			\kappa_N(z,w) & \kappa_N(z,\bar{w}) \\
			\kappa_N(\bar{z},w) & \kappa_N(\bar{z},\bar{w})
		\end{pmatrix}.
	\end{gather*}
	Here it follows from \eqref{kappa_N} and \eqref{bfkappa_N Gaussian} that
	\begin{gather*}
	\kappa_N(z,w)=G_N(z,w)-G_N(w,z),
	\end{gather*}
	where
	\begin{gather} \label{kappa tilde}
	G_N(z,w)=\sqrt{2} \sum_{k=0}^{N-1} \frac{\big( \sqrt{2N}+\sqrt{2}z \big)^{2k+1}}{(2k+1)!!} \sum_{l=0}^k \frac{\big( \sqrt{2N}+\sqrt{2} w\big)^{2l} }{(2l)!!}.
	\end{gather}
	
	Let us write
	\begin{gather} \label{kappa hat, tilde}
		\widehat{\kappa}_N(z,w):= {\rm e}^{ -2( \sqrt{N}+ z)( \sqrt{N}+ w) } \kappa_N(z,w).
	\end{gather}
	Then we have
	\begin{gather*}
		K_N(z,w)={\rm e}^{ - |z|^2-|w|^2 } \begin{pmatrix}
			c_N(z,w)\widehat{\kappa}_N(z,w) {\rm e}^{2zw} & c_N(z,\bar{w}) \widehat{\kappa}_N(z,\bar{w}) {\rm e}^{2z\bar{w}} \\
			c_N(\bar{z},w)\widehat{\kappa}_N(\bar{z},w) {\rm e}^{2\bar{z}w} & c_N(\bar{z},\bar{w}) \widehat{\kappa}_N(\bar{z},\bar{w}) {\rm e}^{2\bar{z}\bar{w}}
		\end{pmatrix},
	\end{gather*}
	where the cocycle $c_N$ is given by
	$c_N(z,w):={\rm e}^{ 2{\rm i}\sqrt{N}\im(z+w) }$.
	
	We claim that $\widehat{\kappa}:=\lim\limits_{N \to \infty} \widehat{\kappa}_N$ satisfies the differential equation
	\begin{gather} \label{kappa hat pde 1 edge}
		\partial_z \widehat{\kappa}(z,w)=2(z-w) \widehat{\kappa}(z,w)+\erfc(z+w)-\frac{ {\rm e}^{ (z-w)^2-2z^2 } }{\sqrt{2}} \erfc\big(\sqrt{2}w\big).
	\end{gather}
	Here the uniform convergence of $\widehat{\kappa}_N$ on compact subsets of $\C$ follows from the Weierstrass $M$-test, see \cite[Section 4]{akemann2021skew} for similar computations.
	
	Differentiating \eqref{kappa tilde}, we have
	\begin{gather*}
		\partial_z G_N(z,w)= 2 \sum_{k=1}^{N-1} \frac{\big( \sqrt{2N}+ \sqrt{2}z \big)^{2k}}{(2k-1)!!} \sum_{l=0}^{k-1} \frac{\big( \sqrt{2N}+\sqrt{2} w\big)^{2l} }{(2l)!!}\\
\hphantom{\partial_z G_N(z,w)=}{}
+ 2 \sum_{k=0}^{N-1} \frac{\big( \sqrt{2N}+ \sqrt{2}z \big)^{2k}}{(2k-1)!!} \frac{\big( \sqrt{2N}+ \sqrt{2} w\big)^{2k} }{(2k)!!}.
	\end{gather*}
	Rearranging the terms, we rewrite this equation as
	\begin{gather*}
 \partial_z G_N(z,w)
		= 2 \big(\sqrt{2N}+\sqrt{2}z\big) \sum_{k=0}^{N-1} \frac{\big( \sqrt{2N}+ \sqrt{2}z \big)^{2k+1}}{(2k+1)!!} \sum_{l=0}^{k} \frac{\big( \sqrt{2N}+\sqrt{2} w\big)^{2l} }{(2l)!!}
		\\
{}+ 2 \sum_{k=0}^{N-1} \frac{\big( \sqrt{2N}+ \sqrt{2}z \big)^{2k}}{(2k-1)!!} \frac{\big( \sqrt{2N}+ \sqrt{2} w\big)^{2k} }{(2k)!!}
		-2 \frac{\big( \sqrt{2N}+ \sqrt{2}z \big)^{2N}}{(2N-1)!!} \sum_{l=0}^{N-1} \frac{\big( \sqrt{2N}+\sqrt{2} w\big)^{2l} }{(2l)!!}.
	\end{gather*}
	Similarly,
	\begin{gather*}
		\partial_z G_N(w,z)= 2 \big( \sqrt{2N}+\sqrt{2} z\big) \sum_{k=0}^{N-1} \frac{\big( \sqrt{2N}+\sqrt{2} w\big)^{2k+1}}{(2k+1)!!} \sum_{l=0}^{k} \frac{\big( \sqrt{2N}+\sqrt{2} z\big)^{2l} }{(2l)!!}
		\\
\hphantom{\partial_z G_N(w,z)=}{}
- 2 \big( \sqrt{2N}+\sqrt{2} z\big) \sum_{k=0}^{N-1} \frac{\big( \sqrt{2N}+\sqrt{2} w\big)^{2k+1}}{(2k+1)!!} \frac{\big( \sqrt{2N}+\sqrt{2} z\big)^{2k} }{(2k)!!}.
	\end{gather*}
	
Now let us recall the definition of the (regularised) incomplete Gamma function \cite[Chapter~8]{olver2010nist}
\begin{gather*}
Q(n,z)= \sum_{k=0}^{n-1} \frac{z^k {\rm e}^{-z}}{k!}, \qquad n=1,2,\dots.
\end{gather*}
	Combining all of the above equations with \eqref{kappa hat, tilde}, we obtain
\begin{gather}
\partial_z \widehat{\kappa}_N(z,w)= 2(z-w) \widehat{\kappa}_N(z,w) +2 Q(2N,\lambda)\nonumber\\
\hphantom{\partial_z \widehat{\kappa}_N(z,w)=}{}
-2 {\rm e}^{ (z-w)^2 } {\rm e}^{ -( \sqrt{N}+ z)^2 } \frac{\big( \sqrt{2N}+ \sqrt{2}z \big)^{2N}}{(2N-1)!!} Q\big(N,\widetilde{\lambda}\big),\label{kappa hat pde 1-0}
\end{gather}
	where
\begin{gather*}
		\lambda:=2\big( \sqrt{N}+ z \big) \big( \sqrt{N}+ w\big),\qquad \widetilde{\lambda}:=\big( \sqrt{N}+ w\big)^2.
\end{gather*}
	
We now derive \eqref{kappa hat pde 1 edge} from \eqref{kappa hat pde 1-0}. By combining Stirling's formula with the elementary asymptotic
	${\rm e}^{ -2\sqrt{N} z } \big( 1+\tfrac{z}{\sqrt{N}} \big)^{2N} = {\rm e}^{ -z^2 } (1+o(1))$,
	we have
	\begin{gather} \label{kappa hat pde 1-1}
		{\rm e}^{ -( \sqrt{N}+ z)^2 } \frac{\big( \sqrt{2N}+ \sqrt{2}z \big)^{2N}}{(2N-1)!!}
		= \frac{{\rm e}^{-2z^2}}{\sqrt{2}} (1+o(1)),
	\end{gather}
 where the $o(1)$-terms are uniform on compact subsets of~$\C$.

	On the other hand, the incomplete Gamma function satisfies the asymptotic
	\begin{gather}\label{Q asym}
		Q\big(a,a+\sqrt{2a}z\big) = \tfrac12 \erfc(z) (1+o(1)), \qquad a\to \infty,
	\end{gather}
	see \cite[equation~(8.11.10)]{olver2010nist}. Substituting \eqref{kappa hat pde 1-1}, \eqref{Q asym} to the equation~\eqref{kappa hat pde 1-0}, and taking the limit as $N \to \infty$, we obtain~\eqref{kappa hat pde 1 edge}.
	
	Let us define
	\begin{gather*}
	f(z,w):=\frac{ 1 }{\sqrt{2}} \int_{-\infty}^0 {\rm e}^{-2(z-u)^2 } \erfc\big( \sqrt{2}(w-u) \big)- {\rm e}^{-2(w-u)^2 } \erfc\big( \sqrt{2}(z-u) \big) \, {\rm d}u.
	\end{gather*}
	Then integration by parts gives that
	\begin{gather*}
		\partial_z f(z,w)=\frac{4}{\sqrt{\pi}}\int_{-\infty}^0 {\rm e}^{-2(z-u)^2-2(w-u)^2}\,{\rm d}u-\frac{1}{\sqrt{2}}\big[ {\rm e}^{-2(z-u)^2} \erfc\big(\sqrt{2}(w-u)\big)\big]_{-\infty}^0,
	\end{gather*}
	which leads to
	\begin{align*}
		\begin{split}
			\partial_z f(z,w)=
			{\rm e}^{-(z-w)^2} \erfc(z+w)-\frac{ {\rm e}^{-2z^2 } }{ \sqrt{2} }\erfc\big(\sqrt{2}w\big).
		\end{split}
	\end{align*}
	Therefore $\widehat{\kappa}(z,w):={\rm e}^{(z-w)^2}f(z,w)$ is an anti-symmetric solution to \eqref{kappa hat pde 1 edge}.
	Since \eqref{kappa hat pde 1 edge} is a~first-order differential equation, the anti-symmetry plays the role of the initial value, which determines the solution uniquely.
	This completes the proof.
\end{proof}

We now prove Corollary~\ref{Cor_Pf to det}.

\begin{proof}[Proof of Corollary~\ref{Cor_Pf to det}]
	By \eqref{rescaling}, for $p_N={\rm e}^{i\theta_N}$ with $\theta_N=\frac{t}{\sqrt{N}}$, we have
	\begin{gather*}
		\zeta_j=\sqrt{2}+\frac{\sqrt{2}(z_j+{\rm i}t)}{\sqrt{N}}+O\left(\frac1N\right), \qquad 	\frac{\bar{\zeta}_j-\zeta_j}{r_N}	= \bar{z}_j-z_j-2{\rm i}t+O\left(\frac{1}{\sqrt{N}}\right).
	\end{gather*}
	Therefore the first assertion \eqref{R_k^t Pf} follows along the same lines of the proof of Theorem~\ref{Thm_edge kernel} by replacing $(z,w) \mapsto (z+{\rm i}t,w+{\rm i}t)$.
	
	We now prove \eqref{Pf to det}. By \eqref{R_k^t Pf}, we have
\begin{gather} \label{R_k^t Pf t asym}
		R_k^t(z_1,\dots, z_k) \sim \Pf \big[ {-}2{\rm i}tK_{{\rm edge}}^t(z,w) \big]_{ j,l=1 }^k
\end{gather}
	as $t \to \infty$.
 Here $f(x)\sim g(x)$ means that $f(x)/g(x)$ tends to $1$ as $x\to \infty$.
	Using the following asymptotic of the error function (see, e.g., \cite[equation~(7.12.1)]{olver2010nist})
	\begin{gather*}
	\erfc(z) \sim \frac{{\rm e}^{-z^2}}{\sqrt{\pi}\,z}, \qquad z\to\infty,
	\end{gather*}
	we have that as $t \to \infty,$
	\begin{align*}
		{\rm e}^{-|z+{\rm i}t|^2-|w+{\rm i}t|^2} \kappa_{{\rm edge}}^\R (z+{\rm i}t,\bar{w}-{\rm i}t)
		& \sim \frac{{\rm i}}{\sqrt{\pi}t} {\rm e}^{-(z+{\rm i}t)(z+\bar{z})-(\bar{w}-{\rm i}t)(w+\bar{w})} \int_{-\infty}^0 {\rm e}^{ -4u^2+4u(z+\bar{w}) } \, {\rm d}u
		\\
		&= \frac{{\rm i}}{4t} {\rm e}^{ -|z|^2-|w|^2+2z\bar{w} } \erfc(z+\bar{w}) {\rm e}^{ -(z+\bar{z}){\rm i}t+(w+\bar{w}){\rm i}t }.
	\end{align*}
	Therefore we obtain that as $t \to \infty$,
	\begin{gather} \label{kappa t lim 1}
		{\rm e}^{-|z+{\rm i}t|^2-|w+{\rm i}t|^2} \kappa_{{\rm edge}}^\R (z+{\rm i}t,\bar{w}-{\rm i}t) \sim \frac{{\rm i}}{2t}\widetilde{K}(z,w),
	\end{gather}
	where
	\begin{gather} \label{K tilde K edge C}
		\widetilde{K}(z,w):=c(z,w)\cdot K_{{\rm edge}}^\C(z,w) .
	\end{gather}
	Here $c(z,w)={\rm e}^{ -(z+\bar{z}){\rm i}t+(w+\bar{w}){\rm i}t }$ is a cocycle for the determinantal point process.
	On the other hand, it follows from similar computations that as $t \to \infty,$
	\begin{gather} \label{kappa t lim 2}
		{\rm e}^{-|z+{\rm i}t|^2-|w+{\rm i}t|^2} \kappa_{{\rm edge}}^\R (z+{\rm i}t,w+{\rm i}t)=O\left(\frac1{t^2}\right).
	\end{gather}
	
	Combining \eqref{R_k^t Pf t asym}, \eqref{kappa t lim 1} and \eqref{kappa t lim 2}, we have
	\begin{gather*}
		R_{k}^t(z_1,\dots,z_k) \sim \Pf
	\begin{pmatrix}
		0 & \widetilde{K}(z_1,z_1) & 0 & \widetilde{K}(z_1,z_2) & \dots
		\\
		-\widetilde{K}(z_1,z_1) &0 & -\widetilde{K}(z_2,z_1) & 0 & \dots
		\\
		0 & \widetilde{K}(z_2,z_1) & 0 & \widetilde{K}(z_2,z_2) & \dots
		\\
		-\widetilde{K}(z_1,z_2) &0 & -\widetilde{K}(z_2,z_2) & 0 & \dots
		\\
		\vdots & \vdots & \vdots & \vdots &\ddots
	\end{pmatrix}_{2k \times 2k}
	\end{gather*}
	as $t \to \infty.$ The remaining argument is similar to that used in \cite{akemann2019universal}. By a proper reordering, we have
	\begin{gather*}
		R_{k}^t(z_1,\dots,z_k) \sim (-1)^{\frac{k(k-1)}{2}} \Pf
	\begin{pmatrix}
		0 & M
		\\
		-M^{\rm T} & 0
	\end{pmatrix}, \qquad M=\big( \widetilde{K}(z_j,z_l) \big)_{j,l=1}^k.
	\end{gather*}
	Now it follows from the identity
	\begin{gather*}
	\Pf\begin{pmatrix}
		0 & M
		\\
		-M^{\rm T} & 0
	\end{pmatrix}=(-1)^{ \frac{k(k-1)}{2}} \det(M)
	\end{gather*}
	and \eqref{K tilde K edge C} that
	\begin{gather*}
		R_{k}^t(z_1,\dots,z_k) \sim \det \big[ \widetilde{K}(z_j,z_l) \big]_{j,l=1}^k= \det \big[ K_{{\rm edge}}^\C(z_j,z_l) \big]_{j,l=1}^k.
	\end{gather*}
	This completes the proof.
\end{proof}

\section{Scaling limit of the Mittag-Leffler ensemble}

In this section, we construct a fractional differential equation satisfied by the local kernel of the Mittag-Leffler ensemble with potentials $Q(\zeta)=|\zeta|^{2\lambda}-(2c/N)\log|\zeta|$, $\lambda>0$ and $c>-1$ (Proposition~\ref{Thm_bulkS}). As a consequence, we prove Theorem~\ref{Thm_ML integer}.

\subsection{Christoffel--Darboux type formula}
For the Mittag-Leffler potentials \eqref{Q ML}, the orthogonal norm $h_k$ is given by
\begin{gather*}
h_k=\int_\C |\zeta|^{2k} {\rm e}^{-N Q(\zeta)} \, {\rm d}A(\zeta)=\frac{\Gamma\big( \frac{k+c+1}{\lambda} \big)}{\lambda N^{\frac{k+c+1}{\lambda}} } .
\end{gather*}
Then by \eqref{skew op_rad}, we have
\begin{gather} \label{SOP ML}
	q_{2k+1}(\zeta)=\zeta^{2k+1}, \qquad q_{2k}(\zeta)=\zeta^{2k}+ \sum_{l=0}^{k-1} \frac{\zeta^{2l}}{N^{\frac{k-l}{\lambda}}} \prod_{j=0}^{k-l-1} \frac{ \Gamma\big( \frac{2l+2j+3+c}{\lambda} \big) }{ \Gamma\big( \frac{2l+2j+2+c}{\lambda} \big) } ,
\end{gather}
and
\begin{gather} \label{sk ML}
	s_k=2h_{2k+1}=\frac{2}{\lambda } \frac{\Gamma( \frac{2k+c+2}{\lambda} )}{N^{\frac{2k+c+2}{\lambda}} } .
\end{gather}
To describe the Christoffel--Darboux type formula, let us recall that the Caputo fractional derivative $D^\nu_z$ is given by
\begin{gather*}
	D^\nu_z f(z):=\frac{1}{\Gamma(n-\nu)} \int_{0}^{z} (z-u)^{n-\nu-1} f^{(n)}(u) \, {\rm d}u, \qquad n-1 < \nu <n, \quad n \in \mathbb{N},
\end{gather*}
see \cite[Section 2.4]{MR2218073}. Here $f^{(n)}$ is the usual $n$:th derivative.

Recall that $p=0$
in the rescaling \eqref{zeta eta} here,
and $K_{\lambda,c}$ is given by \eqref{K ML}.

\begin{Proposition} \label{Thm_bulkS}
	For each $\lambda>0$ and $c>-1$, there exists a sequence of cocycles $(c_N)_{N=1}^\infty$ such that
	\begin{gather*}
			\lim_{N \to \infty} c_N\cdot K_N(z,w)=K_{\lambda,c}(z,w)
			={\rm e}^{ -\frac{ |z|^{2\lambda}+|w|^{2\lambda} }{\lambda} } \begin{pmatrix}
				\kappa_{\lambda,c}(z,w) & \kappa_{\lambda,c}(z,\bar{w}) \\
				\kappa_{\lambda,c}(\bar{z},w) & \kappa_{\lambda,c}(\bar{z},\bar{w})
			\end{pmatrix}
		\end{gather*}
	uniformly for $z$, $w$ in compact subsets of~$\C$, where
	\begin{gather*}
		\kappa_{\lambda,c}(z,w)=\left( \frac{2}{\lambda} \right)^{ \frac{1}{2\lambda} } (zw)^{\lambda-1} \widetilde{\kappa}_{\lambda,c} \left( \sqrt{ \frac{2}{\lambda} } z^\lambda, \sqrt{ \frac{2}{\lambda} } w^\lambda \right).
	\end{gather*}
	Here $\widetilde{\kappa}_{\lambda,c}$ is given in terms of a function $\widetilde{G}$ as
	\begin{gather*}
	\widetilde{\kappa}_{\lambda,c}(z,w)=\widetilde{G}(z,w)-\widetilde{G}(w,z)
	\end{gather*}
 and the following fractional differential equations hold:
	\begin{gather} \label{G zw FDE}
		D_z^{ \frac{1}{\lambda} } \widetilde{G}(z,w)= z^{ \frac{1}{\lambda} } \widetilde{G}(z,w)+(zw)^{\frac{1+c}{\lambda}-1} E_{\frac{2}{\lambda} \frac{1+c}{\lambda} } \big( (zw)^{\frac{2}{\lambda}} \big),
\\
			D_z^{ \frac{1}{\lambda} } \widetilde{\kappa}_{\lambda,c}(z,w)
			= z^{ \frac{1}{\lambda} } \widetilde{\kappa}_{\lambda,c}(z,w)+ (zw)^{ \frac{1+c}{\lambda}-1 } E_{ \frac{1}{\lambda},\frac{1+c}{\lambda} } \big( (zw)^{\frac{1}{\lambda}} \big)
			\nonumber\\
\hphantom{D_z^{ \frac{1}{\lambda} } \widetilde{\kappa}_{\lambda,c}(z,w)=}{}
 -\frac{ \Gamma\big(\frac{1+c}{\lambda}\big) }{\Gamma\big( \frac{c}{\lambda}\big) \Gamma\big( \frac{2+c}{\lambda}\big) } (zw)^{ \frac{1+c}{\lambda}-1 }(w/z)^{\frac{1}{\lambda}} E_{\frac{1}{\lambda},2,3+c-\lambda}\big(w^{\frac{2}{\lambda}}\big) .\label{FDE_bulkS}
		\end{gather}
\end{Proposition}

We call the equation \eqref{FDE_bulkS} (generalised) Christoffel--Darboux formula.
Such a differential equation satisfied by the correlation kernel, that makes the asymptotic analysis possible, is broadly called Christoffel--Darboux type formula, see, e.g., \cite{MR1917675,byun2021lemniscate,MR3450566}.

We also remark that the inhomogeneous term $(zw)^{ \frac{1+c}{\lambda}-1 } E_{ \frac{1}{\lambda},\frac{1+c}{\lambda} } \big( (zw)^{\frac{1}{\lambda}} \big)$ corresponds to the kernel of the complex Mittag-Leffler ensemble, see~\cite{ameur2018random}.
Such a relation has been observed for other models as well, which include the elliptic Ginibre ensemble~\cite{BE} and its chiral counter part~\cite{MR2180006}.
See also~\cite{MR1762659,MR1675356} for a similar relation for Hermitian ensembles, which gives an expression of the kernel of the symplectic ensemble in terms of a small number of orthogonal polynomials.

\begin{proof}[Proof of Proposition~\ref{Thm_bulkS}]
	Let us define
	\begin{gather*} %\label{kappa tilde0 bS}
		\widetilde{G}_{N}(z,w):=\sum_{k=0}^{N-1}z^{ \frac{2k+2+c}{\lambda}-1 } \prod_{j=0}^k \frac{ \Gamma\big( \frac{2j+1+c}{\lambda} \big) }{\Gamma\big( \frac{2j+2+c}{\lambda} \big)}
		\sum_{l=0}^{k} \frac{ w^{ \frac{2l+1+c}{\lambda}-1 } }{\Gamma\big( \frac{2l+1+c}{\lambda}\big)} \prod_{j=1}^l \frac{\Gamma\big( \frac{2j+c}{\lambda} \big)}{ \Gamma\big(\frac{2j-1+c}{\lambda}\big) } .
	\end{gather*}
	Here and henceforth we use the convention that if $l=0$, then $\prod_{j=1}^l \varphi=1$.
	By \eqref{SOP ML}, \eqref{sk ML} and~\eqref{micro scale ML}, the kernel $K_N$ is given by
	\begin{gather*}
			K_N(z,w)
			={\rm e}^{ -\frac{ |z|^{2\lambda}+|w|^{2\lambda} }{\lambda} } \begin{pmatrix}
				\kappa_N(z,w) & \kappa_N(z,\bar{w}) \\
				\kappa_N(\bar{z},w) & \kappa_N(\bar{z},\bar{w})
			\end{pmatrix},
		\end{gather*}
	where
	\begin{gather*}
		\kappa_N(z,w):=G_N(z,w)-G_N(w,z), \\ G_N(z,w):=\left( \frac{2}{\lambda} \right)^{\frac{1}{2\lambda}} (zw)^{\lambda-1} \widetilde{G}_N \left( \sqrt{ \frac{2}{\lambda} } z^\lambda, \sqrt{ \frac{2}{\lambda} } w^\lambda \right).
	\end{gather*}
	
	By definition of the Caputo derivative, we have that if $k>
	\lceil \nu \rceil$
	\begin{gather} \label{Caputo monomial}
		D_z^\nu z^k =\frac{\Gamma(k+1)}{\Gamma(k-\nu+1)}z^{k-\nu}
	\end{gather}
	and vanishes otherwise.
	Let us write $\widetilde{G}:=\lim\limits_{N \to \infty} \widetilde{G}_N,$ where the convergence is uniform on compact subsets of $\C$. The existence of the limit follows from the Weierstrass $M$-test.
	By direct computations using \eqref{Caputo monomial}, we have
	\begin{gather}
\widetilde{G}(z,w)-\frac{1}{\Gamma\big(\frac{2+c}{\lambda}\big)} z^{ \frac{2+c}{\lambda}-1 } w^{ \frac{1+c}{\lambda}-1 }
			\nonumber\\
\qquad{} =\sum_{k=1}^{\infty} z^{ \frac{2k+2+c}{\lambda}-1 } \prod_{j=0}^k \frac{ \Gamma\big( \frac{2j+1+c}{\lambda} \big) }{\Gamma\big( \frac{2j+2+c}{\lambda} \big)}
			\sum_{l=0}^{k} \frac{ w^{ \frac{2l+1+c}{\lambda}-1 } }{\Gamma\big( \frac{2l+1+c}{\lambda}\big)} \prod_{j=1}^l \frac{\Gamma\big( \frac{2j+c}{\lambda} \big)}{ \Gamma\big(\frac{2j-1+c}{\lambda}\big) } .\label{G zw til lam c 0}
		\end{gather}
	Note in particular that as $z\to 0,$
	\begin{gather} \label{G zw til asym}
		\widetilde{G}(z,w) \sim \frac{w^{ \frac{1+c}{\lambda}-1 } }{\Gamma\big(\frac{2+c}{\lambda}\big)} z^{ \frac{2+c}{\lambda}-1 }.
	\end{gather}
	Applying the operator $D_z^{ \frac{1}{\lambda} }$ to the identity \eqref{G zw til lam c 0}, we obtain
	\begin{gather*}
 D_z^{ \frac{1}{\lambda} } \widetilde{G}(z,w)-\frac{ 1}{\Gamma\big( \frac{1+c}{\lambda} \big)}(zw)^{ \frac{1+c}{\lambda}-1 }
		\\
\qquad{} = \sum_{k=1}^{\infty} z^{ \frac{2k+1+c}{\lambda}-1 } \prod_{j=1}^{k}\frac{ \Gamma \big( \frac{2j-1+c}{\lambda} \big) }{\Gamma\big( \frac{2j+c}{\lambda} \big)} \sum_{l=0}^{k} \frac{ w^{ \frac{2l+1+c}{\lambda}-1 } }{\Gamma\big( \frac{2l+1+c}{\lambda} \big)} \prod_{j=1}^l \frac{\Gamma\big(\frac{2j+c}{\lambda}\big) }{ \Gamma\big( \frac{2j-1+c}{\lambda} \big) }
		\\
\qquad{} = z^{ \frac{1}{\lambda} } \sum_{k=0}^{\infty} z^{ \frac{2k+2+c}{\lambda}-1 } \prod_{j=0}^k \frac{ \Gamma \big( \frac{2j+1+c}{\lambda} \big) }{\Gamma\big( \frac{2j+2+c}{\lambda} \big)} \sum_{l=0}^{k+1} \frac{ w^{ \frac{2l+1+c}{\lambda}-1 } }{\Gamma\big( \frac{2l+1+c}{\lambda} \big)} \prod_{j=1}^l \frac{\Gamma\big(\frac{2j+c}{\lambda}\big) }{ \Gamma\big( \frac{2j-1+c}{\lambda} \big) } ,
	\end{gather*}
	which leads to
	\begin{gather*}
		D_z^{ \frac{1}{\lambda} } \widetilde{G}(z,w)= z^{ \frac{1}{\lambda} } \widetilde{G}(z,w)
		+\sum_{k=0}^{\infty} \frac{ (zw)^{ \frac{2k+1+c}{\lambda}-1 } }{\Gamma\big( \frac{2k+1+c}{\lambda} \big)}.
	\end{gather*}
	Now \eqref{G zw FDE} follows from \eqref{ML function}.
	
	Similarly, by rearranging the terms, we have
	\begin{gather}
 \widetilde{G}(w,z)-\frac{ 1 }{\Gamma\big( \frac{1+c}{\lambda}\big)} z^{ \frac{1+c}{\lambda}-1 } \sum_{k=0}^{\infty} w^{ \frac{2k+2+c}{\lambda}-1 } \prod_{j=0}^k \frac{ \Gamma\big( \frac{2j+1+c}{\lambda} \big) }{\Gamma\big( \frac{2j+2+c}{\lambda} \big)}
			\nonumber\\
\qquad{} =\sum_{k=1}^{\infty} w^{ \frac{2k+2+c}{\lambda}-1 } \prod_{j=0}^k \frac{ \Gamma\big( \frac{2j+1+c}{\lambda} \big) }{\Gamma\big( \frac{2j+2+c}{\lambda} \big)} \sum_{l=1}^{k} \frac{ z^{ \frac{2l+1+c}{\lambda}-1 } }{\Gamma\big( \frac{2l+1+c}{\lambda} \big)} \prod_{j=1}^l \frac{\Gamma\big(\frac{2j+c}{\lambda}\big) }{ \Gamma\big( \frac{2j-1+c}{\lambda} \big) } .\label{G wz til lam c 0}
		\end{gather}
	Then by \eqref{ML3}, we have that as $z\to 0$,
	\begin{gather*} %\label{G wz til asym}
		\widetilde{G}(w,z)\sim \frac{1 }{ \Gamma\big( \frac{2+c}{\lambda}\big) } w^{ \frac{2+c}{\lambda}-1 } E_{\frac{1}{\lambda},2,3+c-\lambda}\big(w^{\frac{2}{\lambda}}\big) z^{ \frac{1+c}{\lambda}-1 }
	\end{gather*}
	Again, by applying the operator $D_z^{ \frac{1}{\lambda} }$ to \eqref{G wz til lam c 0}, we obtain
	\begin{gather*}
D_z^{ \frac{1}{\lambda} } \widetilde{G}(w,z)-\frac{ \Gamma\big(\frac{1+c}{\lambda}\big) }{\Gamma\big( \frac{c}{\lambda}\big) \Gamma\big( \frac{2+c}{\lambda}\big) } (zw)^{ \frac{1+c}{\lambda}-1 }(w/z)^{\frac{1}{\lambda}} E_{\frac{1}{\lambda},2,3+c-\lambda}\big(w^{\frac{2}{\lambda}}\big)\\
\qquad{} = z^{ \frac{1}{\lambda} } \sum_{k=1}^{\infty} w^{ \frac{2k+2+c}{\lambda}-1 } \prod_{j=0}^k \frac{ \Gamma\big( \frac{2j+1+c}{\lambda} \big) }{\Gamma\big( \frac{2j+2+c}{\lambda} \big)} \sum_{l=0}^{k-1} \frac{ z^{ \frac{2l+1+c}{\lambda}-1 } }{\Gamma\big( \frac{2l+1+c}{\lambda} \big)} \prod_{j=1}^l \frac{\Gamma\big(\frac{2j+c}{\lambda}\big) }{ \Gamma\big( \frac{2j-1+c}{\lambda} \big) }
		\\
\qquad{} = z^{ \frac{1}{\lambda} } \widetilde{G}(w,z)- \sum_{k=0}^{\infty} \frac{(zw)^{ \frac{2k+2+c}{\lambda}-1 }}{\Gamma\big( \frac{2k+2+c}{\lambda} \big)}.
	\end{gather*}
	Then by \eqref{ML function}, we have
	\begin{gather}
D_z^{ \frac{1}{\lambda} } \widetilde{G}(w,z)= z^{ \frac{1}{\lambda} } \widetilde{G}(w,z) -(zw)^{\frac{2+c}{\lambda}-1} E_{\frac{2}{\lambda},\frac{2+c}{\lambda}} \big((zw)^{\frac{2}{\lambda}}\big) \nonumber\\
\hphantom{D_z^{ \frac{1}{\lambda} } \widetilde{G}(w,z)=}{}
+\frac{ \Gamma(\frac{1+c}{\lambda}) }{\Gamma( \frac{c}{\lambda}) \Gamma( \frac{2+c}{\lambda}) } (zw)^{ \frac{1+c}{\lambda}-1 }(w/z)^{\frac{1}{\lambda}} E_{\frac{1}{\lambda},2,3+c-\lambda}\big(w^{\frac{2}{\lambda}}\big) . \label{G wz FDE}
		\end{gather}
	Now \eqref{FDE_bulkS} follows from \eqref{G zw FDE}, \eqref{G wz FDE} and \eqref{ML function}.
\end{proof}

\subsection[Bulk singularities of the order lambda=1/m with m in N]{Bulk singularities of the order $\boldsymbol{\lambda=\frac1m}$ with $\boldsymbol{m\in \mathbb{N}}$}

A point $p$ in which the eigenvalue density $\frac12\Delta Q(p)$ vanishes or diverges is called bulk singularity, see \cite{ameur2018random} and references therein. In this subsection, we shall consider the case $\lambda=\frac1m$, where $m$ is a positive integer and prove Theorem~\ref{Thm_ML integer} for general $c>-1$.

\begin{proof}[Proof of Theorem~\ref{Thm_ML integer}]
	By \eqref{G zw FDE} and \eqref{G zw til asym}, the function $\widetilde{G}$ is a unique solution to the $m$-th order linear (ordinary) differential equation
	\begin{gather} \label{G zw ODE}
		\partial_z ^m \widetilde{G}(z,w)=z^m \widetilde{G}(z,w)+F_{m,c}(zw),
	\end{gather}
	which satisfies the asymptotic behaviour
	\begin{gather} \label{G zw asym int}
		\widetilde{G}(z,w) \sim \frac{w^{m+mc-1}}{ \Gamma(2m+mc) } z^{2m+mc-1},\qquad z \to 0.
	\end{gather}
	Here the function $F_{m,c}$ is given by \eqref{F integer}. Note that the $m$ initial conditions are provided by differentiating \eqref{G zw asym int}.
	
We aim to solve the initial value problem \eqref{G zw ODE} and~\eqref{G zw asym int}.
	It is well known that the functions~$g_{j,m}$ given by \eqref{m indep sol} provide the solutions to the homogeneous equation of~\eqref{G zw ODE}, see, e.g., \cite[Theorem~5.18]{gorenflo2014mittag}.
	By~\eqref{ML3}, we have that
	\begin{gather} \label{gj asymp}
		g_{j,m}(z) \sim z^{j-1} ,\qquad z \to 0.
	\end{gather}
	Using this, it is easy to observe that
	\begin{gather*}
		W(g_{1,m},\dots,g_{m,m})(z)|_{z=0}=\prod_{j=0}^{m-1} j!=G(m+1),
	\end{gather*}
	where $G$ is the Barnes $G$-function \cite[Section~5.17]{olver2010nist}. Due to Abel's identity, for $m>1$, this leads to
	\begin{gather} \label{Wronskian g1m}
		W(g_{1,m},\dots,g_{m,m})(z) \equiv G(m+1).
	\end{gather}
	On the other hand for $m=1$, we have obviously $W(g_1)(z)=g_1(z).$
	
	For each $j \in \{1,\dots, m\}$, let us write
	\begin{gather*}
		\widetilde{W}_{j,m}(z,w):=\det
		\begin{bmatrix}
			g_{1,m}(z) & \dots & g_{j-1,m}(z) & 0 & g_{j+1,m}(z) & \dots & g_{m,m}(z)
			\\
			g_{1,m}'(z) & \dots & g_{j-1,m}'(z) & 0 & g_{j+1,m}'(z) & \dots & g_{m,m}'(z)
			\\
			\vdots & \vdots & \vdots & \vdots & \vdots & \vdots & \vdots
			\\
			g_{1,m}^{(m-1)}(z) & \dots & g_{j-1,m}^{(m-1)}(z) & F_{m,c}(zw) & g_{j+1,m}^{(m-1)}(z) & \dots & g_{m,m}^{(m-1)}(z)
		\end{bmatrix}.
	\end{gather*}
	Note here that by \eqref{W_j(z)}, we have
	\begin{gather} \label{Wj zw F Wj}
		\widetilde{W}_{j,m}(z,w)=(-1)^{m-j} F_{m,c}(zw) W_{j,m}(z).
	\end{gather}
	
	From now on, we shall only consider the case $m>1$. The other case $m=1$ follows along the same lines. In this case, one can also easily solve the associated linear non-homogenous ODE of degree $1$ employing an integrating factor.
	
	By virtue of the method of variation of parameters and \eqref{Wronskian g1m}, one can write
	\begin{align*}
			\widetilde{G}(z,w)&= \int_0^z \frac{1}{W(g_{1,m},\dots,g_{m,m})(s)}\sum_{j=1}^{m} 	\widetilde{W}_{j,m}(s,w) g_{j,m}(z)\,{\rm d}s+ \sum_{j=1}^{m} C_{j,m}(w) g_{j,m}(z)
			\\
			&=\frac{1}{G(m+1)}\int_0^z \sum_{j=1}^{m} 	\widetilde{W}_{j,m}(s,w) g_{j,m}(z)\,{\rm d}s+ \sum_{j=1}^{m} C_{j,m}(w) g_{j,m}(z)
	\end{align*}
	for some functions $C_{j,m}$. We shall show that in the above expression $\widetilde{G}$ satisfies the asymptotic behaviour~\eqref{G zw asym int} if and only if $C_{j,m} \equiv 0$ for every~$j$.
	
	By \eqref{gj asymp} and \eqref{Wronskian g1m}, direct computations of the determinant give that as $z \to 0$,
	\begin{gather*}
	W_{j,m}(z) \sim \frac{G(m+1)}{(m-1)!} \binom{m-1}{j-1} z^{m-j}.
	\end{gather*}
	On the other hand, by \eqref{ML function}, as $z \to 0$,
	\begin{gather*}
	F_{m,c}(zw) \sim \frac{w^{m+mc-1}}{\Gamma(m+mc)} z^{m+mc-1}.
	\end{gather*}
	Combining these two equations, we have that as $z\to 0$,
	\begin{gather*}
	\widetilde{W}_{j,m}(z,w) \sim \frac{w^{m+mc-1}}{\Gamma(m+mc)} (-1)^{m-j} \frac{G(m+1)}{(m-1)!} \binom{m-1}{j-1} z^{2m+mc-1-j}.
	\end{gather*}
	This leads to
	\begin{gather*}
	\frac{1}{G(m+1)}\int_0^z \widetilde{W}_{j,m}(s,w) g_{j,m}(z)\,{\rm d}s\\
\qquad{} \sim \frac{w^{m+mc-1}}{\Gamma(m+mc)} \frac{1}{(m-1)!} \binom{m-1}{j-1} \frac{(-1)^{m-j} }{2m+mc-j} z^{2m+mc-1}.
	\end{gather*}
	
	Next, we show the identity
	\begin{gather}\label{Comb identity}
		\frac{1}{(m-1)!} \sum_{j=1}^{m} \binom{m-1}{j-1} \frac{(-1)^{m-j} }{2m+mc-j} = \frac{\Gamma(m+mc)}{\Gamma(2m+mc)}.
	\end{gather}
	By the binomial theorem, we have
	\begin{gather*}
	\sum_{k=0}^{m-1} \binom{m-1}{k} x^{m+mc-1+k} = x^{m+mc-1}(1+x)^{m-1}.
	\end{gather*}
	Integrating this identity, we have
	\begin{gather*}
	\sum_{k=0}^{m-1} \binom{m-1}{k} \frac{x^{m+mc+k}}{m+mc+k} = \int_0^x t^{m+mc-1}(1+t)^{m-1}\,{\rm d}t.
	\end{gather*}
	Letting $x=-1$ and using the change of the variable $t=-s$, we have
	\begin{align*}
		\sum_{k=0}^{m-1} \binom{m-1}{k} \frac{(-1)^{k}}{m+mc+k} &= \int_0^1 s^{m+mc-1}(1-s)^{m-1}\,{\rm d}s
		\\
		&=B(m+mc,m)=(m-1)! \frac{\Gamma(m+mc)}{\Gamma(2m+mc)},
	\end{align*}
	where $B$ is the beta function. This gives~\eqref{Comb identity}.
	
	Using \eqref{Comb identity}, we have
	\begin{gather*}
	\frac{1}{G(m+1)}\int_0^z \sum_{j=1}^{m} 	\widetilde{W}_{j,m}(s,w) g_{j,m}(z)\,{\rm d}s \sim \frac{w^{m+mc-1}}{ \Gamma(2m+mc) } z^{2m+mc-1}.
	\end{gather*}
	Thus by \eqref{G zw asym int}, we obtain that $C_{j,m} \equiv 0$ for all $j \in \{1,\dots, m\}$.
	Therefore by~\eqref{Wj zw F Wj}, we conclude that
	\begin{align*}
		\widetilde{G}(z,w)&= \frac{1}{G(m+1)}\int_0^1 \sum_{j=1}^{m} 	\widetilde{W}_{j,m}(sz,w) z g_{j,m}(z)\,{\rm d}s
		\\
		&= \frac{1}{G(m+1)}\int_0^1 \sum_{j=1}^{m} (-1)^{m-j} 	\widetilde{W}_{j,m}(sz) z g_{j,m}(z)F_{m,c}(szw)\,{\rm d}s.
	\end{align*}
	This completes the proof.
\end{proof}

\section{Translation-invariant scaling limit and Ward's equation}

The following two subsections are devoted to proving Propositions~\ref{Thm_mass1}, the mass-one equation for the Ginibre $\infty$-point process, and
Ward's equation at finite-$N$ for general potentials, Proposition~\ref{Thm_Ward N}.

\subsection{Mass-one equation}

In this subsection, we consider the Pfaffian $\infty$-point process with the correlation kernel \eqref{K Gkernel E}.

By \eqref{RNk Pf}, we have
\begin{gather*}
R_1(z)=(\bar{z}-z){\rm e}^{ -2|z|^2 } \kappa(z,\bar{z})
\end{gather*}
and
\begin{gather*}
	R_2(z,w)=(\bar{z}-z)(\bar{w}-w) {\rm e}^{ -2|z|^2-2|w|^2 }
\big( \kappa(z,\bar{z}) \kappa(w,\bar{w})-|\kappa(z,w)|^2+|\kappa(z,\bar{w})|^2 \big).
\end{gather*}
Therefore the Berezin kernel is written as
\begin{gather} \label{B}
	B(z,w)= (\bar{w}-w) {\rm e}^{ -2|w|^2 } \frac{ |\kappa(z,w)|^2-|\kappa(z,\bar{w})|^2 }{ \kappa(z,\bar{z})}.
\end{gather}
We now prove Proposition~\ref{Thm_mass1}.

\begin{proof}[Proof of Proposition~\ref{Thm_mass1}]
	By \eqref{B} and conjugation symmetry, the mass-one equation \eqref{mass-one} is equivalent to
	\begin{gather} \label{mass-one 1}
		\kappa(z,\bar{z})=2 \int_{ \C } (\bar{w}-w) {\rm e}^{ -2|w|^2 } |\kappa(z,w)|^2 \,{\rm d}A(w).
	\end{gather}
	
	Note that the pre-kernel $\kappa \equiv \kappa_a^\R$ given in \eqref{kappa a} is written as
	\begin{gather*}
	\kappa(z,w)=\sqrt{\pi} {\rm e}^{z^2+w^2} \int_{E} f'(w-u)f(z-u)-f'(z-u) f(w-u) \,{\rm d}u, \\
	f(z):=\tfrac12 \erfc\big(\sqrt{2}z\big),
	\end{gather*}
	where $E \equiv E_a=(-\infty,a)$.
	Using this expression, we rewrite the right-hand side of the equation~\eqref{mass-one 1} as
	\begin{gather*}
	-4{\rm i} {\rm e}^{z^2+\bar{z}^2} \int_{ \R^2 } y {\rm e}^{-4y^2} \int_{ E^2 } F(u,v) \,{\rm d}u \,{\rm d}v \,{\rm d}x \,{\rm d}y,
	\end{gather*}
	where
	\begin{gather*}
		F(u,v) = \big( f'(x+{\rm i}y-u)f(z-u)-f'(z-u) f(x+{\rm i}y-u) \big)
		\\
\hphantom{F(u,v) =}{}
\times \big( f'(x-{\rm i}y-v)f(\bar{z}-v)-f'(\bar{z}-v) f(x-{\rm i}y-v) \big).
	\end{gather*}
	
We now compute
\begin{gather*}
		g(u,v):= \int_{ \R^2 } y {\rm e}^{-4y^2} f'(x+{\rm i}y-u)f'(x-{\rm i}y-v) \,{\rm d}y \,{\rm d}x.
\end{gather*}
Due to the translation invariance in $x$, it suffices to consider the case $v=0$. Then we have
	\begin{gather*}
		g(u,0) = \frac{2}{\pi} \int_{ \R^2 } y {\rm e}^{-4y^2} {\rm e}^{-2(x+{\rm i}y-u)^2} {\rm e}^{-2(x-{\rm i}y)^2} \,{\rm d}y \,{\rm d}x
	= \frac{2}{\pi} \int_{\R} {\rm e}^{-4x^2+4xu-2u^2} \int_{\R} y {\rm e}^{4{\rm i}uy} \,{\rm d}y \,{\rm d}x.
	\end{gather*}
Note that in the sense of distributions, we have
	\begin{gather*}
			\int_{\R} y {\rm e}^{4{\rm i}uy} \,{\rm d}y= -\int_{\R} y {\rm e}^{-4{\rm i}uy} \,{\rm d}y=-\frac{1}{16} \int_{\R} y {\rm e}^{-{\rm i}u y} \,{\rm d}y=\frac{\pi }{8{\rm i}} \delta'(u).
		\end{gather*}
	Thus it follows from $\int_\R {\rm e}^{-4x^2}\,{\rm d}x=\tfrac{\sqrt{\pi}}{2}$ that
	\begin{gather*}
		g(u,v)= \frac{\sqrt{\pi}}{8{\rm i}} {\rm e}^{-(u-v)^2} \delta'(u-v).
	\end{gather*}
	This leads to
	\begin{gather*}
 \frac{8{\rm i}}{\sqrt{\pi}} {\rm e}^{(u-v)^2} \int_{ \R^2 } y {\rm e}^{-4y^2} F(u,v) \,{\rm d}x\,{\rm d}y
		\\
\qquad{} =f(z-u)f(\bar{z}-v) \delta'(u-v)-f(z-u)f'(\bar{z}-v) \delta(u-v)
		\\
\qquad\quad{} +f'(z-u) f(\bar{z}-v) \delta(u-v)-f'(z-u) f'(\bar{z}-v) \Theta(u-v),
	\end{gather*}
	where $\Theta$ denotes the Heaviside theta function. Combining the above equations, interchanging the integrals, and using integration by parts, we obtain
	\begin{gather*}
 2 \int_{ \C } (\bar{w}-w) {\rm e}^{ -2|w|^2 } |\kappa(z,w)|^2 \,{\rm d}A(w) \\
 \qquad{} =-4{\rm i} {\rm e}^{z^2+\bar{z}^2} \int_{ E^2 } \int_{ \R^2 } y \, {\rm e}^{-4y^2} F(u,v) \,{\rm d}x \,{\rm d}y \,{\rm d}u \,{\rm d}v
		\\
\qquad{}= \sqrt{\pi} {\rm e}^{z^2+\bar{z}^2} \int_{E} f'(\bar{z}-u)f(z-u)-f'(z-u) f(\bar{z}-u)\,{\rm d}u=\kappa(z,\bar{z}),
	\end{gather*}
	which completes the proof.
\end{proof}

\subsection[Ward's equation at finite-N]{Ward's equation at finite-$\boldsymbol{N}$} In this subsection, we consider the symplectic ensemble~\eqref{Gibbs} with general external potential~$Q$ and prove Proposition~\ref{Thm_Ward N}. In the sequel, we assume that $\theta=0$ without loss of generality.

\begin{proof}[Proof of Proposition~\ref{Thm_Ward N}]
	We denote by $\E_N$ the expectation with respect to the Gibbs measure~\eqref{Gibbs}. Let $\psi_N$ be a fixed test function.	
	Integration by parts gives that, for each $j$,
	\begin{gather*}
		\E_N[\d\psi_N(\zeta_j)]= \E_N[\d_j\Ham_N(\zeta_1,\dots,\zeta_N)\psi_N(\zeta_j)],
	\end{gather*}
	where $\Ham_N$ is the Hamiltonian given by~\eqref{Ham}.
	Summing in $j$, we obtain
	\begin{gather*}
\sum_{j=1}^N\E_N [\d\psi_N(\zeta_j)]= \E_N\sum_{j=1}^N\psi_N(\zeta_j)\left[N\d Q(\zeta_j)- \frac{2}{\zeta_j-\overline{\zeta}_j} \right]
			\\
\hphantom{\sum_{j=1}^N\E_N [\d\psi_N(\zeta_j)]=}{} -\E_N\sum_{j=1}^N\psi_N(\zeta_j)\Bigg[\sum_{ \substack{k=1 \\ (k\ne j)}}^N \left(\frac {1}{\zeta_j-\zeta_k}+\frac{1}{\zeta_j-\overline{\zeta}_k}\right)\Bigg].
		\end{gather*}
	Thus we have
	\begin{gather} \label{Ward N}
		\E_N \I_N[\psi_N]-\E_N \II_N[\psi_N] +\E_N \III_N[\psi_N]+\E_N \RN{4}_N[\psi_N]=0,
	\end{gather}
	where
	\begin{alignat*}{3}
&			\I_N[\psi_N]=\sum_{\substack{j,k=1\\j\ne k}}^N \psi_N(\zeta_j) \left(\frac {1}{\zeta_j-\zeta_k}+\frac{1}{\zeta_j-\overline{\zeta}_k}\right), \qquad&&
			\II_N[\psi_N]=N\sum_{j=1}^N [\psi_N\d Q](\zeta_j),& \\
&			\III_N[\psi_N]=\sum_{j=1}^N\d\psi_N(\zeta_j), \qquad&&
			\RN{4}_N[\psi_N]=\sum_{j=1}^N \psi_N(\zeta_j) \frac{2}{\zeta_j-\overline{\zeta}_j} .&
	\end{alignat*}
	
	Due to the conjugation symmetry $Q(\zeta)=Q\big(\bar{\zeta}\big)$, one can easily observe that for any $j$,
	\begin{gather*}
	\Ham_N(\zeta_1,\dots, \zeta_j,\dots ,\zeta_N)= \Ham_N\big(\zeta_1,\dots, \bar{\zeta}_j,\dots \zeta_N\big).
	\end{gather*}
	Then it follows from the definition \eqref{bfRNk def} that $\bfR_{N,2}(\zeta,\eta)=\bfR_{N,2}(\zeta,\bar{\eta})$.
	Using the change of variable $\eta \mapsto \bar{\eta}$, we have
	\begin{gather*}
		\int_\C \frac{ \bfR_{N,2}(\zeta,\eta)}{\zeta-\bar{\eta}} \,{\rm d}A(\eta) = \int_\C \frac{ \bfR_{N,2}(\zeta,\eta)}{\zeta-\eta} \,{\rm d}A(\eta),
	\end{gather*}
	which leads to
	\begin{gather*}
		\E_N \RN{1}_N[\psi_N]= 2\int_{ \C^2 } \psi_N(\zeta) \frac{ \bfR_{N,2}(\zeta,\eta) }{\zeta-\eta} \,{\rm d}A(\zeta) \,{\rm d}A(\eta).
	\end{gather*}
	Also it follows from the definition and integration by parts that
	\begin{gather*}
		\E_N \RN{2}_N[\psi_N]=N \int_\C \psi_N(\zeta) \bfR_{N}(\zeta)\partial Q(\zeta) \,{\rm d}A(\zeta),
		\\
		\E_N \RN{3}_N[\psi_N]=-\int_\C \psi_N(\zeta) \partial \bfR_{N}(\zeta) \,{\rm d}A(\zeta), \qquad
		\E_N \RN{4}_N[\psi_N]=2\int_\C \psi_N(\zeta) \frac{\bfR_{N}(\zeta)}{\zeta-\overline{\zeta}} \,{\rm d}A(\zeta).
	\end{gather*}
	
	To describe the rescaled version of Ward's equation, let $\psi(z):=\psi_N(\zeta)$, where $\zeta=p+ r_N z.$ By the relation
	\begin{gather*}
		R_{N,k}(z_1,\dots,z_k)=r_N^{2k} \, \bfR_{N,k}(\zeta_1,\dots, \zeta_k), \qquad \zeta_j=p+ r_N z_j
	\end{gather*}
	and the change of variables $\zeta \mapsto p+ r_N z$, $\eta \mapsto p+ r_N w$, we have
	\begin{gather*}
		\E_N \I_N[\psi_N]=\frac{2}{r_N} \int_{ \C } \psi(z) \int_{ \C } \frac{ R_{N,2}(z,w)}{z-w} \,{\rm d}A(w) \,{\rm d}A(z).
	\end{gather*}
	Similarly,
	\begin{gather*}
		\E_N \II_N[\psi_N]
		=N \int_{ \C } \psi (z) \partial Q(p+ r_N z) R_{N}(z) \,{\rm d}A(z),
		\\
		\E_N \III_N[\psi_N]=-\frac{1}{r_N}\int_{ \C } \psi(z) \partial R_{N}(z) \,{\rm d}A(z), \qquad \E_N \RN{4}[\psi_N]= \frac{2}{r_N} \int_{ \C } \psi(z) \frac{ R_{N}(z) }{ z-\bar{z} } \,{\rm d}A(z).
	\end{gather*}	
	Therefore by \eqref{Ward N}, we obtain
	\begin{gather*}
	\int_{ \C } \frac{ R_{N,2}(z,w)}{z-w} \,{\rm d}A(w)
	= \frac{Nr_N}{2} \partial Q(p+r_Nz) R_{N}(z)+\frac{\partial R_N(z)}{2} -\frac{ R_{N}(z) }{ z-\bar{z} }.
	\end{gather*}
	Dividing by $R_N(z)$, this identity is rewritten in terms of Berezin kernel $B_N$ as
	\begin{gather*}
	\int_{ \C } \frac{ B_N(z,w) }{z-w} \,{\rm d}A(w)=\int_{ \C } \frac{R_N(w)}{z-w} \,{\rm d}A(w)
	-\frac{Nr_N}{2} \partial Q(p+r_Nz)
	-\frac{ \partial \log R_N(z)}{2} + \frac{ 1 }{ z-\bar{z} }.
	\end{gather*}
	By taking $\bar{\partial}$ derivative on both sides of this equation, Proposition~\ref{Thm_Ward N} follows.	
\end{proof}

In the next two subsections, we characterise translation-invariant scaling limits of symplectic ensembles, first, by solving the mass-one equation, and second, by the solution of the limiting Ward's equation. This leads to the proof of Theorem~\ref{Thm_TIsol}. By \eqref{kappa Psi} together with \eqref{Psi J hat}, it suffices to show that
\begin{gather}\label{Psi}
	\Psi(z)
	=\frac{1}{2{\rm i}\sqrt{\pi}} \int_E {\rm e}^{-u^2}\big({\rm e}^{ 2{\rm i}uz }-{\rm e}^{ -2{\rm i}uz } \big) \, \frac{{\rm d}u}{u}, \qquad 	\wh{J}(u)
	=\frac{2\sqrt{\pi}}{{\rm i}} \frac{{\rm e}^{-u^2/4}}{u}\mathbbm{1}_{E}(u/2),
\end{gather}
where $E$ is of the form $(-a,a)$ with $a \in [0,\infty]$.

\subsection{Translation-invariant solutions to the mass-one equation}
In this subsection, we prove the following proposition.

\begin{Proposition}\label{Prop_mass1 TIsol}
	Translation-invariant solution to the mass-one equation \eqref{mass-one} are of the form \eqref{Psi}, where $E$ is some symmetric Borel set.
\end{Proposition}

It follows from this proposition that the specified form of the pre-kernel \eqref{kappa t.i.} with any symmetric Borel set satisfies the mass-one equation \eqref{mass-one}.
However, in the following subsection, it will be shown that the if $E$ is not a single interval, then the pre-kernel~\eqref{kappa t.i.} does not satisfy Ward's equation.
(An analogous result for the random normal matrix ensemble was obtained in~\cite{MR4030288}.)

\begin{proof}[Proof of Proposition~\ref{Prop_mass1 TIsol}]
	In terms of $\Psi$, the mass-one equation \eqref{mass-one 1} is written as
	\begin{gather*}
		\Psi(z-\bar{z})=2 \int_{ \C } (\bar{w}-w) {\rm e}^{-4(\im w)^2} |\Psi(z-w)|^2\,{\rm d}A(w).
	\end{gather*}
	By translation invariance, it is equivalent to the fact that for any $t \in \R$,
	\begin{gather} \label{mass1 Psi}
			\Psi(2{\rm i}t)
			=\frac{4}{\pi {\rm i}} \int_{ \R^2 } y \, {\rm e}^{-4y^2} |\Psi( x+{\rm i}(y-t) )|^2\,{\rm d}x\,{\rm d}y.
	\end{gather}
	
	In terms of $\wh{J}$, the equation \eqref{mass1 Psi} is written as
	\begin{gather*}
			\Psi(2{\rm i}t)
			=\frac{1}{\pi^3{\rm i}} \int_{ \R^4 } y \, {\rm e}^{-4y^2} {\rm e}^{{\rm i}(u+v)x} {\rm e}^{-(u-v)(y-t)} \wh{J}(u) \wh{J}(v)\,{\rm d}x\,{\rm d}y\,{\rm d}u\,{\rm d}v .
	\end{gather*}
	Since
	\begin{gather*}
	\frac{1}{2\pi } \int_\R {\rm e}^{{\rm i}(u+v)x}\,{\rm d}x=\delta(u+v),
	\end{gather*}
	the above equation is simplified to
	\begin{gather*}
			\Psi(2{\rm i}t)=\frac{2}{\pi^2{\rm i}} \int_{ \R^2 } y {\rm e}^{-4y^2} {\rm e}^{-2u(y-t)} \wh{J}(u) \wh{J}(-u)\,{\rm d}y\,{\rm d}u
		= \frac{2{\rm i}}{\pi^2} \int_{ \R^2 } y {\rm e}^{-4y^2} {\rm e}^{-2u(y-t)} \wh{J}(u)^2\,{\rm d}y\,{\rm d}u.
		\end{gather*}
	Moreover, it follows from
	\begin{gather*}
		\int_{\R} y {\rm e}^{-4y^2-2uy}\,{\rm d}y=-\frac{\sqrt{\pi}}{8} \, u \, {\rm e}^{u^2/4}
	\end{gather*}
	that the mass-one equation holds if and only if
	\begin{gather} \label{mass1 var1}
		f(t)=g(t),
	\end{gather}
	where
	\begin{gather*}
		f(t):=\int_{\R} {\rm e}^{-2ut} \wh{J}(u)\,{\rm d}u,
		\qquad
		g(t):=\frac{{\rm i}}{2\sqrt{\pi} } \int_{ \R } {\rm e}^{-2ut} \, u \, {\rm e}^{u^2/4} \wh{J}(u)^2\,{\rm d}u.
	\end{gather*}
	Therefore we can readily see that the kernel \eqref{kappa t.i.} satisfies the mass-one equation.
	
	Since $\wh{J} \in L^1 \cap L^2$, we can extend the domain of $f$, $g$ to the complex plane. Then by \eqref{mass1 var1} we have $f(z)=g(z)$ for $z\in\C$, which in particular leads to $f({\rm i}t)=g({\rm i}t)$ for $t \in \R$, i.e.,
	\begin{gather*}
		\int_{\R} {\rm e}^{-2{\rm i}ut} \wh{J}(u)\,{\rm d}u=\frac{{\rm i}}{2\sqrt{\pi} } \int_{ \R } {\rm e}^{-2{\rm i}ut} u {\rm e}^{u^2/4} \wh{J}(u)^2\,{\rm d}u.
	\end{gather*}
	By Fourier inversion, this is equivalent to
	\begin{gather*}
		\wh{J}(u)\left( 1+\frac{u {\rm e}^{u^2/4}}{2{\rm i}\sqrt{\pi} } \wh{J}(u) \right)=0,\qquad \text{a.e.}
	\end{gather*}
It gives rise to
	\begin{gather*}
		\wh{J}(u)=\frac{2\sqrt{\pi}}{{\rm i}} \frac{{\rm e}^{-u^2/4}}{u}\mathbbm{1}_{E}(u/2), \qquad \text{a.e.}
	\end{gather*}
	for some Borel set $E,$ which is symmetric with respect to the origin since $\wh{J}$ is odd. Here the argument $u/2$ is merely used for convenience.
\end{proof}

\subsection{Translation-invariant solution of Ward's equation}
In this subsection, we characterise the translation-invariant solutions to Ward's equation \eqref{Ward} and complete the proof of Theorem~\ref{Thm_TIsol}.

Recall that by \eqref{kappa Psi}, we have
\begin{gather*}
R(z) =-2{\rm i} \im z {\rm e}^{-4 (\im z)^2} \Psi(z-\bar{z}),
\\
B(z,w) =-2{\rm i} \im w\, {\rm e}^{ -4(\im w)^2 } \frac{ |\Psi(z-w)|^2-|\Psi(z-\bar{w})|^2 }{ \Psi(z-\bar{z})}.
\end{gather*}
Let us first observe the following.

\begin{Lemma}\label{Lem_Ward var1}
	Ward's equation \eqref{Ward} holds if and only if
	\begin{gather*}
 \Psi \Psi''(z-\bar{z}) - \Psi'(z-\bar{z})^2+\dfrac{\Psi(z-\bar{z})^2}{(z-\bar{z})^2}
			\\
\qquad{} = \frac{4}{{\rm i}} \Psi(z-\bar{z})\int_\C \im w\, {\rm e}^{ -4(\im w)^2 } \Psi(z-w) \overline{ \Psi'(z-w) } \left( \frac{1}{z-w}+\frac{1}{z-\bar{w}} \right) {\rm d}A(w)
			\\
\qquad\quad{} + \frac{4}{{\rm i}} \Psi'(z-\bar{z}) \int_\C \im w\, {\rm e}^{ -4(\im w)^2 } |\Psi(z-w)|^2 \left( \frac{1}{z-w}+\frac{1}{z-\bar{w}} \right) {\rm d}A(w).
		\end{gather*}
\end{Lemma}

\begin{proof}Note that
\begin{align*}
\bp_z C(z)& =\int_\C \frac{\bp_z B(z,w)}{z-w}\,{\rm d}A(w)+\int B(z,w) \bp_z \left( \frac{1}{z-w} \right) {\rm d}A(w)\\
&=\int_\C \frac{\bp_z B(z,w)}{z-w}\,{\rm d}A(w)+B(z,z).
\end{align*}
Since $B(z,z)=R(z)$, Ward's equation \eqref{Ward} is equivalent to
\begin{gather*}
\int_\C \frac{\bp_z B(z,w)}{z-w}\,{\rm d}A(w)=-1-\dfrac12 \Delta \log R(z)+\dfrac{1}{(z-\bar{z})^2}.
\end{gather*}
	Notice also that
	\begin{gather*}
			\dfrac12 \Delta \log R(z) =\frac12 \Delta \log \big( \im z {\rm e}^{-4 (\im\,z)^2} \Psi(z-\bar{z}) \big)
			=\dfrac{1}{2(z-\bar{z})^2}-1 +\frac12 \Delta \log \Psi(z-\bar{z}).
		\end{gather*}
	Thus one can rewrite above equation as
	\begin{gather*}
			\int_\C \frac{\bp_z B(z,w)}{z-w}\,{\rm d}A(w)
			 =\frac12 \frac{ \Psi''(z-\bar{z}) }{ \Psi(z-\bar{z}) }-\frac12 \left( \frac{\Psi'(z-\bar{z})}{\Psi(z-\bar{z})} \right)^2+\dfrac{1}{2(z-\bar{z})^2}.
		\end{gather*}
	Using the change of variable $w \mapsto \bar{w}$, the left-hand side of this equation is computed as
	\begin{gather*}
 \frac{2}{{\rm i}} \int_\C \frac{ \im w\, {\rm e}^{ -4(\im w)^2 }}{z-w} \bp_z \left( \frac{ |\Psi(z-w)|^2-|\Psi(z-\bar{w})|^2 }{ \Psi(z-\bar{z})} \right) {\rm d}A(w)
			\\
\qquad{}= \frac{2}{{\rm i}} \frac{1}{\Psi(z-\bar{z})}\int_\C \im w\, {\rm e}^{ -4(\im\,w)^2 } \Psi(z-w) \overline{ \Psi'(z-w) } \left( \frac{1}{z-w}+\frac{1}{z-\bar{w}} \right) {\rm d}A(w)
			\\
\qquad\quad{}+ \frac{2}{{\rm i}} \frac{ \Psi'(z-\bar{z}) }{ \Psi(z-\bar{z})^2 } \int_\C \im w\, {\rm e}^{ -4(\im w)^2 } |\Psi(z-w)|^2 \left( \frac{1}{z-w}+\frac{1}{z-\bar{w}} \right) {\rm d}A(w).
		\end{gather*}
	Thus Ward's equation is rewritten as
	\begin{gather*}
\frac12 \frac{ \Psi''(z-\bar{z}) }{ \Psi(z-\bar{z}) }-\frac12 \left( \frac{\Psi'(z-\bar{z})}{\Psi(z-\bar{z})} \right)^2+\dfrac{1}{2(z-\bar{z})^2}
			\\
\qquad{} = \frac{2}{{\rm i}} \frac{1}{\Psi(z-\bar{z})}\int_\C \im w\, {\rm e}^{ -4(\im w)^2 } \Psi(z-w) \overline{ \Psi'(z-w) } \left( \frac{1}{z-w}+\frac{1}{z-\bar{w}} \right) {\rm d}A(w)
			\\
\qquad\quad{} + \frac{2}{{\rm i}} \frac{ \Psi'(z-\bar{z}) }{ \Psi(z-\bar{z})^2 } \int_\C \im w\, {\rm e}^{ -4(\im w)^2 } |\Psi(z-w)|^2 \left( \frac{1}{z-w}+\frac{1}{z-\bar{w}} \right) {\rm d}A(w).\tag*{\qed}
\end{gather*}
\renewcommand{\qed}{}
\end{proof}

Let us write
\begin{gather*}
	L(t) :=\frac{4}{\pi {\rm i}} \int_{\R^2} y {\rm e}^{ -4y^2 } |\Psi({\rm i}t-x-{\rm i}y)|^2 \left( \frac{1}{{\rm i}t-x-{\rm i}y}+\frac{1}{{\rm i}t-x+{\rm i}y} \right) {\rm d}x\,{\rm d}y,
	\\
	K(t) :=\frac{4}{\pi {\rm i}} \int_{\R^2} y {\rm e}^{ -4y^2 } \Psi({\rm i}t-x-{\rm i}y) \overline{ \Psi'({\rm i}t-x-{\rm i}y) } \left( \frac{1}{{\rm i}t-x-{\rm i}y}+\frac{1}{{\rm i}t-x+{\rm i}y} \right) {\rm d}x\,{\rm d}y.
\end{gather*}
We now show the following lemma, the key reduction of Ward's equation.

\begin{Lemma}\label{Lem_Ward var2}
	Ward's equation \eqref{Ward} holds if and only if
	\begin{gather*} %\label{Ward var2}
		\frac{{\rm d}}{{\rm d}t} \left[ \frac{L(t)}{\Psi(2{\rm i}t)} \right] = -\frac{{\rm d}}{{\rm d}t}\left[ \frac{\Psi'(2{\rm i}t)}{\Psi(2{\rm i}t)}+\frac{{\rm i}}{2t} \right]-8t {\rm e}^{-4t^2} \Psi(2{\rm i}t).
	\end{gather*}
\end{Lemma}

\begin{proof}
	
	By Lemma~\ref{Lem_Ward var1}, Ward's equation holds if and only if for any $t \in \R,$
	\begin{gather} \label{Ward var1}
		\Psi \Psi''(2{\rm i}t) - \Psi'(2{\rm i}t)^2-\dfrac{\Psi(2{\rm i}t)^2}{4t^2}= \Psi(2{\rm i}t) K(t)+\Psi'(2{\rm i}t) L(t).
	\end{gather}
	Note that
	\begin{gather*}
		\Psi(z)=\frac{1}{2\pi} \int_{\R} {\rm e}^{{\rm i}zu} \wh{J}(u)\,{\rm d}u, \qquad \Psi'(z)=\frac{{\rm i}}{2\pi} \int_{\R} u\,{\rm e}^{{\rm i}zu} \wh{J}(u)\,{\rm d}u.
	\end{gather*}
	Using this, we have
	\begin{gather*}
			L(t) =\frac{1}{\pi^3 {\rm i}} \int_{ \R^4 } y {\rm e}^{ -4y^2 } {\rm e}^{{\rm i}(u+v)x} {\rm e}^{-(u-v)(y-t)} \frac{\wh{J}(u) \wh{J}(v)}{{\rm i}t-x-{\rm i}y} \,{\rm d}x\,{\rm d}y\,{\rm d}u\,{\rm d}v
			\\
\hphantom{L(t) =}{} +\frac{1}{\pi^3 {\rm i}} \int_{ \R^4 } y {\rm e}^{ -4y^2 } {\rm e}^{{\rm i}(u+v)x} {\rm e}^{-(u-v)(y-t)} \frac{\wh{J}(u) \wh{J}(v)}{{\rm i}t-x+{\rm i}y} \,{\rm d}x\,{\rm d}y\,{\rm d}u\,{\rm d}v
		\end{gather*}
	and
	\begin{gather*}
			K(t) =-\frac{1}{\pi^3 } \int_{ \R^4 } y {\rm e}^{ -4y^2 } {\rm e}^{{\rm i}(u+v)x} {\rm e}^{-(u-v)(y-t)} v \frac{\wh{J}(u) \wh{J}(v)}{{\rm i}t-x-{\rm i}y} \,{\rm d}x\,{\rm d}y\,{\rm d}u\,{\rm d}v
			\\
\hphantom{K(t) =}{}
 -\frac{1}{\pi^3 } \int_{ \R^4 } y \, {\rm e}^{ -4y^2 } {\rm e}^{{\rm i}(u+v)x} {\rm e}^{-(u-v)(y-t)} v \frac{\wh{J}(u) \wh{J}(v)}{{\rm i}t-x+{\rm i}y} \,{\rm d}x\,{\rm d}y\,{\rm d}u\,{\rm d}v.
		\end{gather*}
	
	Let us first compute $L(t)$. It follows from the well-known Fourier transform of the following rational function that
	\begin{gather*}
			-\frac{1}{2\pi {\rm i}}\int_\R \frac{{\rm e}^{{\rm i}(u+v)x}}{{\rm i}t-x-{\rm i}y} \,{\rm d}x =\sgn(u+v) \mathbbm{1}_{ \{ (u+v)(y-t)<0\} } {\rm e}^{(u+v)(y-t)}.
		\end{gather*}
	This leads to
	\begin{gather*}
L(t) =-\frac{2}{\pi^2} \int_{ \R^3 } \sgn(u+v)\mathbbm{1}_{ \{ (u+v)(y-t)<0\} } y {\rm e}^{ -4y^2+2vy } {\rm e}^{-2vt} \wh{J}(u) \wh{J}(v) \,{\rm d}y\,{\rm d}u\,{\rm d}v\\
\hphantom{L(t) =}{}
 -\frac{2}{\pi^2} \int_{ \R^3 } \sgn(u+v) \mathbbm{1}_{ \{ (u+v)(y+t)>0\} } y {\rm e}^{ -4y^2-2uy } {\rm e}^{-2vt} \wh{J}(u) \wh{J}(v) \,{\rm d}y\,{\rm d}u\,{\rm d}v.
\end{gather*}
Notice here that
\begin{gather*}
\int_{-\infty}^t y {\rm e}^{-4y^2+2vy}\,{\rm d}y
			 =\frac{\sqrt{\pi}}{ 16 }v {\rm e}^{ \frac{v^2}{4} } \erfc\left( \frac{v}{2}-2t \right)-\frac18 {\rm e}^{-4t^2+2vt},
			\\
\int_t^\infty y {\rm e}^{-4y^2+2vy} {\rm d}y
=\frac{\sqrt{\pi}}{ 16 }v {\rm e}^{ \frac{v^2}{4} } \erfc\left( 2t-\frac{v}{2} \right)+\frac18 {\rm e}^{-4t^2+2vt}
		\end{gather*}
	and
\begin{gather*}
\int_{-t}^\infty y {\rm e}^{-4y^2-2uy}\,{\rm d}y
=-\frac{\sqrt{\pi}}{ 16 }u {\rm e}^{ \frac{u^2}{4} } \erfc\left( \frac{u}{2}-2t \right)+\frac18 {\rm e}^{-4t^2+2ut},\\
\int_{-\infty}^{-t} y {\rm e}^{-4y^2-2uy}\,{\rm d}y
=-\frac{\sqrt{\pi}}{ 16 }u {\rm e}^{ \frac{u^2}{4} } \erfc\left( 2t-\frac{u}{2} \right)-\frac18 {\rm e}^{-4t^2+2ut}.
		\end{gather*}
	Using this, the function $L$ is computed as
	\begin{gather*}%\label{L(t)}
			L(t)=\frac{1}{8\pi^2} \int_{ u+v<0 } \left(\sqrt{\pi}v {\rm e}^{ \frac{v^2}{4}-2vt } \erfc\left( 2t-\frac{v}{2} \right)+2 {\rm e}^{-4t^2}\right) \wh{J}(u) \wh{J}(v) \,{\rm d}u\,{\rm d}v
			\\
\hphantom{L(t)=}{}
 +\frac{1}{8\pi^2} \int_{ u+v>0 } \left(-\sqrt{\pi}v {\rm e}^{ \frac{v^2}{4}-2vt } \erfc\left( \frac{v}{2}-2t \right)+2 {\rm e}^{-4t^2}\right) \wh{J}(u) \wh{J}(v) \,{\rm d}u\,{\rm d}v
			\\
\hphantom{L(t)=}{} +\frac{1}{8\pi^2} \int_{ u+v>0 } \left(\sqrt{\pi}u {\rm e}^{ \frac{u^2}{4}-2vt } \erfc\left( \frac{u}{2}-2t \right)-2 {\rm e}^{-4t^2+2(u-v)t}\right) \wh{J}(u) \wh{J}(v) \,{\rm d}u\,{\rm d}v
			\\
\hphantom{L(t)=}{} +\frac{1}{8\pi^2} \int_{ u+v<0 } \left(-\sqrt{\pi}u {\rm e}^{ \frac{u^2}{4}-2vt } \erfc\left( 2t-\frac{u}{2} \right)-2 {\rm e}^{-4t^2+2(u-v)t}\right) \wh{J}(u) \wh{J}(v) \,{\rm d}u\,{\rm d}v.	
		\end{gather*}
	Similarly, we have
	\begin{gather*}
			K(t) =\frac{-{\rm i}}{8\pi^2} \int_{ u+v<0 } v\left(\sqrt{\pi}v {\rm e}^{ \frac{v^2}{4}-2vt } \erfc\left( 2t-\frac{v}{2} \right)+2 {\rm e}^{-4t^2}\right) \wh{J}(u) \wh{J}(v) \,{\rm d}u\,{\rm d}v
			\\
\hphantom{K(t) =}{}
 +\frac{{\rm i}}{8\pi^2} \int_{ u+v>0 } v\left(\sqrt{\pi}v {\rm e}^{ \frac{v^2}{4}-2vt } \erfc\left( \frac{v}{2}-2t \right)-2 {\rm e}^{-4t^2}\right) \wh{J}(u) \wh{J}(v) \,{\rm d}u\,{\rm d}v
			\\
\hphantom{K(t) =}{} -\frac{{\rm i}}{8\pi^2} \int_{ u+v>0 } v\left(\sqrt{\pi}u {\rm e}^{ \frac{u^2}{4}-2vt } \erfc\left( \frac{u}{2}-2t \right)-2 {\rm e}^{-4t^2+2(u-v)t}\right) \wh{J}(u) \wh{J}(v) \,{\rm d}u\,{\rm d}v
			\\
\hphantom{K(t) =}{} +\frac{{\rm i}}{8\pi^2} \int_{ u+v<0 } v\left(\sqrt{\pi}u {\rm e}^{ \frac{u^2}{4}-2vt } \erfc\left( 2t-\frac{u}{2} \right)+2 {\rm e}^{-4t^2+2(u-v)t}\right) \wh{J}(u) \wh{J}(v) \,{\rm d}u\,{\rm d}v.
		\end{gather*}
	
	Observe here that
	\begin{gather*}
	\frac{{\rm d}}{{\rm d}t} \left(\sqrt{\pi}v {\rm e}^{ \frac{v^2}{4}-2vt } \erfc\left( 2t-\frac{v}{2} \right)+2 {\rm e}^{-4t^2}\right)\\
\qquad{} =-2v \left(\sqrt{\pi}v {\rm e}^{ \frac{v^2}{4}-2vt } \erfc\left( 2t-\frac{v}{2} \right)+2 {\rm e}^{-4t^2}\right)-16t {\rm e}^{-4t^2}
	\end{gather*}
	and
	\begin{gather*}
	\frac{{\rm d}}{{\rm d}t} \left(\sqrt{\pi}v {\rm e}^{ \frac{v^2}{4}-2vt } \erfc\left( \frac{v}{2}-2t \right)-2 {\rm e}^{-4t^2}\right)\\
\qquad{}
 =-2v \left(\sqrt{\pi}v {\rm e}^{ \frac{v^2}{4}-2vt } \erfc\left( \frac{v}{2}-2t \right)-2 {\rm e}^{-4t^2}\right)+16t {\rm e}^{-4t^2}.
	\end{gather*}
	Similarly, we have
	\begin{gather*}
 \frac{{\rm d}}{{\rm d}t} \left(\sqrt{\pi}u {\rm e}^{ \frac{u^2}{4}-2vt } \erfc\left( \frac{u}{2}-2t \right)-2 {\rm e}^{-4t^2+2(u-v)t}\right)
		\\
\qquad{}=-2v \left(\sqrt{\pi}u {\rm e}^{ \frac{u^2}{4}-2vt } \erfc\left( \frac{u}{2}-2t \right)-2 {\rm e}^{-4t^2+2(u-v)t}\right)+16t{\rm e}^{-4t^2+2(u-v)t}
	\end{gather*}
	and
	\begin{gather*}
 \frac{{\rm d}}{{\rm d}t} \left(\sqrt{\pi}u {\rm e}^{ \frac{u^2}{4}-2vt } \erfc\left( 2t-\frac{u}{2} \right)+2 {\rm e}^{-4t^2+2(u-v)t}\right)
		\\
\qquad=-2v \left(\sqrt{\pi}u {\rm e}^{ \frac{u^2}{4}-2vt } \erfc\left( 2t-\tfrac{u}{2} \right)+2 {\rm e}^{-4t^2+2(u-v)t}\right) -16t{\rm e}^{-4t^2+2(u-v)t}.
	\end{gather*}
	
	Combining the above equations, we obtain
	\begin{gather} \label{L K rel}
		L'(t)=-2{\rm i} K(t)+M(t),
	\end{gather}
	where
	\begin{gather*}
			M(t)
			:=\frac{2t {\rm e}^{-4t^2}}{\pi^2} \int_{\R^2 } {\rm e}^{2(u-v)t} \wh{J}(u) \wh{J}(v) \,{\rm d}u\,{\rm d}v=-8t {\rm e}^{-4t^2} \Psi(2{\rm i}t)^2.
	\end{gather*}
	Substituting \eqref{L K rel} into \eqref{Ward var1}, we obtain that Ward's equation is equivalent to
	\begin{gather*}
		\Psi \Psi''(2{\rm i}t) - \Psi'(2{\rm i}t)^2-\dfrac{\Psi(2{\rm i}t)^2}{4t^2}= -\Psi(2{\rm i}t) \frac{L'(t)-M(t)}{2{\rm i}}+\Psi'(2{\rm i}t) L(t).
	\end{gather*}
	Dividing each side of the identity by $\Psi(2{\rm i}t)^2$, we have
	\begin{gather*}
			\frac{{\rm d}}{{\rm d}t}\left[ \frac{1}{2{\rm i}} \frac{\Psi'(2{\rm i}t)}{\Psi(2{\rm i}t)}+\frac{1}{4t} \right]=-\frac{1}{2{\rm i}} \frac{{\rm d}}{{\rm d}t} \left[ \frac{L(t)}{\Psi(2{\rm i}t)} \right] +\frac{1}{2{\rm i}}\frac{M(t)}{\Psi(2{\rm i}t)},
		\end{gather*}
	which completes the proof.
\end{proof}

In the following lemma, we analyse the structure of the function $L$ by rewriting
\begin{gather}\label{L L1L2}
	L=L_1+L_2,
\end{gather}
where
\begin{gather} \label{L1}
	L_1(t):={\rm e}^{-4t^2} \Psi(2{\rm i}t)^2+	\frac{1}{4\sqrt{\pi}} \Psi(2{\rm i}t) \int_{\R} u {\rm e}^{ \frac{u^2}{4} } \erfc\left( \frac{u}{2}-2t \right) \wh{J}(u) \,{\rm d}u
\end{gather}
and
\begin{gather}
		L_2(t) :=-\frac{1}{4\pi\sqrt{\pi}} \int_{ u+v>0 } v {\rm e}^{ \frac{v^2}{4}-2vt } \wh{J}(u) \wh{J}(v) \,{\rm d}u\,{\rm d}v\nonumber\\
\hphantom{L_2(t) :=}{}
-\frac{1}{4\pi\sqrt{\pi}} \int_{ u+v<0 } u {\rm e}^{ \frac{u^2}{4}-2vt } \wh{J}(u) \wh{J}(v) \,{\rm d}u\,{\rm d}v.\label{L2}
	\end{gather}

\begin{Lemma}\label{Lem_Ward var3}
	Ward's equation \eqref{Ward} holds if and only if there exists a constant $c$ such that
	\begin{gather*}
		L_2(t)=-\Psi'(2{\rm i}t)-\left( \frac{{\rm i}}{2t}+c \right)\Psi(2{\rm i}t).
	\end{gather*}
\end{Lemma}

\begin{proof}
	By Lemma~\ref{Lem_Ward var2} and \eqref{L L1L2}, we have shown that Ward's equation \eqref{Ward} is equivalent to
	\begin{gather*}
		\frac{{\rm d}}{{\rm d}t} \left[ \frac{L_1(t)+L_2(t)}{\Psi(2{\rm i}t)} \right] = 	-\frac{{\rm d}}{{\rm d}t}\left[ \frac{\Psi'(2{\rm i}t)}{\Psi(2{\rm i}t)}+\frac{{\rm i}}{2t} \right]-8t {\rm e}^{-4t^2} \Psi(2{\rm i}t).
	\end{gather*}
	Notice that since
	\begin{gather*}
		-2\pi {\rm i} \Psi'(2{\rm i}t)= \int_{\R} u{\rm e}^{-2ut} \wh{J}(u)\,{\rm d}u,
	\end{gather*}
	we have
	\begin{gather*}
	\frac{1}{4\sqrt{\pi}} \int_{\R} \frac{{\rm d}}{{\rm d}t}\left[ u {\rm e}^{ \frac{u^2}{4} } \erfc\left( \frac{u}{2}-2t \right) \right] \wh{J}(u) \,{\rm d}u
	= \frac{{\rm e}^{-4t^2}}{\pi} \int_{\R} u {\rm e}^{2ut} \wh{J}(u) \,{\rm d}u=-2{\rm i}{\rm e}^{-4t^2}\Psi'(2{\rm i}t).
	\end{gather*}
	Therefore by \eqref{L1}, we obtain
	\begin{gather*}
		\frac{{\rm d}}{{\rm d}t} \left[\frac{L_1(t)}{\Psi(2{\rm i}t)}\right]=-8t {\rm e}^{-4t^2} \Psi(2{\rm i}t).
	\end{gather*}
	Thus Ward's equation is equivalent to
	\begin{gather*}
		\frac{{\rm d}}{{\rm d}t} \left[ \frac{L_2(t)}{\Psi(2{\rm i}t)} \right] = 	-\frac{{\rm d}}{{\rm d}t}\left[ \frac{\Psi'(2{\rm i}t)}{\Psi(2{\rm i}t)}+\frac{{\rm i}}{2t} \right].\tag*{\qed}
	\end{gather*}\renewcommand{\qed}{}
\end{proof}

We are now ready to prove Theorem~\ref{Thm_TIsol}. It results from this part of the proof that $E=(-a,a)$ for some $a>0$.

\begin{proof}[Proof of Theorem~\ref{Thm_TIsol}]
	By Proposition~\ref{Prop_mass1 TIsol}, we have
	\begin{gather*}
		\wh{J}(u)=\frac{2\sqrt{\pi}}{{\rm i}} \frac{{\rm e}^{-u^2/4}}{u} \mathbbm{1}_{E}(u/2)
	\end{gather*}
	for some symmetric Borel set $E$. Suppose that $E=(-\alpha/2,\alpha/2)$ for some $\alpha>0.$ We shall compute $L_2$ in \eqref{L2}.
	Then we have
	\begin{gather*}
 -\frac{\sqrt{\pi}}{4\pi^2} \int_{ u+v>0 } v {\rm e}^{ \frac{v^2}{4}-2vt } \wh{J}(u) \wh{J}(v) \,{\rm d}u\,{\rm d}v
		\\
\qquad{} = - \frac{1}{2\pi {\rm i}} \int_\R \int_{ v>-u } {\rm e}^{-2vt } \mathbbm{1}_{E}(v/2) \,{\rm d}v \, \mathbbm{1}_{E}(u/2)\, \wh{J}(u)\,{\rm d}u
		\\
\qquad{} = - \frac{1}{2\pi {\rm i}} \int_{-\alpha}^{\alpha} \frac{{\rm e}^{2ut}-{\rm e}^{-2\alpha t} }{ 2t } \wh{J}(u)\,{\rm d}u=\frac{1}{2t {\rm i}}\frac{1}{2\pi } \int_{-\alpha}^{\alpha} {\rm e}^{-2ut} \wh{J}(u)\,{\rm d}u=-\frac{{\rm i}}{2t} \Psi(2{\rm i}t).
	\end{gather*}
	On the other hand,
	\begin{gather*}
 -\frac{\sqrt{\pi}}{4\pi^2} \int_{ u+v<0 } u {\rm e}^{ \frac{u^2}{4}-2vt } \wh{J}(u) \wh{J}(v) \,{\rm d}u\,{\rm d}v\\
 \qquad{} =-\frac{1}{2\pi {\rm i}}\int_\R \int_{u<-v} \mathbbm{1}_{E}(u/2)\,{\rm d}u \, {\rm e}^{-2vt}\, \mathbbm{1}_{E}(v/2) \wh{J}(v) \,{\rm d}v
		\\
\qquad {}=-\frac{1}{2\pi {\rm i}} \int_{-\alpha}^\alpha (\alpha-v) {\rm e}^{-2vt} \wh{J}(v) \,{\rm d}v=- \Psi'(2{\rm i}t)-\frac{\alpha}{{\rm i}} \Psi(2{\rm i}t).
	\end{gather*}
	Combining the above equations, we obtain that
	\begin{gather*}
		L_2(t)=-\Psi'(2{\rm i}t)-\left( \frac{{\rm i}}{2t}+\frac{\alpha}{{\rm i}} \right)\Psi(2{\rm i}t).
	\end{gather*}
	Then by Lemma~\ref{Lem_Ward var3}, we conclude that Ward's equation holds.
	
It is also easy to see that if $E$ is not connected, Ward's equation does not hold due to the discontinuity of the functions $\int_{ v>-u } {\rm e}^{-2vt } \mathbbm{1}_{E}(v/2) \,{\rm d}v$ and $\int_{u<-v} \mathbbm{1}_{E}(u/2)\,{\rm d}u$. For instance let $E=(-\alpha/2,-\beta/2) \cup (\beta/2,\alpha/2)$ for some $\alpha>\beta>0$.
	Note that
	\begin{gather*}
	\int_{ v>-u } {\rm e}^{-2vt } \mathbbm{1}_{E}(v/2) \,{\rm d}v=
	\begin{cases}
		\displaystyle \frac{{\rm e}^{2ut}-{\rm e}^{-2\alpha t}}{2t} &\text{if }u \in (-\alpha,-\beta),
		\vspace{1mm}\\
		\displaystyle \frac{{\rm e}^{2ut}-{\rm e}^{-2\alpha t}+{\rm e}^{-2\beta t}-{\rm e}^{2\beta t}}{2t} &\text{if }u \in (\beta,\alpha).
	\end{cases}
	\end{gather*}
	Thus we have
	\begin{gather*}
	-\frac{\sqrt{\pi}}{4\pi^2} \int_{ u+v>0 } v {\rm e}^{ \frac{v^2}{4}-2vt } \wh{J}(u) \wh{J}(v) \,{\rm d}u\,{\rm d}v= -\frac{{\rm i}}{2t} \Psi(2{\rm i}t)+ \frac{1}{\pi {\rm i}} \frac{\sinh(2\beta t)}{2t} \int_{\beta}^\alpha \wh{J}(u)\,{\rm d}u.
	\end{gather*}
	Also since
	\begin{gather*}
		\int_{u<-v} \mathbbm{1}_{E}(u/2)\,{\rm d}u=
	\begin{cases}
		\alpha-v &\text{if }v \in (\beta,\alpha),
		\\
		\alpha-v-2\beta 	&\text{if }v \in (-\alpha,-\beta),
	\end{cases}
	\end{gather*}
	we have
	\begin{gather*}
 -\frac{\sqrt{\pi}}{4\pi^2} \int_{ u+v<0 } u {\rm e}^{ \frac{u^2}{4}-2vt } \wh{J}(u) \wh{J}(v) \,{\rm d}u\,{\rm d}v
		\\
\qquad{} =-\frac{1}{2\pi {\rm i}} \int_{-\alpha}^{-\beta} (\alpha-v) {\rm e}^{-2vt} \wh{J}(v) \,{\rm d}v-\frac{1}{2\pi {\rm i}} \int_{\beta}^{\alpha} (\alpha-v-2\beta) {\rm e}^{-2vt} \wh{J}(v) \,{\rm d}v
		\\
\qquad{} = -\Psi'(2{\rm i}t)-\frac{\alpha}{{\rm i}} \Psi(2{\rm i}t)+ \frac{\beta}{\pi {\rm i}} \int_{\beta}^{\alpha} {\rm e}^{-2vt} \wh{J}(v) \,{\rm d}v.
	\end{gather*}
	Combining above equations, we obtain
	\begin{gather*}
		L_2(t) =-\Psi'(2{\rm i}t)-\left( \frac{{\rm i}}{2t}+\frac{\alpha}{{\rm i}} \right)\Psi(2{\rm i}t)
		+ \frac{1}{\pi {\rm i}} \frac{\sinh(2\beta t)}{2t} \int_{\beta}^\alpha \wh{J}(u)\,{\rm d}u+ \frac{\beta}{\pi {\rm i}} \int_{\beta}^{\alpha} {\rm e}^{-2vt} \wh{J}(v) \,{\rm d}v.
	\end{gather*}
	Thus by Lemma~\ref{Lem_Ward var3}, one can see that Ward's equation does not hold.
	
	For a general $E$ it follows from similar computations that
	\begin{gather*}
 -\frac{\sqrt{\pi}}{4\pi^2} \int_{ u+v>0 } v {\rm e}^{ \frac{v^2}{4}-2vt } \wh{J}(u) \wh{J}(v) \,{\rm d}u\,{\rm d}v+\frac{{\rm i}}{2t} \Psi(2{\rm i}t)
	\\
\qquad{} = - \frac{1}{2\pi {\rm i}} \int_\R \left( \int_{ v>-u } {\rm e}^{-2(u+v)t } \mathbbm{1}_{E}(v/2) \,{\rm d}v + \frac{1 }{ 2t } \right) {\rm e}^{2ut} \mathbbm{1}_{E}(u/2) \wh{J}(u)\,{\rm d}u
	\end{gather*}
	and
		\begin{gather*}
 -\frac{\sqrt{\pi}}{4\pi^2} \int_{ u+v<0 } u {\rm e}^{ \frac{u^2}{4}-2vt } \wh{J}(u) \wh{J}(v) \,{\rm d}u\,{\rm d}v
	+ \Psi'(2{\rm i}t)
		\\
	\qquad{} =-\frac{1}{2\pi {\rm i}}\int_\R \left( \int_{u<-v} \mathbbm{1}_{E}(u/2)\,{\rm d}u + v \right) {\rm e}^{-2vt} \mathbbm{1}_{E}(v/2) \wh{J}(v) \,{\rm d}v.
	\end{gather*}
	Therefore by Lemma~\ref{Lem_Ward var3}, Ward's equation holds if and only if
	\begin{gather*}
	 \int_{ v>-u } {\rm e}^{-2(u+v)t } \mathbbm{1}_{E}(v/2) \,{\rm d}v + \frac{1 }{ 2t }=c_1, \qquad u \textrm{ a.e.},
	 \\
\int_{u<-v} \mathbbm{1}_{E}(u/2)\,{\rm d}u + v =c_2, \qquad v\textrm{ a.e.}
	\end{gather*}
	for some constants $c_1$, $c_2$. This gives that $E$ is connected (up to a null set). Now the proof is complete.
\end{proof}

\subsection*{Acknowledgements}
It is our pleasure to thank Boris Khoruzhenko for discussions (G.A.) and both Boris Khoruzhenko and Serhii Lysychkin for sharing with us their results \cite{BL,BL2,Lysychkin} prior to publication.

The authors are grateful to the DFG-NRF International Research Training Group IRTG 2235 supporting the Bielefeld-Seoul graduate exchange programme.
Furthermore, Gernot Akemann was partially supported
by the DFG through the grant CRC 1283 ``Taming uncertainty and profiting from randomness and low regularity in analysis, stochastics and their applications''.
Sung-Soo Byun and Nam-Gyu Kang were partially supported by Samsung Science and Technology Foundation (SSTF-BA1401-51) and by the National Research Foundation of Korea (NRF-2019R1A5A1028324).
Nam-Gyu Kang was partially supported by a KIAS Individual Grant (MG058103) at Korea Institute for Advanced Study.

\pdfbookmark[1]{References}{ref}
\LastPageEnding


\begin{thebibliography}{99}
\footnotesize\itemsep=0pt

\bibitem{MR1762659}
Adler M., Forrester P.J., Nagao T., van Moerbeke P., Classical skew orthogonal
 polynomials and random matrices, \href{https://doi.org/10.1023/A:1018644606835}{\textit{J.~Statist. Phys.}} \textbf{99}
 (2000), 141--170, \href{https://arxiv.org/abs/solv-int/9907001}{arXiv:solv-int/9907001}.

\bibitem{MR2180006}
Akemann G., The complex {L}aguerre symplectic ensemble of non-{H}ermitian
 matrices, \href{https://doi.org/10.1016/j.nuclphysb.2005.09.039}{\textit{Nuclear Phys.~B}} \textbf{730} (2005), 253--299,
 \href{https://arxiv.org/abs/hep-th/0507156}{arXiv:hep-th/0507156}.

\bibitem{MR2302902}
Akemann G., Basile F., Massive partition functions and complex eigenvalue
 correlations in matrix models with symplectic symmetry, \href{https://doi.org/10.1016/j.nuclphysb.2006.12.008}{\textit{Nuclear
 Phys.~B}} \textbf{766} (2007), 150--177, \href{https://arxiv.org/abs/math-ph/0606060}{arXiv:math-ph/0606060}.

\bibitem{ABi06}
Akemann G., Bittner E., Unquenched complex {D}irac spectra at nonzero chemical
 potential: two-color {QCD} lattice data versus matrix model, \href{https://doi.org/10.1103/physrevlett.96.222002}{\textit{Phys.
 Rev. Lett.}} \textbf{96} (2006), 222002, 4~pages, \href{https://arxiv.org/abs/hep-lat/0603004}{arXiv:hep-lat/0603004}.

\bibitem{MR4229527}
Akemann G., Byun S.-S., Kang N.-G., A non-{H}ermitian generalisation of the
 {M}archenko--{P}astur distribution: from the circular law to
 multi-criticality, \href{https://doi.org/10.1007/s00023-020-00973-7}{\textit{Ann. Henri Poincar\'e}} \textbf{22} (2021),
 1035--1068, \href{https://arxiv.org/abs/2004.07626}{arXiv:2004.07626}.

\bibitem{akemann2021skew}
Akemann G., Ebke M., Parra I., Skew-orthogonal polynomials in the complex plane
 and their {B}ergman-like kernels, \href{https://doi.org/10.1007/s00220-021-04230-8}{\textit{Comm. Math. Phys.}} \textbf{389}
 (2022), 621--659, \href{https://arxiv.org/abs/2103.12114}{arXiv:2103.12114}.

\bibitem{akemann2019universal}
Akemann G., Kieburg M., Mielke A., Prosen T., Universal signature from
 integrability to chaos in dissipative open quantum systems, \href{https://doi.org/10.1103/physrevlett.123.254101}{\textit{Phys.
 Rev. Lett.}} \textbf{123} (2019), 254101, 6~pages, \href{https://arxiv.org/abs/1910.03520}{arXiv:1910.03520}.

\bibitem{MR3192169}
Akemann G., Phillips M.J., The interpolating {A}iry kernels for the {$\beta=1$}
 and {$\beta=4$} elliptic {G}inibre ensembles, \href{https://doi.org/10.1007/s10955-014-0962-6}{\textit{J.~Stat. Phys.}}
 \textbf{155} (2014), 421--465, \href{https://arxiv.org/abs/1308.3418}{arXiv:1308.3418}.

\bibitem{Ameur18low}
Ameur Y., Repulsion in low temperature {$\beta$}-ensembles, \href{https://doi.org/10.1007/s00220-017-3027-2}{\textit{Comm. Math.
 Phys.}} \textbf{359} (2018), 1079--1089, \href{https://arxiv.org/abs/1701.04796}{arXiv:1701.04796}.

\bibitem{AB21}
Ameur Y., Byun S.-S., Almost-{H}ermitian random matrices and bandlimited point
 processes, \href{https://arxiv.org/abs/2101.03832}{arXiv:2101.03832}.

\bibitem{MR2817648}
Ameur Y., Hedenmalm H., Makarov N., Fluctuations of eigenvalues of random
 normal matrices, \href{https://doi.org/10.1215/00127094-1384782}{\textit{Duke Math.~J.}} \textbf{159} (2011), 31--81,
 \href{https://arxiv.org/abs/0807.0375}{arXiv:0807.0375}.

\bibitem{MR3342661}
Ameur Y., Hedenmalm H., Makarov N., Random normal matrices and {W}ard
 identities, \href{https://doi.org/10.1214/13-AOP885}{\textit{Ann. Probab.}} \textbf{43} (2015), 1157--1201,
 \href{https://arxiv.org/abs/1109.5941}{arXiv:1109.5941}.

\bibitem{MR3975882}
Ameur Y., Kang N.-G., Makarov N., Rescaling {W}ard identities in the random
 normal matrix model, \href{https://doi.org/10.1007/s00365-018-9423-9}{\textit{Constr. Approx.}} \textbf{50} (2019), 63--127,
 \href{https://arxiv.org/abs/1410.4132}{arXiv:1410.4132}.

\bibitem{MR4030288}
Ameur Y., Kang N.-G., Makarov N., Wennman A., Scaling limits of random normal
 matrix processes at singular boundary points, \href{https://doi.org/10.1016/j.jfa.2019.108340}{\textit{J.~Funct. Anal.}}
 \textbf{278} (2020), 108340, 46~pages, \href{https://arxiv.org/abs/1510.08723}{arXiv:1510.08723}.

\bibitem{ameur2018random}
Ameur Y., Kang N.-G., Seo S.-M., The random normal matrix model: insertion of
 a~point charge, \href{https://doi.org/10.1007/s11118-021-09942-z}{\textit{Potential Anal.}}, {t}o appear, \href{https://arxiv.org/abs/1804.08587}{arXiv:1804.08587}.

\bibitem{MR3668632}
Balogh F., Grava T., Merzi D., Orthogonal polynomials for a class of measures
 with discrete rotational symmetries in the complex plane, \href{https://doi.org/10.1007/s00365-016-9356-0}{\textit{Constr.
 Approx.}} \textbf{46} (2017), 109--169, \href{https://arxiv.org/abs/1509.05331}{arXiv:1509.05331}.

\bibitem{MR2934715}
Benaych-Georges F., Chapon F., Random right eigenvalues of {G}aussian
 quaternionic matrices, \href{https://doi.org/10.1142/S2010326311500092}{\textit{Random Matrices Theory Appl.}} \textbf{1}
 (2012), 1150009, 18~pages, \href{https://arxiv.org/abs/1104.4455}{arXiv:1104.4455}.

\bibitem{MR3849128}
Bertola M., Elias~Rebelo J.G., Grava T., Painlev\'e {IV} critical asymptotics
 for orthogonal polynomials in the complex plane, \href{https://doi.org/10.3842/SIGMA.2018.091}{\textit{SIGMA}} \textbf{14}
 (2018), 091, 34~pages, \href{https://arxiv.org/abs/1802.01153}{arXiv:1802.01153}.

\bibitem{MR1917675}
Bertola M., Eynard B., Harnad J., Duality, biorthogonal polynomials and
 multi-matrix models, \href{https://doi.org/10.1007/s002200200663}{\textit{Comm. Math. Phys.}} \textbf{229} (2002), 73--120,
 \href{https://arxiv.org/abs/nlin.SI/0108049}{arXiv:nlin.SI/0108049}.

\bibitem{MR2530159}
Borodin A., Sinclair C.D., The {G}inibre ensemble of real random matrices and
 its scaling limits, \href{https://doi.org/10.1007/s00220-009-0874-5}{\textit{Comm. Math. Phys.}} \textbf{291} (2009), 177--224,
 \href{https://arxiv.org/abs/0805.2986}{arXiv:0805.2986}.

\bibitem{BE}
Byun S.-S., Ebke M., Universal scaling limits of the symplectic elliptic
 {G}inibre ensemble, \href{https://arxiv.org/abs/2108.05541}{arXiv:2108.05541}.

\bibitem{byun2021wronskian}
Byun S.-S., Ebke M., Seo S.-M., Wronskian structures of planar symplectic
 ensembles, \href{https://arxiv.org/abs/2110.12196}{arXiv:2110.12196}.

\bibitem{byun2021lemniscate}
Byun S.-S., Lee S.-Y., Yang M., Lemniscate ensembles with spectral singularity,
 \href{https://arxiv.org/abs/2107.07221}{arXiv:2107.07221}.

\bibitem{Abanov}
Cardoso G., St\'ephan J.-M., Abanov A.G., The boundary density profile of a
 {C}oulomb droplet. {F}reezing at the edge, \href{https://doi.org/10.1088/1751-8121/abcab9}{\textit{J.~Phys.~A: Math. Theor.}}
 \textbf{54} (2021), 015002, 24~pages, \href{https://arxiv.org/abs/2009.02359}{arXiv:2009.02359}.

\bibitem{MR1643533}
Chau L.-L., Zaboronsky O., On the structure of correlation functions in the
 normal matrix model, \href{https://doi.org/10.1007/s002200050420}{\textit{Comm. Math. Phys.}} \textbf{196} (1998),
 203--247.

\bibitem{DBB}
Dahlhaus J.P., B\'eri B., Beenakker C.W.J., Random-matrix theory of thermal
 conduction in superconducting quantum dots, \href{https://doi.org/10.1103/PhysRevB.82.014536}{\textit{Phys. Rev.~B}} \textbf{82}
 (2010), 014536, 7~pages, \href{https://arxiv.org/abs/1004.2438}{arXiv:1004.2438}.

\bibitem{Dubach2020}
Dubach G., Symmetries of the quaternionic {G}inibre ensemble, \href{https://doi.org/10.1142/S2010326321500131}{\textit{Random
 Matrices Theory Appl.}} \textbf{10} (2021), 2150013, 19~pages,
 \href{https://arxiv.org/abs/1811.03724}{arXiv:1811.03724}.

\bibitem{MR2881072}
Fischmann J., Bruzda W., Khoruzhenko B.A., Sommers H.-J., \.{Z}yczkowski K.,
 Induced {G}inibre ensemble of random matrices and quantum operations,
 \href{https://doi.org/10.1088/1751-8113/45/7/075203}{\textit{J.~Phys.~A: Math. Theor.}} \textbf{45} (2012), 075203, 31~pages,
 \href{https://arxiv.org/abs/1107.5019}{arXiv:1107.5019}.

\bibitem{MR3458536}
Forrester P.J., Analogies between random matrix ensembles and the one-component
 plasma in two-dimensions, \href{https://doi.org/10.1016/j.nuclphysb.2016.01.014}{\textit{Nuclear Phys.~B}} \textbf{904} (2016),
 253--281, \href{https://arxiv.org/abs/1511.02946}{arXiv:1511.02946}.

\bibitem{MR1690355}
Forrester P.J., Honner G., Exact statistical properties of the zeros of complex
 random polynomials, \href{https://doi.org/10.1088/0305-4470/32/16/006}{\textit{J.~Phys.~A: Math. Gen.}} \textbf{32} (1999),
 2961--2981, \href{https://arxiv.org/abs/cond-mat/9812388}{arXiv:cond-mat/9812388}.

\bibitem{ginibre1965statistical}
Ginibre J., Statistical ensembles of complex, quaternion, and real matrices,
 \href{https://doi.org/10.1063/1.1704292}{\textit{J.~Math. Phys.}} \textbf{6} (1965), 440--449.

\bibitem{gorenflo2014mittag}
Gorenflo R., Kilbas A.A., Mainardi F., Rogosin S.V., Mittag-{L}effler
 functions, related topics and applications, \textit{Springer Monographs in
 Mathematics}, \href{https://doi.org/10.1007/978-3-662-43930-2}{Springer}, Heidelberg, 2014.

\bibitem{hedenmalm2017planar}
Hedenmalm H., Wennman A., Planar orthogogonal polynomials and boundary
 universality in the random normal matrix model, \href{https://doi.org/10.4310/ACTA.2021.v227.n2.a3}{\textit{Acta Math.}}
 \textbf{227} (2021), 309--406, \href{https://arxiv.org/abs/1710.06493}{arXiv:1710.06493}.

\bibitem{MR3066113}
Ipsen J.R., Products of independent quaternion {G}inibre matrices and their
 correlation functions, \href{https://doi.org/10.1088/1751-8113/46/26/265201}{\textit{J.~Phys.~A: Math. Theor.}} \textbf{46} (2013),
 265201, 16~pages, \href{https://arxiv.org/abs/1301.3343}{arXiv:1301.3343}.

\bibitem{MR1928853}
Kanzieper E., Eigenvalue correlations in non-{H}ermitean symplectic random
 matrices, \href{https://doi.org/10.1088/0305-4470/35/31/308}{\textit{J.~Phys.~A: Math. Gen.}} \textbf{35} (2002), 6631--6644,
 \href{https://arxiv.org/abs/cond-mat/0109287}{arXiv:cond-mat/0109287}.

\bibitem{kanzieper2005exact}
Kanzieper E., Exact replica treatment of non-{H}ermitean complex random
 matrices, in Frontiers in Field Theory, Editor O.~Kovras, Nova Science
 Publishers, New York, 2005, 23--51, \href{https://arxiv.org/abs/cond-mat/0109287}{arXiv:cond-mat/0109287}.

\bibitem{BL}
Khoruzhenko B.A., Lysychkin S., Truncations of random symplectic unitary
 matrices, \href{https://arxiv.org/abs/2111.02381}{arXiv:2111.02381}.

\bibitem{BL2}
Khoruzhenko B.A., Lysychkin S., Scaling limits of eigenvalue statistics in the
 quaternion-real {G}inibre ensemble, {i}n preparation.

\bibitem{MR2218073}
Kilbas A.A., Srivastava H.M., Trujillo J.J., Theory and applications of
 fractional differential equations, \textit{North-Holland Mathematics
 Studies}, Vol.~204, Elsevier Science~B.V., Amsterdam, 2006.

\bibitem{EK}
Kolesnikov A.V., Efetov K.B., Distribution of complex eigenvalues for
 symplectic ensembles of non-{H}ermitian matrices, \href{https://doi.org/10.1088/0959-7174/9/2/301}{\textit{Waves Random Media}}
 \textbf{9} (1999), 71--82, \href{https://arxiv.org/abs/cond-mat/9809173}{arXiv:cond-mat/9809173}.

\bibitem{MR3450566}
Lee S.-Y., Riser R., Fine asymptotic behavior for eigenvalues of random normal
 matrices: ellipse case, \href{https://doi.org/10.1063/1.4939973}{\textit{J.~Math. Phys.}} \textbf{57} (2016), 023302,
 29~pages, \href{https://arxiv.org/abs/1501.02781}{arXiv:1501.02781}.

\bibitem{MR3670735}
Lee S.-Y., Yang M., Discontinuity in the asymptotic behavior of planar
 orthogonal polynomials under a~perturbation of the {G}aussian weight,
 \href{https://doi.org/10.1007/s00220-017-2888-8}{\textit{Comm. Math. Phys.}} \textbf{355} (2017), 303--338,
 \href{https://arxiv.org/abs/1607.02821}{arXiv:1607.02821}.

\bibitem{Lysychkin}
Lysychkin S., Complex eigenvalues of high dimensional quaternion random
 matrices, Ph.D.~Thesis, {Q}ueen Mary University of London, UK, 2021.

\bibitem{Mehta}
Mehta M.L., Random matrices, 2nd ed., Academic Press, Inc., Boston, MA, 1991.

\bibitem{olver2010nist}
Olver F.W.J., Lozier D.W., Boisvert R.F., Clark C.W. (Editors), N{IST} handbook
 of mathematical functions, Cambridge University Press, Cambridge, 2010.

\bibitem{MR1986426}
Rider B., A limit theorem at the edge of a non-{H}ermitian random matrix
 ensemble, \href{https://doi.org/10.1088/0305-4470/36/12/331}{\textit{J.~Phys.~A: Math. Gen.}} \textbf{36} (2003), 3401--3409.

\bibitem{Serfaty}
Serfaty S., Microscopic description of {L}og and {C}oulomb gases, in Random
 {M}atrices, \textit{IAS/Park City Math. Ser.}, Vol.~26, Amer. Math. Soc.,
 Providence, RI, 2019, 341--387, \href{https://arxiv.org/abs/1709.04089}{arXiv:1709.04089}.

\bibitem{MR1675356}
Widom H., On the relation between orthogonal, symplectic and unitary matrix
 ensembles, \href{https://doi.org/10.1023/A:1004516918143}{\textit{J.~Statist. Phys.}} \textbf{94} (1999), 347--363,
 \href{https://arxiv.org/abs/solv-int/9804005}{arXiv:solv-int/9804005}.

\bibitem{ZS}
\.{Z}yczkowski K., Sommers H.-J., Truncations of random unitary matrices,
 \href{https://doi.org/10.1088/0305-4470/33/10/307}{\textit{J.~Phys.~A: Math. Gen.}} \textbf{33} (2000), 2045--2057,
 \href{https://arxiv.org/abs/chao-dyn/9910032}{arXiv:chao-dyn/9910032}.

\end{thebibliography}
\end{document}